\documentclass[article]{pcms-l}
\usepackage{amsmath,amsfonts,amsthm,amscd,amssymb,stmaryrd}
\usepackage{graphicx}
\newtheorem{theorem}{Theorem}[chapter]

\newtheorem{proposition}[theorem]{Proposition}
\newtheorem{lemma}[theorem]{Lemma}
\newtheorem{definition}[theorem]{Definition}
\newtheorem{conjecture}[theorem]{Conjecture}
\newtheorem{corollary}[theorem]{Corollary}
\newtheorem{exercise}[theorem]{Exercise}
\newtheorem{example}[theorem]{Example}
\newtheorem{remark}[theorem]{Remark}
\numberwithin{equation}{chapter}
\numberwithin{theorem}{chapter}
\numberwithin{figure}{chapter}
\numberwithin{table}{chapter}
\newcommand{\ZZ}{\mathbb{Z}}

\newcommand{\RR}{\mathbb{R}}
\newcommand{\RP}{\mathbb{RP}}
\newcommand{\CP}{\mathbb{CP}}
\newcommand{\CC}{\mathbb{C}}

\newcommand{\p}{\varphi}

\renewcommand{\p}{\partial}

\begin{document}
\frontmatter
\tableofcontents
\mainmatter
\bibliographystyle{amsalpha}
\LectureSeries[Critical Metrics]%
{Critical Metrics for Riemannian Curvature~Functionals
\author{Jeff A. Viaclovsky}}
\address{Department of Mathematics, University of Wisconsin, Madison, WI 53706}
\email{jeffv@math.wisc.edu}
\thanks{The author would like to thank the organizers of this PCMI Program for the invitation 
to give these lectures. Also, thanks are due to Michael T. Lock, who served as teaching assistant, and who helped make many significant 
improvements to these lectures.} 
\section*{Introduction}

The goal of these lectures is to gain an understanding of critical points 
of certain Riemmannian functionals. 
The starting point will be the (normalized) Einstein-Hilbert functional:
\begin{align*}
\tilde{\mathcal{E}}(g) = Vol(g)^{\frac{2-n}{n}} \int_M R_g dV_g,
 \end{align*}
where $R_g$ is the scalar curvature. The Euler-Lagrange equations of 
$\tilde{\mathcal{E}}(g)$ are
\begin{align}
\label{eine}
Ric(g) = \lambda \cdot g,
\end{align}
where $Ric$ denotes the Ricci tensor, and $\lambda$ is a constant. 
A Riemannian manifold $(M,g)$ satisfying \eqref{eine}  
is called an {\em{Einstein manifold}}. 

In Lecture \ref{L1}, we will study the first and second variation of the
 functional $ \tilde{\mathcal{E}}(g)$, and give an analysis of the Jacobi 
operator on transverse-traceless tensors. 
In Lecture \ref{L2}, we will study 
conformal variations, discuss the Lichnerowicz
eigenvalue estimate and Obata's Theorem, and give a survey
of the Yamabe Problem.  In Lecture \ref{L3},
 we will introduce an important splitting of the space of symmetric $2$-tensors
into pure-trace directions, transverse-traceless directions, 
and ``diffeomorphism'' directions. An infinitesimal version 
of the Ebin-Palais slice theorem will show that the diffeomorphism 
directions can be ignored. It will follow that 
critical points of the Einstein-Hilbert functional are in general
saddle points. This leads one naturally to define the smooth 
Yamabe invariant of a manifold, or $\sigma$-constant. 

Next, in Lecture \ref{L4}, we will study the space of Einstein metrics
modulo diffeomorphism, and use Fredholm theory to construct
a map between finite-dimensional spaces called the {\em{Kuranishi map}}
whose zero set is locally in one-to-one correspondence with the 
moduli space of Einstein metrics (locally). We will also discuss some
basic rigidity results for Einstein metrics. 

 The topic in Lecture \ref{L5} will be quadratic curvature 
functionals in dimension four, that is 
linear combinations of the following quadratic 
curvature functionals:
\begin{align*}
\mathcal{W}(g) = \int_M |W_g|^2 dV_g, \ \ \rho(g) = \int_M |Ric_g|^2 dV_g,
\ \ \mathcal{S}(g) = \int_M R_g^2 dV_g.
\end{align*} 
Einstein metrics are critical for these functionals, and we will 
give a discussion of some known results about Einstein metrics. 
There are also many non-Einstein critical metrics for various
linear combinations of these functionals.  
There is a special family of 
critical metrics for $\mathcal{W}$ known as anti-self-dual metrics. 
In Lecture~\ref{L6}, we will study the deformation theory of such metrics, and discuss 
local properties of the moduli space and existence of the Kuranishi map.
We will also discuss several other interesting properties of anti-self-dual metrics.

 In Lecture \ref{L7}, we will discuss some rigidity and stability results regarding
critical metrics for quadratic curvature functionals, which are joint work with Matt Gursky. 
As mentioned above, critical points of the Einstein-Hilbert functional in 
general have a saddle-point structure. However, critical points
for certain quadratic functionals have a nicer local variational structure,
see Theorem~\ref{snit}. Several rigidity results will also be discussed
(which will be crucial in the final lecture).

In Lecture \ref{L8}, we will study a special class of metrics 
called asymptotically locally Euclidean metrics (ALE), and present
several hyperk\"ahler examples. 
We will also discuss a result about non-collapsed limits of Einstein 
metrics:
 with certain geometric assumptions, a subsequence will 
converge to an orbifold Einstein metric. 
A natural question is whether one can reverse this process; that is, 
can one start with an orbifold Einstein metric, use
Ricci-flat ALE metrics to resolve the 
singularities, and find an Einstein metric on the resolution?
We will discuss a recent result of Biquard in the asymptotically
hyperbolic Einstein setting which says that this is possible,
provided that a certain obstruction vanishes. 

We will next discuss a generalization of the Einstein 
condition, called $B^t$-flat metrics, give several examples, 
and discuss an analogous orbifold convergence result which is
joint work with Gang Tian.  
In Lecture \ref{L9}, we will present some of the key points of the proof of this
result, and also give a discussion 
an Einstein metric on $\CP^2 \# 2  \overline{\CP}^2$
found by Chen-LeBrun-Weber.

 Finally, in Lecture \ref{L10} we will discuss an existence theorem 
for critical metrics on certain $4$-manifolds which is 
joint work with Matt Gursky. The general idea is 
to ``glue'' together two metrics which are critical for a functional
to get an ``approximate'' critical metric, and then find conditions
so that one can perturb to an actual solution. Theorem \ref{gvthm}
produces critical metrics for specific functionals 
on the manifolds $\CP^2 \# \overline{\CP}^2,  \CP^2 \# 2  \overline{\CP}^2$,
and $S^2 \times S^2 \# S^2 \times S^2$.

  This is an expanded version of lectures the author gave at the 
PCMI Program in Geometric Analysis  in Park City 
from July 16-19, 2013.

\lecture{The Einstein-Hilbert functional}
\label{L1}


\section{Notation and conventions}
The notation $\nabla_X Y$ will denote the 
covariant derivative on a Riemannian manifold $(M,g)$. 
In a coordinate system $\{x^i\}$, $i = 1 \dots n$,
the Christoffel symbols are defined by
\begin{align}
\nabla_{\partial_i} \partial_j = \Gamma^k_{ij} \partial_k,
\end{align}
where $\partial_i$ denotes the $i$th coordinate tangent vector 
field.
The Christoffel symbols can be expressed in terms of the metric as
\begin{align}
\label{Chrisform}
\Gamma_{ij}^k = \frac{1}{2} g^{kl} \Big(
 \partial_i g_{jl} + \partial_j g_{il} -  \partial_l g_{ij} \Big).
\end{align}
The curvature tensor as a $(1,3)$-tensor is given by
\begin{align}
R(X,Y)Z = \nabla_X \nabla_Y Z - \nabla_Y \nabla_X Z - \nabla_{[X,Y]} Z,
\end{align}
and in coordinates our convention is
\begin{align}
R( \p_i, \p_j) \p_k = R_{ijk}^{\ \ \ l} \p_l.
\end{align}
The curvature tensor as a $(0,4)$-tensor is given by
\begin{align}
Rm(X,Y,Z,W) \equiv -g \big( R(X, Y)Z, W \big),
\end{align}
and in coordinates
\begin{align}
R_{ijkl} = Rm( \p_i, \p_j, \p_k, \p_l).
\end{align}
Note that our convention is
\begin{align}
R_{ijlk} = R_{ijk}^{\ \ \ m} g_{ml}.
\end{align}
That is, we lower the upper index to the {\em{third}} position
(warning: some authors to lower this index to a different position).
The components of the Ricci tensor are given by
\begin{align}
R_{ij} = R_{lij}^{\ \ \ l} = g^{lm} R_{limj} = R_{ji},
\end{align}
and the scalar curvature is
\begin{align}
R = g^{pq} R_{pq} = g^{pq} g^{lm} R_{lpmq}. 
\end{align}

\section{First variation}
We let $\mathcal{M}$ denote the space of Riemannian metrics on a manifold $M$: 
\begin{align}
\mathcal{M} = \{ g \in \Gamma(S^2(T^* M)), \ g \mbox{ is positive definite} \},
\end{align}
where $\Gamma(S^2(T^* M))$ denotes the space of smooth symmetric covariant 
$2$-tensors on $M$. 
The (unnormalized) 
Einstein-Hilbert functional $\mathcal{E} : \mathcal{M} \rightarrow \mathbb{R}$
is defined by 
\begin{align}
\mathcal{E}(g) = \int_M R_g dV_g.
\end{align}
This is a {\em{Riemannian functional}} in the sense that it 
is invariant under diffeomorphisms:
\begin{align}
\mathcal{E}(\varphi^*g) = \mathcal{E}(g).
\end{align}
Next, we compute the Euler-Lagrange equations of the
unnormalized functional:
\begin{proposition}
If $M$ is closed and $n \geq 3$, 
then a metric $g \in \mathcal{M}$ is critical for $\mathcal{E}$ if and 
only if $g$ is Ricci-flat. 
\end{proposition}
\begin{proof} Let $g(t)$ be a variation, with $h = g'(0)$. Then 
\begin{align}
\begin{split}
\label{e'}
\mathcal{E}(g(t))' &= \int_M (R_{g(t)} dV_{g(t)})' \\
& = \int_M R_{g(t)}' dV_{g(t)} + R_{g(t)} (  dV_{g(t)})'
\end{split}
\end{align}
Recall the formula for the linearization of the scalar curvature
\begin{align}
\label{r'}
(R_{g(t)})'\big|_{t=0} &=  - \Delta ( tr h) + \delta^2 h - R_{lp}h^{lp},
\end{align}
where $\delta^2$ is the double-divergence operator defined in 
coordinates by 
\begin{align}
\delta^2 h = \nabla^i \nabla^j h_{ij},
\end{align}
and $\Delta$ is the Laplacian on functions (note that
we use the analysts' Laplacian, which has negative eigenvalues). 
We also recall the formula for the linearization of the volume 
element
\begin{align}
\label{dv'}
(dV_{g(t)})'\big|_{t=0} & =  \frac{1}{2} tr_g (h) dV_g.
\end{align}
Next, we evaluate \eqref{e'} at $t=0$, and consider $\mathcal{E}_g'$ as
a mapping on symmetric tensors. 
Substituting the formulas \eqref{r'} and \eqref{dv'} into 
\eqref{e'}, and integrating by parts, we obtain
\begin{align*}
\mathcal{E}_g'(h) 
&= \int_M \Big( R' + \frac{R}{2}tr_g (h) \Big) dV_g\\
& = \int_M 
\Big( - \Delta( tr h) + \delta^2 h - R^{lp} h_{lp} 
+  \frac{R}{2}tr_g (h)\Big) dV_g\\
& =\int_M \Big( (- R^{lp} + \frac{R}{2}g^{lp}) h_{lp}\Big) dV_g. 
\end{align*}
If this vanishes for all variations $h$, then 
\begin{align*}
Ric = \frac{R}{2}g. 
\end{align*}
If $n > 2$, taking a trace, we find that $R = 0$, so $(M,g)$ is Ricci-flat. 
\end{proof}

\begin{remark}{\em
If $n = 2$ then $\mathcal{E}$ has zero variation, 
thus is constant. This is not surprising in view of the 
Gauss-Bonnet Theorem:
\begin{align}
\int_{M^2} K_g dV_g = 2 \pi \chi(M^2), 
\end{align}
where $K_g = R_g/2$ is the Gaussian curvature, and  $\chi$ denotes the Euler characteristic.
}
\end{remark} 

\begin{exercise}\label{ex1}
{\em
(i) Prove \eqref{r'}. (Hint: first prove that \eqref{Chrisform} linearizes to 
\begin{align}
\label{teq}
(\Gamma_{ij}^k)' = \frac{1}{2} g^{kl} \Big(
 \nabla_i h_{jl} + \nabla_j h_{il} -  \nabla_l h_{ij} \Big).
\end{align}
Next, write out a formula for the scalar curvature in 
terms of Christoffel symbols, and use \eqref{teq}. 
Note: these computations are much simpler if 
one works in a normal coordinate system, since the Christoffel 
symbols vanish at the base point in normal coordinates.)

\noindent
(ii) Prove \eqref{dv'} using that the 
volume element is locally $dV_g = \sqrt{\det(g_{ij})} \cdot dx$.
}
\end{exercise}


\subsection{Diffeomorphism invariance  $\Rightarrow$ Bianchi identity}
Define the divergence operator $\delta : \Gamma(S^2(T^*M)) \rightarrow \Gamma(T^*M)$ 
by 
\begin{align}
(\delta h)_j = g^{pq} \nabla_p h_{qj}. 
\end{align}
The tensor that arises in the above calculation
\begin{align}
G = - Ric + \frac{R}{2} g,
\end{align}
is known as the Einstein tensor. By the contracted second 
Bianchi identity, it is divergence-free. This is actually 
a consequence of diffeomorphism invariance of the functional.
To see this, let $\phi_t$ be a path of diffeomorphisms, and let $g_t = \phi_t^* g$. 
Then $g' = \mathcal{L}_X g$, where $X$ is the tangent vector field 
of this $1$-parameter group of diffeomorphisms at $t=0$, 
and $\mathcal{L}$ is the Lie derivative operator. Integrating by parts:
\begin{align*}
\int_M  \langle G,  \mathcal{L}_X g \rangle dV_g
= - \int_M  \langle 2 \delta G, X \rangle dV_g,
\end{align*}
for any vector field $X$, which implies that $\delta G = 0$. 
Moreover, this argument shows that if $\mathcal{F}$ is {\em{any}} 
Riemannian functional, then $\delta \nabla \mathcal{F} = 0$. 


\section{Normalized functional}

The functional $\mathcal{E}$ is not scale-invariant for $n \geq 3$. 
To account for this, we define the normalized Einstein-Hilbert 
functional by  
\begin{align}
\tilde{\mathcal{E}}(g) = Vol(g)^{\frac{2-n}{n}} \int_M R_g dV_g. 
\end{align}

\begin{proposition} A metric $g$ is critical for $\tilde{\mathcal{E}} $
under all conformal variations (those of the form $h = f \cdot g$ for 
$f: M \rightarrow \RR$) if and only if $g$ has constant 
scalar curvature.  Furthermore, a metric $g \in \mathcal{M}$ is critical 
for $\tilde{\mathcal{E}}$ if and 
only if $g$ is Einstein, that is, $Ric(g) = \lambda \cdot g$ for some 
constant $\lambda \in \mathbb{R}$.  
\end{proposition}
\begin{proof} We compute
\begin{align*}
\tilde{\mathcal{E}}'(h) 
& =  Vol(g)^{\frac{2-n}{n}} \Big(  \frac{2-n}{n} Vol(g)^{-1}\int_M \frac{1}{2}
(tr_g h) dV_g \cdot \int_M R_g dV_g  \Big)\\
& \ \ \ \ \ +  Vol(g)^{\frac{2-n}{n}} \int_M 
\Big( - R^{lp} + \frac{R}{2}g^{lp} \Big) h_{lp} dV_g.
\end{align*}
If $g(t) = f(t)g$, then 
\begin{align*}
\tilde{\mathcal{E}}'(h) =
  \frac{n-2}{2n} Vol(g)^{\frac{2-n}{n}}
\Big( \int_M (tr_g h) (R_g - \overline{R}) dV_g \Big), 
\end{align*}
where $\overline{R}$ denotes the average scalar curvature.
If this is zero for an arbitrary function $tr_g h$, then $R_g$ must be constant. 
The full variation then simplifies to 
\begin{align}
\tilde{\mathcal{E}}'(h) 
=  Vol(g)^{\frac{2-n}{n}} \int_M \Big(-R^{lp} + \frac{R}{n}g^{lp} \Big) h_{lp}  dV_g.
\end{align}
If this vanishes for all variations, then the traceless Ricci 
tensor must vanish, so $(M,g)$ is Einstein. 
\end{proof}


\section{Second  variation}
Since the functional is scale invariant, from now on we will
always restrict to variations satisfying 
\begin{align}
\int_M tr_g(h) dV_g = 0. 
\end{align}
\begin{proposition}
\label{2vprop}
Let $g$ be Einstein with $Ric(g) = \lambda \cdot g$. 
Then the second derivative of $\tilde{\mathcal{E}}$ at $t = 0$ 
is given by 
\begin{align}
\label{epp}
\tilde{\mathcal{E}}''= Vol(g)^{\frac{2-n}{n}} \Big\{
\frac{2-n}{2} \lambda \int_M |h|^2 dV_g + \int_M \langle G' (h), h \rangle dV_g \Big\}.
\end{align}
\end{proposition}
\begin{proof}
The proof is left as an exercise. An important point is that
the second derivative of a functional is well-defined at
a critical point (it only depends on the tangent to the 
variation). 
\end{proof}
\begin{exercise}{\em
Prove the formula for the linearization of the Ricci tensor, 
\begin{align}
\begin{split}
\label{linric}
(Ric')_{ij} &= \frac{1}{2} \Big(  - \Delta h_{ij} + \nabla_i (\delta h)_j 
+ \nabla_j (\delta h)_i- \nabla_i \nabla_j (tr_g h)\\
&  \ \ \ \ \ \ \ - 2 R_{iljp}h^{lp} 
+  R_{i}^{p} h_{jp} + R_{j}^{p} h_{ip} \Big),
\end{split}
\end{align}
where $\Delta: \Gamma(S^2(T^*M)) \rightarrow \Gamma(S^2(T^*M))$ is
the rough Laplacian defined by 
\begin{align}
(\Delta h)_{ij} = g^{pq} \nabla_p \nabla_q h_{ij}. 
\end{align} 
(Hint: write out a formula for the Ricci tensor in 
terms of Christoffel symbols, and use \eqref{teq}. Use normal coordinates to 
simplify the computation.)
}
\end{exercise}
Next, letting
\begin{align}
\label{Rm*}
(Rm * h )_{ij} =  R_{iljp} h^{lp}, 
\end{align}
equation \eqref{epp} can be rewritten as follows.
\begin{proposition}
Let $g$ be Einstein with $Ric(g) = \lambda \cdot g$. 
Then the second derivative of $\tilde{\mathcal{E}}$ at $t = 0$ is given by 
\begin{align}
\tilde{\mathcal{E}}''= Vol(g)^{\frac{2-n}{n}} \int_M \langle h, J h \rangle dV_g,
\end{align}
where $J: \Gamma(S^2(T^*M)) \rightarrow \Gamma(S^2(T^*M))$ is the operator
\begin{align}
Jh = \frac{1}{2} \Delta h - \frac{1}{2} \mathcal{L} (\delta h)
+ (\delta^2 h) g - \frac{1}{2} ( \Delta tr_g h) g 
- \frac{\lambda}{2} (tr_g h) g + Rm * h.
\end{align}
\end{proposition}

\section{Transverse-traceless variations}
We next have the following definition. 
\begin{definition} A symmetric $2$-tensor $h$ is called transverse-traceless
(TT for short) if  $\delta_g h =0$ and $tr_g(h) = 0$.
\end{definition}
The second variation formula simplifies considerably for TT variations:
\begin{proposition}
If $h$ is transverse-traceless, then
\begin{align}
\tilde{\mathcal{E}}''= Vol(g)^{\frac{2-n}{n}} \Big\{
&\int_M \Big \langle  h, \frac{1}{2} \Delta h + Rm * h \Big \rangle dV_g \Big\}.
\end{align}
\end{proposition}

\vspace{2mm}
The first term is manifestly negative, which shows that 
critical metrics for $\tilde{\mathcal{E}}$ always have a saddle 
point structure. In other words, modulo a finite dimensional 
space, $\tilde{\mathcal{E}}$ is locally strictly {\em{maximizing}} in  TT directions.


\subsection{The case of constant curvature}
If $(M,g)$ has constant sectional curvature, then
\begin{align}
R_{ijkl} = k_0 ( g_{ik} g_{jl} - g_{jk} g_{il} ).
\end{align}
The above second variation formula for TT tensors simplifies to 
\begin{align}
\tilde{\mathcal{E}}''= Vol(g)^{\frac{2-n}{n}} \Big\{
&\int_M \Big \langle  h, \frac{1}{2} \Delta h - k_o h \Big \rangle dV_g \Big\}.
\end{align}
This immediately yields:
\begin{corollary} 
\label{scorr}
Let $(M^n,g)$ have constant sectional 
curvature $k_0$. If $k_0 > 0$ and $n > 2$, then the second variation is
strictly negative when restricted to transverse-traceless variations. 
If $k_0 = 0$, then the second variation is strictly 
negative except for parallel $h$. 
\end{corollary}
For $n =2$, as pointed out above, we know that $\tilde{\mathcal{E}}$ is constant; 
thus our computation 
shows that if $h$ is TT then $h$ is identically zero on $S^2$, 
and $h$ must be parallel on $T^2$. 
The parallel sections in the case $k_0 = 0$
correspond to deformations of the flat structure, this will 
be discussed in more detail in Lecture \ref{L4} (also see Exercise~\ref{m2e}
below). 


\subsection{Hyperbolic manifolds}
In the hyperbolic case, we have
\begin{lemma}
\label{hyplem}
If $(M^n,g)$ is compact and hyperbolic, then the smallest eigenvalue 
of the rough Laplacian on TT tensors is at least $n$.
\end{lemma}
\begin{proof} Exercise. Hint: start with the inequality 
\begin{align}
\int_M | \nabla_i h_{jk} - \nabla_j h_{ik}|^2 dV_g \geq 0,
\end{align}
integrate by parts, commute covariant derivatives, etc.
\end{proof}
This immediately yields
\begin{corollary} If $(M^n, g)$ is hyperbolic and $n > 2$, then
$\tilde{\mathcal{E}}$ is locally strictly maximizing in TT directions. 
\end{corollary}

Define the operator $d^{\nabla} : \Gamma(S^2_0(T^*M)) \rightarrow 
\Gamma(\Lambda^2 \otimes T^*M)$ by
\begin{align}
(d^{\nabla} h)_{ijk} =  \nabla_i h_{jk} - \nabla_j h_{ik},
\end{align}
which is called the {\em{Codazzi}} operator. From the proof of Lemma \ref{hyplem}, 
the eigentensors corresponding to the least eigenvalue of the rough Laplacian 
are exactly the Codazzi tensor, that is, $d^{\nabla} h = 0$. 
These yield kernel elements of the second variation when $n =2$, this 
will be examined in more detail below in Lecture \ref{L4}. 

\subsection{The case of $S^2 \times S^2$} 
A nice example with saddle-point structure in 
the TT-directions is given by the product metric on $S^2 \times S^2$. 
Let $\pi_i : S^2 \times S^2 \rightarrow S^2$ denote the projection 
onto the $i$th factor for $i = 1,2$. 
The product metric is $g = g_1 + g_2$ where $g_i = \pi_i^* g_{S^2}$,
and $g_{S^2}$ is the round metric on $S^2$ of constant Gaussian 
curvature equal to $1$. 
\begin{proposition}
On $S^2 \times S^2$ with the product metric $g_{1} + g_{2}$,  the 
lowest eigenvalue of the operator $\frac{1}{2} \Delta h + Rm * h$
on TT tensors is $-1$. The corresponding eigenspace is $1$-dimensional, and is 
spanned by $h = g_1 - g_2$. The next largest eigenvalue is $1$. 
\end{proposition}
\begin{proof}This is left as an exercise, with the following hint:
decompose a traceless symmetric $2$-tensor as
\begin{align}
\label{hdec}
h = \overset{\circ}{h_1}  + \frac{f}{m} g_1
  + \hat{h} + \overset{\circ}{h_2} - \frac{f}{m} g_2,
\end{align}
where $h_i$ is the pull-back of a tensor from the $i$th factor, $\overset{\circ}{h_i}$ is its trace-free part,
for $i = 1,2$, 
and $\hat{h}$ are the mixed directions.
Since the curvature tensor is given by
\begin{align}
R_{ijkl} =  (g_1)_{ik} (g_1)_{jl} -  (g_1)_{jk} (g_1)_{il}
+  (g_2)_{ik} (g_2)_{jl} -  (g_2)_{jk} (g_2)_{il},
\end{align}
the eigenvalue equation reduces to three separate equations on the pieces in 
the decomposition \eqref{hdec}, which can be analyzed separately.
For more details, see for example \cite[Proposition 7.9]{GVRS}.
\end{proof}
\begin{exercise}
\label{s2s2ex}
{\em
(i) Find a constant scalar curvature deformation of the product metric 
corresponding to $h = g_1 - g_2$, and which increases 
the functional $\tilde{\mathcal{E}}$. 

\noindent
(ii) Show that $\alpha_1 \odot \alpha_2$ ($\odot$ = symmetric
product), where $\alpha_i$ are $1$-forms dual to Killing fields
are eigentensors with eigenvalue $1$. 
}
\end{exercise}

\lecture{Conformal geometry}
\label{L2}


\section{Conformal variations}
We will next look at {\em{conformal variations}}, that is, those variations 
of the form $h = f \cdot g$, for a function $f : M \rightarrow \RR$. 
\begin{proposition} Assume that $g$ has constant scalar curvature, 
and let $h = f g$ with $\int_M f dV_g = 0$. Then
\begin{align}
\tilde{\mathcal{E}}''= Vol(g)^{\frac{2-n}{n}} \frac{2-n}{2} \Big\{
&\int_M  \Big((n-1) \Delta f + R_g f \Big) f dV_g \Big\}.
\end{align}
\end{proposition}
\begin{proof} Assuming $g$ is Einstein, this follows from the 
above formulas, but it is easy to verify that this holds 
more generally for constant scalar curvature metrics, the 
calculation is left as an exercise. 
\end{proof}
Let $(S^n, g_S)$ denote the unit $n$-sphere with round metric $g_S$.
Note that the Ricci tensor satisfies $Ric(g_S) = (n-1) g_S$.
An immediate corollary is the following
\begin{corollary}Let $g$ have constant scalar curvature. 
If $R_g \leq 0$ , then $\tilde{\mathcal{E}}$ 
is locally strictly minimizing in the conformal
direction. If $R_g >0$ and $Ric \geq (n-1) g$, then 
the same is true, unless $g$ is isometric to $(S^{n}, g_S)$. 
\end{corollary}

The case $R_g \leq 0$ is obvious and the case $R_g > 0$ 
will follow from a result due to Lichnerowicz which we will 
discuss next. 
Let $\lambda_1$ denote the lowest non-trivial eigenvalue of the 
Laplacian on functions, that is $\Delta u = - \lambda_1 u$.
We have the following eigenvalue estimate which was proved by 
Lichnerowicz in 1958 in \cite{Lich1958}, and for which the equality 
case was characterized by Obata in 1962 \cite{Obata1962}:
\begin{theorem}[Lichnerowicz-Obata] 
\label{lothm}
If a compact manifold $(M^n,g)$ satisfies
\begin{align}
Ric \geq (n-1) \cdot g,
\end{align}
then 
$\lambda_1 \geq n$, with equality if and only if $(M^n,g)$ is 
isometric to $(S^n, g_S)$. 
\end{theorem}
\begin{proof}We only give an outline of the proof. 
First, commuting covariant derivatives, write 
\begin{align}
\int_M (\Delta f)^2 dV_g= \int_M |\nabla^2 f|^2 dV_g + \int_M Ric(\nabla f, \nabla f) dV_g,
\end{align}
and then use the matrix inequality $|A|^2 \geq (1/n) (tr (A))^2$.
If $\lambda_1 = n$, then equality in this inequality implies that there
is a non-trivial solution of the equation 
\begin{align}
\nabla^2 f = \frac{\Delta f}{n} g = - f \cdot g. 
\end{align}
This implies that along any unit-speed geodesic,
\begin{align}
f(s) = A \cos ( s) + B \sin(s),
\end{align}
where $s$ is the arc-length from a fixed point $P_+$. 
If we choose the point $P_+$ to be a maximum of $f$, then 
$f(s) = A \cos(s)$ along any geodesic through $P_+$. 
One then uses this information to construct an isometry with $(S^n, g_S)$. 
For more details, see \cite{Obata1962} and 
also \cite{Kuhnel} for an excellent exposition. 
\end{proof}


\subsection{Conformal variations on $S^n$}
On $(S^n, g_S)$, eigenfunctions corresponding to the eigenvalue $n$
yield directions with $\tilde{\mathcal{E}}'' = 0$. 
There is a nice geometric explanation for this fact:
\begin{proposition}
Let $\phi_t$ be a $1$-parameter group of 
conformal automorphisms of~$g_S$. Then
\begin{align}
\frac{d}{dt} ( \phi_t^* g_S) |_{t=0} = f g, 
\end{align}
where $f$ is an eigenfunction satisfying $\Delta f = -n f$.
\end{proposition}
\begin{proof}We leave this as an exercise, with the following hint: 
use the Hodge decomposition to write any $1$-form $\alpha$
dual to a conformal vector field as $\alpha = df + \omega$, with 
$\omega$ divergence free. Apply the conformal Killing operator 
to $\alpha$ and use the resulting equation to 
show that the trace-free Hessian of $f$ vanishes, and that $\omega$ is Killing. 
\end{proof}


\section{Global conformal minimization}
Actually, it turns out that something much stronger is true for Einstein metrics:
\begin{theorem}
\label{gmthm} An Einstein metric $(M^n,g)$ is
the unique global minimizer of $\tilde{\mathcal{E}}$ in its
conformal class (up to scaling), 
unless $(M,g)$ is isometric to $(S^n, g_S)$. In this case,
any critical point is the pull-back of $g_S$ under a conformal diffeomorphism. 
\end{theorem}
This will be proved below.
The first key point in the proof is the following theorem of Obata:
\begin{theorem}[\cite{Obata1971}]
\label{obthm}
If $(M^n,g)$ is Einstein, then $g$ is the unique constant scalar 
curvature metric in its conformal class (up to scaling), 
unless $(M,g)$ is isometric to $(S^n, g_S)$, in which case all critical 
points are the pull-back of $g_S$ under a conformal diffeomorphism. 
\end{theorem}
\begin{proof} To prove this, 
assume that $\hat{g}$ is a constant scalar curvature 
metric which is conformal to $g$.
Letting $E$ denote the traceless Ricci tensor, we recall the 
transformation formula:  if $g = \phi^{-2} \hat{g}$, then 
\begin{align}
\label{eintrans}
E_g = E_{\hat{g}} + (n-2) \phi^{-1} \big( \nabla^2 \phi - (\Delta \phi/n) \hat{g} \big),
\end{align} 
where $n$ is the dimension, and the covariant derivatives are taken 
with respect to $\hat{g}$.
Since $g$ is Einstein, we have 
\begin{align}
E_{\hat{g}} = (2-n) \phi^{-1} \big( \nabla^2 \phi -\frac{1}{n} (\Delta \phi) \hat{g} \big).
\end{align} 
Integrating, 
\begin{align*}
\int_{M} \phi | E_{\hat{g}}|^2 dV_{\hat{g}} & = (2-n) \int_{M} \phi E_{\hat{g}}^{ij} 
\left\{  \phi^{-1} \big( \nabla^2 \phi - \frac{1}{n}(\Delta \phi) \hat{g} \big)_{ij} \right\} dV_{\hat{g}} \\
& =   (2-n)\int_{M}  E_{\hat{g}}^{ij} \nabla^2 \phi_{ij} dV_{\hat{g}}\\
& = (n-2)  \int_{M}  (\nabla_j E_{\hat{g}}^{ij} \cdot \nabla_i \phi) dV_{\hat{g}}  = 0,
\end{align*}
by the Bianchi identity. Consequently, $\hat{g}$ is also Einstein. 
If $\hat{g}$ is not a constant multiple of $g$, 
then $(M, g)$ admits a nonconstant solution of the equation
\begin{align}
\nabla^2 \phi = \frac{\Delta \phi}{n} g.
\end{align} 
Taking a divergence of this equation, it follows that $\phi + c$, 
where $c$ is a constant,  is an eigenfunction 
of the Laplacian with eigenvalue $n$, so 
$(M,g)$ is isometric to $(S^n, g_S)$ by the same argument in Theorem
\ref{lothm} given above. 
\end{proof}
We will next take a slight detour and discuss
the Yamabe Problem, before returning to the proof of Theorem \ref{gmthm}.

\section{Green's function metric and mass} 
A key idea in the final resolution of the Yamabe Problem 
is the following construction of an asymptotically 
flat metric, called the Green's function metric. 
First, we define an asymptotically flat metric:
\begin{definition}
\label{AFdef}
{\em
 A complete Riemannian manifold $(X^n,g)$ 
is called {\em{asymptotically flat}} or {\em{AF}} of order $\tau$ if 
there exists a diffeomorphism 
$\psi : X \setminus K \rightarrow ( \mathbb{R}^n \setminus B(0,R))$ 
where $K$ is a compact subset of $X$, and such that under this identification, 
\begin{align}
\label{eqgdfale1in}
(\psi_* g)_{ij} &= \delta_{ij} + O( \rho^{-\tau}),\\
\label{eqgdfale2in}
\ \partial^{|k|} (\psi_*g)_{ij} &= O(\rho^{-\tau - k }),
\end{align}
for any partial derivative of order $k$, as
$r \rightarrow \infty$, where $\rho$ is the distance to some fixed basepoint.  
}
\end{definition}
The conformal Laplacian is the operator:
\begin{align}
\square u = -  4 \frac{n-1}{n-2} \Delta u + Ru.
\end{align}
If $(M,g)$ is compact and $R > 0$, 
then for any $p \in M$, there is a unique positive solution to
the equation
\begin{align}
\begin{split}
\square G &= 0 \ \ \mathrm{ on } \ M \setminus \{p\}\\
G &= \rho^{2-n}(1 + o(1))
\end{split}
\end{align}
as $\rho \rightarrow 0$, where $\rho$ is geodesic distance to the basepoint $p$.
This function $G$ is called the {\em{Green's function for the conformal 
Laplacian}}.
\begin{exercise}\label{scex}
{\em 
Show that if $\tilde{g} = u^{\frac{4}{n-2}} g$, then
\begin{align}
\square_g u = R_{\tilde{g}} u^{\frac{n+2}{n-2}}.
\end{align}
}
\end{exercise}

Denote $N = M \setminus \{p\}$ with metric $g_N = G^{\frac{4}{n-2}} g_M$.
From Exercise \ref{scex}, $g_N$ is scalar-flat. 
From a more careful expansion of the Green's function, it is 
possible to show that $g_N$ is also asymptotically flat, but we 
omit the proof. 

The mass of an AF space is defined by
\begin{align}
\label{massdef}
{\mathrm{mass}}(g_N) = \lim_{R \rightarrow \infty} \frac{1}{\omega_{n-1}}\int_{S(R)} \sum_{i,j}
( \partial_i g_{ij} - \partial_{j} g_{ii} ) ( \partial_i \ \lrcorner \ dV_g),
\end{align}
where $\omega_{n-1} = Vol(S^{n-1})$, and $S(R)$ denotes the sphere of radius $R$. 
It was shown in \cite{Bartnik} that if $\tau > (n-2)/2$, then this
mass is well-defined, that is, it is independent of the coordinate
system chosen around infinity. The mass is consequently a geometric 
invariant of an AF metric, and plays an important r\^ole in 
the final resolution of the Yamabe Problem, which we discuss next. 

\section{The Yamabe Problem}
By H\"older's inequality, the functional $\tilde{\mathcal{E}}$ is 
bounded from below when restricted to any fixed conformal class. 
It is then natural to minimize in the conformal direction:
\begin{align}
Y(M,[g]) = \inf_{\tilde{g} \in [g]} \tilde{\mathcal{E}}(\tilde{g}).
\end{align}
This is called the {\em{conformal Yamabe invariant}}.
\begin{theorem}If $(M^n,g)$ is compact, then 
there exists a conformal metric $\tilde{g} \in [g]$ which has
constant scalar curvature, and which minimizes $\tilde{\mathcal{E}}$
in its conformal class. 
\end{theorem}
\begin{proof}[Outline of Proof.] 
For any conformal class, Aubin showed that 
\begin{align}
\label{Aubinineq}
Y(M,[g]) \leq \tilde{\mathcal{E}}(g_S).
\end{align}
The idea of the proof of this estimate is to choose a conformal 
factor which is spherical in an $\epsilon$-neighborhood of a point, 
and zero everywhere else (this is called a ``bubble''). Expanding the Yamabe 
energy of this test function in the parameter $\epsilon$ then yields
a leading term which is exactly the Yamabe energy of the 
spherical metric. 

 Next, one shows that if this inequality is strict, then a solution exists.
This step is now considered ``trivial'' by experts, but 
in fact this took a long time to figure out. Yamabe's 
original paper \cite{Yamabe} contains a serious mistake on this point, 
this was fixed by Trudinger \cite{TrudingerYam}, and then 
optimized by Aubin \cite{Aubin2}. 

The more difficult step is to show that if $(M,[g])$ is not
conformally diffeomorphic to $(S^n, [g_S])$ then 
the inequality \eqref{Aubinineq} is strict. 
In case $n \geq 6$ and $g$ is not locally conformally flat, 
this was proved by Aubin \cite{Aubin2, Aubin3} by basing the above
test function at a point where the Weyl tensor does not vanish. 
The locally conformally flat case was proved by Schoen-Yau \cite{SchoenYau}
using ideas involving the developing map. 
The case $n \leq 6$ was proved by Schoen \cite{Schoen}. 
The main idea is the following. 
Instead of making the above test function be zero away from the bubble, Schoen's 
idea was to instead choose the conformal factor to be the Green's 
function for the conformal Laplacian away from the bubble. 
The mass of the associated asympotically flat metric arises 
as the next term in the expansion of the Yamabe energy of this
test function, so the result follows from the positive mass
theorem of Schoen-Yau \cite{SYI, SYII, Schoen1}. 
\end{proof}

We next return to the global minimization statement in 
Theorem \ref{gmthm}:
\begin{proof}[Proof of Theorem \ref{gmthm}]
The uniqueness follows from Theorem \ref{obthm}. For the 
minimization statement,
of course, we know a minimizer exists from the resolution of
the Yamabe problem, but there is an ``easy'' proof in the 
Einstein case. In the negative or
zero scalar curvature case, one can apply a standard argument 
from the calculus of variations to show that a minimizing 
sequence converges. In the positive case, scale so that 
$Ric = (n-1)g$.  Then 
\begin{align}
\tilde{\mathcal{E}}(g) = n (n-1) Vol(g)^{2/n}.
\end{align}
By Bishops' volume comparison theorem, $Vol(M,g) \leq Vol(S^n, g_S)$ with equality 
if and only if $g$ is isometric to $g_S$. So if $g$ is not isometric 
to $g_S$, we have
\begin{align}
Y(M,[g]) = \inf_{\tilde{g} \in [g]} \tilde{\mathcal{E}}(g)
< \tilde{\mathcal{E}}(g_S).
\end{align}
As discussed above, this estimate implies that a minimizing sequence 
converges (no bubbles are possible).  

Finally, the case of $(S^n, g_S)$ takes some extra work. One needs
to suitably re-normalize a minimizing sequence using the 
conformal group to obtain a minimizing sequence which 
converges, see \cite[Proposition 4.6]{LeeParker} for an argument
due to Karen Uhlenbeck.

\end{proof}

\begin{remark}{\em
An important question is if the set of unit volume constant 
scalar curvature metrics in a conformal class is compact if the manifold is not conformally diffeomorphic
to the sphere. This is true in dimensions $n \leq 24$ \cite{KMS}.
Surprisingly, it is false in higher dimensions \cite{BrendleBU,
BrendleBU2}. 
}
\end{remark}
\section{Generalizations of the Yamabe Problem}

We mention that there are many other Yamabe-type
conformal deformation problems which also have
variational characterizations. We describe one such 
example next. 
Define the {\em{Schouten tensor}} by 
\begin{align}
A_g = \frac{1}{n-2} \left(  Ric_g - \frac{R_g}{2(n-1)} g \right).
\end{align}
Consider the functional 
\begin{align}
\label{fsigma2}
\tilde{\mathcal{F}}_{\sigma_2}(g) = Vol(g)^{\frac{4}{n} -1} \int_M \sigma_2 (g^{-1} A_g) dV_g,
\end{align}
where $\sigma_2$ denotes the second elementary symmetric 
function of the eigenvalues. 
This functional has a nice conformal variational property, analogous
to that for the Einstein-Hilbert functional. 
\begin{theorem}[\cite{ViaclovskyDuke}]
If $n \neq4$, a metric $g$ is a critical for $\tilde{\mathcal{F}}_{\sigma_2}$ under all
conformal variations if and only if 
\begin{align}
\label{s2eqn}
\sigma_2 ( g^{-1} A_g ) = C,
\end{align}
for some constant $C$.
\end{theorem}
One may also generalize the Yamabe problem by asking if it 
is possible to conformally deform a metric so that \eqref{s2eqn}
is satisfied. Note that, in contrast to the Yamabe equation
which is semi-linear, equation \eqref{s2eqn} is a fully nonlinear equation,
and some assumption must be made on the conformal structure to ensure 
that the equation is elliptic. 
There has been much progress on this $\sigma_2$ problem, 
see for example \cite{CGY2, CGY, GeWang, GV2003, ShengTrudingerWang}. 
More generally, one can consider other symmetric functions
of the eigenvalues, and for the $k$th elementary function, this 
is known as the $\sigma_k$-Yamabe Problem.
The locally conformally flat case has been solved for all $k$, see  
\cite{GuanWang, LiLi}. 
This has also been solved for the case $k > n/2$, see \cite{GV2007}.
There have been many other related works involving various
symmetric functions of the eigenvalues, we refer the
reader to~\cite{ViaclovskyChern} for a more detailed description
and other references.

For solving conformal deformation problems, 
we note that parabolic methods also play an important r\^ole, 
see for example \cite{BrendleI, BrendleII, GuanWang, ShengTrudingerWang}.
Another generalization of the Yamabe Problem is to the class of higher 
order equations, and deals with prescribing $Q$-curvature, which is a
higher order generalization of the scalar curvature. We will
not discuss this further, and refer the reader to~\cite{Q3,Q2} for details 
about the notion of $Q$-curvature.

The Yamabe Problem can also be generalized to the setting of orbifolds
\cite{AB1, AB2, Akutagawa2012}. This turns out to be more
subtle than the Yamabe Problem on manifolds -- there in fact exist
conformal classes on compact orbifolds  which do not contain any constant scalar
curvature metrics. For example, the conformal compactifications of hyperk\"ahler
ALE metrics and also the conformal classes of certain Bochner-K\"ahler
metrics on weighted projective spaces do not admit any solution 
of the orbifold Yamabe Problem \cite{ViaclovskyMM, ViaclovskyTohoku}.

\lecture{Diffeomorphisms and gauging }
\label{L3}

\section{Splitting}
We begin by discussing a decomposition of the space of symmetric 
$2$-tensors; some references for this material are \cite{BergerEbin, Besse}. 
We let $\mathcal{K} : T^*M \rightarrow S^2_0(T^*M)$ be the 
conformal Killing operator
\begin{align}
(\mathcal{K} \alpha)_{ij} = \nabla_i \alpha_j + \nabla_j \alpha_i 
- \frac{2}{n} (\delta \alpha) g_{ij}.
\end{align}
Also, consider the operator $\square: \Gamma(T^*M) 
\rightarrow \Gamma(T^*M) $, defined by  $\square = \delta \mathcal{K}$
where $\delta : \Gamma( S^2(T^*M)) \rightarrow \Gamma(T^*M)$ is 
the divergence defined by 
\begin{align}
(\delta h)_j = g^{pq} \nabla_p h_{iq}.
\end{align}
\begin{exercise}{\em
(i) Show that the operator $\square$ is elliptic and self-adjoint. 

\noindent
(ii) Prove that the kernel of $\square$ is exactly the space of conformal 
Killing forms, i.e., they satisfy $\mathcal{K} \alpha = 0$. 
}
\end{exercise}
\begin{lemma} The space of symmetric $2$-tensors admits
the following orthogonal decomposition:
\begin{align}
\label{odec}
S^2 (T^*M) = \{ f \cdot g\} \oplus \{\mathcal{K}(\alpha) \} \oplus
\{ \delta h = 0 , tr_g(h) = 0\}.
\end{align}
\end{lemma}
\begin{proof}
Given $h \in S^2_0(T^*M)$, consider the $1$-form $\delta h$. By Fredholm
theory, the equation $\square \alpha = \delta h$ has a solution 
if and only if $\delta h$ is orthogonal to the kernel of the adjoint 
operator, which is exactly the space of conformal Killing $1$-forms (by the 
exercise). If $\kappa$ is 
any conformal Killing $1$-form, then 
\begin{align*} 
\int_M \langle \delta h, \kappa \rangle
=\frac{1}{2}\int_M \langle h, \mathcal{K} \kappa \rangle = 0.
\end{align*}
So the equation $\square \alpha = \delta h$ has a solution, which 
proves that $h - \mathcal{K}\alpha$ is divergence-free. 
\end{proof}


\subsection{Another decomposition}
The orthogonal decomposition given in \eqref{odec}
implies the decomposition 
\begin{align}
S^2 (T^*M) = \{ f \cdot g\} + \{\mathcal{L}(\alpha) \} \oplus
\{ \delta h = 0 , tr_g(h) = 0\}.
\end{align}
\begin{proposition}
If $(M,g)$ is Einstein, with $Ric = \lambda \cdot g$, then 
this latter decomposition is a direct sum, unless $(M,g)$ is
isometric to $(S^n, g_S)$. 
\end{proposition}
\begin{proof}
We need to show that the spaces $\{f \cdot g \}$ 
and $\{ \mathcal{L}(\alpha)\}$ have intersection $\{0\}$. 
So if  $\mathcal{L}(\alpha) = f \cdot g$, then taking a trace, we have 
\begin{align*}
2 \delta \alpha = n f, 
\end{align*}
which implies that $\mathcal{K} (\alpha) = 0$. Taking a 
divergence of this equation, we have
\begin{align*}
\nabla_i ( \nabla_i \alpha_j +  \nabla_j \alpha_i - (2/n) (\delta \alpha) g_{ij})
&= \Delta \alpha_j + \nabla_i \nabla_j \alpha_i  - (2/n) \nabla_j (\delta \alpha)\\
& =  \Delta \alpha_j  + \left(1 - \frac{2}{n}\right)  \nabla_j (\delta \alpha) + \lambda \alpha_j.
\end{align*}
Next, recall the Bochner formula for $1$-forms
\begin{align}
\label{Bochner}
(\Delta \alpha)_i = - (\Delta_H \alpha)_i + R_{ip} g^{pj} \alpha_j,
\end{align}
where $\Delta_H$ is the Hodge Laplacian. This yields that 
\begin{align*}
\Delta \alpha = - (d \delta_H + \delta_H d) \alpha + \lambda \alpha, 
\end{align*}
where $\delta_H$ is the 
Hodge divergence (which is the negative of our divergence). 
Putting these together, we obtain 
\begin{align}
\label{topeqn}
\square \alpha = -2 \left( \frac{n-1}{n} \right) d \delta_H \alpha - \delta_H d \alpha + 2 \lambda \alpha = 0.
\end{align}
Next, pairing \eqref{topeqn} with $\alpha$ and integrating, 
\begin{align*}
- 2 \left( \frac{n-1}{n} \right)  \int_M |\delta \alpha|^2dV_g -  \int_M |d \alpha|^2dV_g 
+ 2 \lambda \int_M |\alpha|^2 dV_g= 0.
\end{align*}
This implies that $\alpha = 0$ if $\lambda < 0$ (so any conformal 
Killing field vanishes for a negative Einstein metric). 
If $\lambda = 0$, we see that $\delta \alpha = 0 $ and $d \alpha = 0$.
In particular, $\alpha$ is a Killing $1$-form, and we are done. 

In the case $\lambda > 0$, applying a divergence to \eqref{topeqn} yields
\begin{align}
2 \left( \frac{n-1}{n} \right) \Delta (\delta \alpha) + 2 \lambda (\delta \alpha) = 0. 
\end{align}
By Lichnerowicz' Theorem, this implies that $\delta \alpha = 0$ 
unless $(M,g)$ is isometric to $(S^n, g_S)$, so $\alpha$ is Killing. 
\end{proof}


\section{Second variation as a bilinear form}
From Proposition \ref{2vprop}, let us recall the second variation is
\begin{align}
\tilde{\mathcal{E}}''(h,h) = Vol(g)^{\frac{2-n}{n}} \int_M  \langle h, J h \rangle dV_g,
\end{align}
where $J$ is the operator 
\begin{align}
J h = \frac{2-n}{2}\lambda h  + G' h.
\end{align}
Using polarization, the Hessian of $\tilde{\mathcal{E}}$ is the 
bilinear form given by 
\begin{align}
\tilde{\mathcal{E}}''(h_1,h_2) 
= Vol(g)^{\frac{2-n}{n}} \int_M  \langle h_1, J h_2 \rangle dV_g.
\end{align}


\begin{proposition} The decomposition
\begin{align}
S^2 (T^*M) = \{ f \cdot g\} \oplus \{\mathcal{L}(\alpha) \} \oplus
\{ \delta h = 0 , tr_g(h) = 0\}
\end{align}
is orthogonal with respect to $\tilde{\mathcal{E}}''(\cdot,\cdot)$.
\end{proposition}
\begin{proof}
First, $\tilde{\mathcal{E}}''( \mathcal{L}(\alpha), \cdot) = 0$
from diffeomorphism invariance. So we just need to check that
\begin{align}
\tilde{\mathcal{E}}''(f \cdot g, z) = 0
\end{align}
if $z$ is TT. To see this,
\begin{align*}
\tilde{\mathcal{E}}''(f \cdot g, z) &= 
Vol(g)^{\frac{2-n}{n}} \int_M  \langle f \cdot g, J z \rangle dV_g\\
& = Vol(g)^{\frac{2-n}{n}} \int_M  \langle f \cdot g, \frac{1}{2} \Delta z
+ Rm * z \rangle dV_g\\
& =  Vol(g)^{\frac{2-n}{n}} \int_M  f (R_{ijip} z_{jp})  dV_g = 0.
\end{align*}
\end{proof}

To summarize: if $h$ is any symmetric $2$-tensor, then decompose $h$ as
\begin{align}
h = f \cdot g  + \mathcal{L} \alpha + z,
\end{align}
where $z$ is TT. Then 
\begin{align}
\tilde{\mathcal{E}}''(h, h)
=  \tilde{\mathcal{E}}''(f \cdot g, f \cdot g) +  \tilde{\mathcal{E}}''(z, z).
\end{align}
So we have shown that to check the second variation, we really only 
need to consider conformal variations and TT variations separately.


\section{Ebin-Palais slice theorem (infinitesimal version)}
The above discussion was at the level of the ``tangent space to the space
of Riemannian metrics at $g$''. We will next transfer this statement 
directly to the space of Riemannian metrics near $g$ {{modulo diffeomorphism}}.
\begin{theorem} 
\label{difft}
The local behavior of $\tilde{\mathcal{E}}$, when 
considered as a map on $\mathcal{M}/\mathcal{D}$ (the space 
of Riemannian metrics modulo diffeomorphism),
is determined by the conformal and TT directions (to second order).  
\end{theorem}
The main tool for this is the following infinitesimal version 
of a ``slice'' theorem due to Ebin-Palais. The notation $C^{k,\alpha}$ 
will denote the space of H\"older continuous mappings (or tensors) 
with $0 < \alpha < 1$. 
\begin{theorem} 
\label{EP}
For each metric $g_1$ in a sufficiently small $C^{\ell+1,\alpha}$-neighborhood of~$g$ ($\ell \geq 1$),
there is a $C^{\ell+2,\alpha}$-diffeomorphism $\varphi : M \rightarrow M$ such that
\begin{align} 
\tilde{\theta} \equiv \varphi^{*}g_1 - g
\end{align}
satisfies 
\begin{align}
\delta_g \Big( \tilde{\theta} - \frac{1}{n} tr_g ( \tilde{\theta}) g  \Big) = 0.
\end{align}
\end{theorem}
\begin{proof}
 Let $\{ \omega_1, \dots , \omega_{\kappa} \}$ denote a basis of the space of conformal Killing forms with respect to $g$.
Consider the map
\begin{align}
\mathcal{N} : C^{\ell +2,\alpha}(TM) \times \mathbb{R}^{\kappa}  \times C^{\ell +1,\alpha}(S^2(T^{*}M)) \rightarrow C^{\ell,\alpha}(T^{*}M)
\end{align}
 given by
\begin{align} 
\mathcal{N}(X,v,\theta) = \mathcal{N}_{\theta}(X,v)
= \big( {\delta_g} \big[ \overbrace{\varphi_{X,1}^{*}(g + \theta)}^{\circ} \big]  + \sum_i v_i \omega_i \big),
\end{align}
where $\varphi_{X,1}$ denotes the diffeomorphism obtained by following the flow
generated by the vector field $X$ for unit time, and $\circ$ denotes
the traceless part with respect to $g$. 
Linearizing in $(X,v)$ at $(X,v,\theta) = (0,0,0)$, we find
\begin{align*}
\mathcal{N}'_0 (Y,a) &= \frac{d}{d\epsilon} \big( \delta_g 
\big[ \overbrace{ \varphi_{\epsilon Y,1}^{*}(g)}^{\circ} \big] + \sum_i (\epsilon a_i) \omega_i \big) \Big|_{\epsilon = 0} \\
&= \big( \delta_g [ \overbrace{\mathcal{L}_g Y^{\flat}}^{\circ}] + \sum_i a_i \omega_i \big) \\
&= \big( \Box Y^{\flat} + \sum_i a_i \omega_i \big),
\end{align*}
where $Y^{\flat}$ is the dual one-form to $Y$.
The adjoint map $(\mathcal{N}'_0)^{*} : C^{m + 2,\alpha}(T^{*}M) \rightarrow 
C^{m,\alpha}(TM) \times \mathbb{R}^{\kappa}$ is given by
\begin{align}
\label{nadj}
(\mathcal{N}'_0)^{*}(\eta) = \Big( (\Box \eta)^{\sharp}, 
\int_M \langle \eta, \omega_i \rangle\ dV_g \Big),
\end{align}
where $(\Box \eta)^{\sharp}$ is the vector field dual to $\Box \eta$. 

If $\eta$ is in the kernel of the adjoint, 
the first equation implies that $\eta$ is a conformal Killing form, 
while the second implies that $\eta$ is orthogonal (in $L^2$) to the space
of conformal Killing forms.  
It follows that $\eta = 0$, so the map $\mathcal{N}'_0$ is surjective.

Omitting a few technical details for simplicity, applying an infinite-dimensional 
version of the implicit function theorem (which will be discussed in detail 
below in Lecture \ref{L4}), 
given $\theta_1 \in C^{\ell+1,\alpha}(S^2(T^{*}M))$ small enough we can solve the equation $\mathcal{N}_{\theta_1} = 0$; i.e.,
there is a vector field $X \in C^{\ell+2,\alpha}(TM)$, and a $v \in \mathbb{R}^{\kappa}$, 
such that
\begin{align} 
\delta_g  [ \overbrace{\varphi^{*}g_1}^{\circ}] + \sum_i v_i \omega_i = 0,
\end{align}
where $\varphi = \varphi_{X,1}$. Letting $ \tilde{\theta} = \varphi^{*}g_1 - g$,
then $\tilde{\theta}$ satisfies
\begin{align} 
\delta_g  [ \overset{\circ}{\tilde{\theta}}] + \sum_i v_i \omega_i = 0,
\end{align}
Pairing with $\omega_j$, for $j = 1 \dots \kappa$, and integrating by parts, 
we see that $v_j = 0$, and we are done.
\end{proof}

\begin{exercise}{\em
Verify the above formula \eqref{nadj} for $(\mathcal{N}'_0)^{*}$.
}
\end{exercise}
\begin{exercise}{\em
By adding a scaling factor to the map $\mathcal{N}$, modify 
the above argument to show that we can find a constant $c$ 
(depending upon $g_1$), and find 
\begin{align} 
\label{thtildef1}
\tilde{\theta} \equiv e^c \varphi^{*}g_1 - g,
\end{align}
so that in addition to the traceless part of $\tilde{\theta}$ being 
TT, $\tilde{\theta}$ also satisfies 
\begin{align*}
\int_M tr_g \tilde{\theta}\ dV_g = 0.
\end{align*}
That is, we can also ``gauge away'' the scale-invariance of the 
functional. Equivalently, we can look at a slice of unit-volume
metrics modulo diffeomorphism.}
\end{exercise}

\begin{remark}{\em
The reason this is called an ``infinitesimal'' version of the Slice Theorem 
is because the full Ebin-Palais Slice Theorem constructs a local 
slice for the action of the diffeomorphism group, see \cite{Ebin}. 
The main difficulty is that the natural action of the 
diffeomorphism group on the space of Riemannian metrics 
is not differentiable as a mapping of Banach spaces 
(with say Sobolev or H\"older norms). It is however differentiable 
as a mapping of ILH spaces, see \cite{Omori, Koiso1978}.
For the purposes of these lectures, we will content ourselves with
the infinitesimal version, and will not go into details about the full slice theorem 
}
\end{remark}


\begin{proof}[Proof of Theorem \ref{difft}]
Combining the above discussions, given any $g_1$ sufficiently near $g$, 
we can write 
\begin{align}
\varphi^* g_1 = g + \tilde{\theta},
\end{align}
with $\tilde{\theta} = f \cdot g + z$ with $\int_M f dV_g = 0$, 
and $z$ is TT. Then 
\begin{align*}
\tilde{\mathcal{E}}(g_1) 
&= \tilde{\mathcal{E}}(\varphi^* g_1)  \text{ (from diffeomorphism invariance)}\\
&= \tilde{\mathcal{E}} (g +\tilde{\theta} ) \\
&= \tilde{\mathcal{E}} (g) + \tilde{\mathcal{E}}'_g (\tilde{\theta})
+ \tilde{\mathcal{E}}''_g ( f \cdot g + z,  f \cdot g + z)
+ \text{remainder}\\
& = \tilde{\mathcal{E}} (g) + \tilde{\mathcal{E}}''_g ( f \cdot g , f \cdot g)
+ \tilde{\mathcal{E}}''_g (z,  z) + \text{remainder}.
\end{align*}
\end{proof}


\section{Saddle point structure and the smooth Yamabe invariant.}

We have seen that the functional $\tilde{\mathcal{E}}$ is minimizing in the conformal 
directions, but maximizing (modulo a finite-dimensional subspace) in 
the TT directions. So an Einstein metric is always a saddle point for 
$\mathcal{E}$. This suggests defining the following min-max type invariant.

First, we minimize in the conformal direction:
\begin{align*}
Y(M,[g]) = \inf_{\tilde{g} \in [g]} \tilde{\mathcal{E}}(g).
\end{align*}
This is called the {\em{conformal Yamabe invariant}}.

The min-max invariant is then defined by 
\begin{align*}
Y(M) = \sup_{g \in \mathcal{M}} Y(M,[g]),
\end{align*} 
which we will call the {\em{smooth Yamabe invariant}} of $M$,
also known as the {\em{$\sigma$-invariant}} of $M$. This was defined 
independently by Osamu Kobayashi \cite{Kobayashi1987}  
and Richard Schoen \cite{Schoen1}.


\subsection{Some known cases}
We will not focus on smooth Yamabe invariants in this lecture, but
only list a few known cases:

\vspace{2mm}
\begin{itemize}
\item  $Y(S^n) = Y(S^n, [g_S]) = n (n-1) Vol(S^n)^{\frac{2}{n}} $. 

\vspace{2mm}
\item $Y(S^1 \times S^{n-1}) =  Y(S^n, [g_S])$, proved by Schoen \cite{Schoen1}.

\vspace{2mm}
\item $Y(\RP^3) = Y(\RP^3, [g_S]) $, proved by Bray-Neves \cite{BrayNeves}.

\vspace{2mm}
\item If $(M^3,g_H)$ is compact hyperbolic, then $Y(M^3) = Y(M^3, [g_H])$. This
follows from Perelman's work, see \cite{AIL}.

\vspace{2mm}
\item $Y(\CP^2) = Y(\CP^2, [g_{\rm{FS}}]) = 12 \pi \sqrt{2}$, 
where $g_{\rm{FS}}$ is the Fubini-Study metric, proved by LeBrun \cite{LeBrun1997},
see also \cite{GL1998}.

\end{itemize}
There are many other cases for which the Yamabe invariant is
known, but we do not list them here. We note that an effective tool in dimension 
four is Seiberg-Witten Theory, see \cite{LeBrunEMYP, Sung}.
Also, there are also many known 
estimates on Yamabe invariants (see for example \cite{ADH, Petean}), 
but there is not a single known example 
of a compact manifold $M$ with positive Yamabe invariant which
has been shown to satisfy $0 < Y(M) < Y(S^n)$ in 
dimensions $n \geq 5$.  


\subsection{Some unknown cases}

It is a very difficult problem to determine Yamabe invariants in general. 
Here are a few prominent unknown cases:

\vspace{2mm}
\begin{itemize}
\item  What is $Y (S^n/ \Gamma)$, where $S^n/ \Gamma$ is a spherical space form
with $|\Gamma| > 1$? 
Is it achieved by the round metric? The only known case
is the Bray-Neves result listed above. 

\vspace{2mm}
\item What is $Y ( \CP^2 \# \CP^2)$?  The only result known is due to 
O. Kobayashi \cite{Kobayashi1987}: 
\begin{align*}Y ( \CP^2 \# \CP^2) \geq Y(\CP^2).
\end{align*} 

\vspace{2mm}
\item What is $Y (  \CP^2 \# \overline{\CP}^2)$? Again, 
the only result known is 
\begin{align*}
Y ( \CP^2 \# \overline{\CP}^2) \geq Y(\CP^2).
\end{align*}

\vspace{2mm}
\item What is $Y(S^2 \times S^2)$? The only known result is 
that 
\begin{align*}
Y(S^2 \times S^2) > Y(S^2 \times S^2, g_{S^2} + g_{S^2})
\end{align*}
(strict inequality). 
This follows from Exercise \ref{s2s2ex} and 
a result of B\"ohm-Wang-Ziller that CSC metrics sufficiently near 
an Einstein metric are
also Yamabe minimizers in their conformal class \cite[Theorem C]{BWZ}. 
\end{itemize}
We end by noting there are relatively few theorems giving
conditions for the uniqueness of a Yamabe metric. 
There is Obata's Theorem \ref{obthm},
and the result of \cite{BWZ} mentioned above; also see \cite{DLPZ} and \cite{Kato}.
\lecture{The moduli space of Einstein metrics}
\label{L4}

\section{Moduli space of Einstein metrics}
Next, given an Einstein metric $g$ with $Ric(g) = \lambda \cdot g$,
we would like understand the space of solutions of the equation
\begin{align}
Ric(\tilde{g}) = \lambda \cdot \tilde{g}
\end{align}
with $\tilde{g}$ near $g$. This will be infinite-dimensional since 
$\varphi^*g$ 
will also be a solution for any diffeomorphism $\varphi: M \rightarrow M$. 
Therefore, we need to look at the space of Einstein metrics modulo diffeomorphism. 
Our goal is to prove:
\begin{theorem}
\label{existthm}
Assume $g$ is Einstein with $Ric(g) = \lambda \cdot g$ and $\lambda < 0$. 
Then the space of Einstein metrics near $g$ modulo diffeomorphism
is locally isomorphic to the zero set of a map 
\begin{align}
\Psi : H^1_E \rightarrow H^1_E,
\end{align}
where 
\begin{align}
\label{H1Edef}
H^1_E = \{ h \in S^2(T^*M) : \delta_g h = 0, tr_g h = 0, \Delta h + 2 Rm * h = 0\},
\end{align}
where $Rm * h$ is the operator defined above in \eqref{Rm*}. 
\end{theorem}
\noindent
Elements in the space $H^1_E$ are called {\em{infinitesimal Einstein 
deformations}}.  
\subsection{Ellipticity}
 
 The diffeomorphism invariance also means that the above equation cannot 
be elliptic. Indeed, differentiating 
\begin{align}
Ric( \varphi_t^*g) = \varphi_t^*( Ric(g) )
\end{align}
yields 
\begin{align}
Ric' (\mathcal{L}_X g) = \mathcal{L}_X (Ric(g)) = \lambda \cdot \mathcal{L}_X g.
\end{align}
\begin{exercise}{\em{Show that this implies that the symbol of 
$Ric'$ is not elliptic}}. 
\end{exercise}
We will next describe a procedure called ``gauging'' which shows in effect, 
that the diffeomorphism directions are the only obstruction to 
ellipticity. This is somewhat analogous to the ``Coulomb gauge'' in 
electrodynamics.


\subsection{A gauge choice}
Recall from above, that at an Einstein metric satisfying 
$Ric (g) = \lambda \cdot g$, the linearized Ricci tensor is given by
\begin{align}
(Ric')_{ij} &= \frac{1}{2} \Big(  - \Delta h_{ij} + \nabla_i (\delta h)_j 
+ \nabla_j (\delta h)_i- \nabla_i \nabla_j (tr h) - 2 R_{iljp}h^{lp} 
+  2 \lambda h_{ij} \Big).
\end{align}
Define the operator
\begin{align} 
\beta_g h = \delta_g h - \frac{1}{2} d (tr_g h).
\end{align}
\begin{exercise} Show that
\begin{align}
\frac{1}{2} \mathcal{L} \beta_g h = \frac{1}{2} \Big( \nabla_i (\delta h)_j 
+ \nabla_j (\delta h)_i- \nabla_i \nabla_j (tr h) \Big).
\end{align}
\end{exercise}
Combining the above expressions, 
\begin{align}
(Ric' - \frac{1}{2} \mathcal{L} \beta_g) h
= \frac{1}{2} \Big(  - \Delta h - 2 Rm * h + 2 \lambda h  \Big).
\end{align}


\section{The nonlinear map}
Given $\theta \in C^{2, \alpha}(S^2 T^*M)$, consider the map
\begin{align}
P_g : C^{2, \alpha}(S^2 (T^*M)) \rightarrow C^{0, \alpha}(S^2 (T^*M))
\end{align}
defined by
\begin{align}
P_g(\theta) = Ric ( g + \theta) - \lambda \cdot (g + \theta) 
-  \frac{1}{2} \mathcal{L}_{g+\theta} \beta_g \theta.
\end{align}
\begin{proposition}
The operator $P_g$ is elliptic. 
\end{proposition}
\begin{proof}
This is immediate: 
from the above, the linearized operator at $\theta = 0$ 
is 
\begin{align}
P_g' h = \frac{1}{2} \Big(  - \Delta h - 2 Rm * h \Big),
\end{align}
which is clearly elliptic. 
\end{proof}
We next see that zeroes of $P_g$ are in fact Einstein metrics.
\begin{proposition}
\label{zerop}
Assume that $\lambda < 0$. If $\theta \in C^{3,\alpha}$ is sufficiently near $g$ 
and satisfies $P_g(\theta) = 0$, then 
$Ric(g+\theta) = \lambda ( g + \theta)$, and $\theta \in C^{\infty}$.
\end{proposition}
\begin{proof}
Apply the operator $\beta_{g+\theta}$ to the equation
$P_g(\theta) = 0$ 
to obtain 
\begin{align}
\beta_{g+\theta} \mathcal{L}_{g+\theta} \beta_g \theta = 0
\end{align}
A computation shows that this is equivalently (exercise):
\begin{align}
( \Delta_{g+\theta} + Ric(g+\theta)) (\beta_g \theta) = 0. 
\end{align}
Since $\theta$ is sufficiently small in $C^{2,\alpha}$ norm, and $Ric(g)$ is 
strictly negative definite, then $Ric_{g+\theta}$ is also strictly 
negative definite. Pairing with $\beta_g \theta$  
and integrating by parts then shows that $\beta_g \theta = 0$. 
\end{proof}

\begin{exercise}{\em
Prove the regularity statement in Proposition \ref{zerop}. 
(Hint: Letting $\tilde{g} = g + \theta$, in harmonic coordinates, 
the Ricci tensor can be written in the form 
\begin{align}
Ric_{kl}(\tilde{g})
= -\frac{1}{2}\tilde{g}^{ij}\partial^{2}_{ij} \tilde{g}_{kl}+Q_{kl}(\partial \tilde{g}, \tilde{g})
\end{align}
where $Q(\partial \tilde{g},\tilde{g})$ is an expression that is quadratic in
$\partial \tilde{g}$,
polynomial in $\tilde{g}$ and has  $\sqrt{|\tilde{g}|}$ in its denominator.
Use a bootstrap argument in these coordinates. )
}
\end{exercise}
\begin{exercise}{\em
Show that we only need
to assume that $\theta \in C^{2,\alpha}$. (Hint: instead
of differentiating in the first step, integrate by parts.)
}
\end{exercise}
Next, we have a converse up to diffeomorphism: 
Einstein metrics near to $g$ can be gauged to yield zeroes of $P_g$.
\begin{proposition}
\label{convprop}
If $\tilde{g}$ is an Einstein metric near  $g$ with Einstein constant $\lambda$, 
then there exists a diffeomorphism $\varphi: M \rightarrow M$ such that 
$\tilde{\theta} = \varphi^* \tilde{g} - g$ satisfies 
$P_g(\tilde{\theta}) = 0$.
\end{proposition}
The proof uses a modified (infinitesimal) Ebin-Palais slice theorem using the 
Bianchi gauge:
\begin{lemma} \label{BasicSlice}  For each metric $g_1$ in a sufficiently small $C^{\ell+1,\alpha}$-neighborhood of $g$ ($\ell \geq 1$),
there is a $C^{\ell+2,\alpha}$-diffeomorphism $\varphi : M \rightarrow M$ such that
\begin{align} 
\tilde{\theta} \equiv \varphi^{*}g_1 - g
\end{align}
satisfies 
\begin{align}
\beta_g (\tilde{\theta}) = 0
\end{align}
\end{lemma}
\begin{proof}
The proof is almost identical to that of Theorem \ref{EP}, and is omitted. 
\end{proof}
\begin{proof}[Proof of Proposition \ref{convprop}]
If $\tilde{g}$ is Einstein 
then  $\varphi^{*}\tilde{g}$ is also Einstein. Since
\begin{align} 
\beta_g(\tilde{\theta}) = \beta_g( \varphi^{*}\tilde{g} - g) = 0,
\end{align}
we obviously obtain a zero of $P_g$. 
\end{proof}


\section{Structure of nonlinear terms}
Let us write
\begin{align}
P_g (\theta) = P_g(0) + P'_g(\theta) + Q_g(\theta).
\end{align}
The following proposition is crucial, and shows that 
the nonlinear term is manageable.
\begin{proposition}
\label{nonlinprop}
For $\theta_1, \theta_2$ sufficiently small, there exists a 
constant $C$ so that 
\begin{align}
\label{qnon}
\Vert Q_g(\theta_1) - Q_g(\theta_2) \Vert_{C^{0,\alpha}}
\leq C ( \Vert \theta_1 \Vert_{C^{2,\alpha}} + \Vert \theta_2 \Vert_{C^{2,\alpha}})
\Vert \theta_1 - \theta_2 \Vert_{C^{2,\alpha}}.
\end{align}
\end{proposition}
\begin{proof}
The proof is left as an exercise, with a few hints.  
First, show that 
\begin{align}
\Gamma(g + h)^{k}_{ij} = \Gamma(g)^{k}_{ij} + \frac{1}{2} (g+h)^{km}
\left\{ \nabla_j h_{im} + \nabla_i h_{jm} - \nabla_m h_{ij} \right\}.
\end{align}
In shorthand, we can therefore write the covariant derivative of
any tensor $T$ as
\begin{align}
\label{nablat}
\nabla_{g+ h} T = \nabla_g T + (g +h)^{-1} * \nabla_g h * T,
\end{align}
where the notation $*$ denotes taking various contractions of the 
tensors involved (the exact indices contracted do not matter for
the conclusion). 

Next, for any metric $\tilde{g}$, the curvature tensor can be written 
in shorthand as 
\begin{align}
Rm_{\tilde{g}} = \nabla_{\tilde{g}} \Gamma_{\tilde{g}}. 
\end{align}
Using \eqref{nablat}, show that this implies an expansion of the form 
\begin{align}
Rm(g+h) = Rm(g) + (g + h)^{-1} * \nabla^2 h + (g+h)^{-2}* \nabla h * \nabla h.
\end{align}
Contract with $(g +h)^{-1}$ to get $Ric(g+h)$ and then use the formula
\begin{align}
(g+h)^{-1} = g^{-1} - g^{-1} (g + h)^{-1} h.
\end{align}
to pull out the terms in the linearization, and \eqref{qnon} will then 
follow from the resulting expression for $Q_g$. 
\end{proof}
\section{Existence of the Kuranishi map}
The following is the main tool used to construct the map $\Psi$,
see for example \cite[Lemma~8.3]{Biquard}.
\begin{lemma}
\label{IFT}
Let $H : E \rightarrow F$ be a smooth map between Banach spaces.
Define $Q = H - H(0) - H'(0)$. Assume that there are positive constants
$C_1, s_0, C_2$ so that the following are satisfied:
\begin{itemize}
\item $(1)$ The nonlinear term $Q$ satisfies
\begin{align*}
\Vert Q(x) - Q(y) \Vert_F \leq C_1 (\Vert x \Vert_E + \Vert y \Vert_E)
\Vert x - y \Vert_E
\end{align*}
for every $x, y \in B_E (0, s_0)$.
\item $(2)$ The linearized operator at $0$, $H'(0) : E \rightarrow F$
is an isomorphism with inverse bounded by $C_2$.
\end{itemize}
If $s$ and $\Vert H(0) \Vert_F$ are sufficiently small (depending
upon $C_1, s_0, C_2$), 
then there is a unique solution $x \in B_E(0,s)$ of the
equation $H(x) = 0$.
\end{lemma}
\begin{proof}[Outline of Proof]
The equation $H(x) = 0$ expands to
\begin{align}
H(0) + H'(0) (x) + Q(x) = 0. 
\end{align}
If we let $x = G y$, where $G$ is the inverse of $H'(0)$, then we have
\begin{align}
H(0) + y + Q(Gy) = 0,
\end{align}
or
\begin{align}
y = - H(0) - Q(Gy).
\end{align}
In other words, $y$ is a fixed point of the mapping 
\begin{align}
T :  y \mapsto - H(0) - Q(Gy).
\end{align}
With the assumptions in the lemma, it follows that $T$ is a 
contraction mapping, so a fixed point exists by the 
standard fixed point theorem ($T^n y_0$ converges 
to a unique fixed point for any $y_0$ sufficiently small).  
\end{proof}
To prove Theorem \ref{existthm}, we next construct the map 
\begin{align}
\Psi: H^1_E \rightarrow H^1_E,
\end{align}
whose zero set is locally isomorphic to the zero set of $P$.
Consider $H = \Pi \circ P$, where $\Pi$ is projection to the 
orthogonal complement of $H^1_E$.
The differential of this map is now surjective. 
Choose any complement $K$ to the
space $H^1_E$, and restrict the mapping to this complement.
Equivalently, let $G$ be any right inverse, i.e.,
$H'(0) G = Id$, and let $K$ be the image of $G$. 
Given a kernel element $x_0 \in H^1_E$, the
equation $H(x_0 + G y ) = 0$ expands to 
\begin{align}
H(0) + H'(0) ( x_0 + Gy) + Q(x_0 + Gy) = 0.
\end{align}
We therefore need to find a fixed point of the map
\begin{align}
T_{x_0} : y \mapsto - H(0) - Q(x_0 + Gy),
\end{align}
and the proof is the same as before. 

To finish the proof of Theorem \ref{existthm}, we need to identify the 
kernel of the linearized operator.
\begin{proposition}
\label{lgpr}
If $\lambda < 0$, then $Ker(P_g')$ consists exactly of transverse-traceless 
tensors satisfying 
\begin{align}
\Delta h + 2 Rm * h = 0.
\end{align}
\end{proposition}
\begin{proof}
If $P'(h) = 0$, then $h$ is smooth by elliptic regularity, Also, 
\begin{align}
P' h = Ric'(h) - \lambda h -  \frac{1}{2} \mathcal{L}_{g} \beta_g h.
\end{align}
Applying $\beta_g$ to this equation, yields $\beta_g \mathcal{L}_{g} \beta_g h = 0$, 
so $\beta_g h = 0$ by the above argument. Taking a trace, we find that 
\begin{align}
\Delta (tr_g(h)) + 2 \lambda \cdot tr_g(h) = 0, 
\end{align}
so $tr_g(h) = 0$ since $\lambda < 0$. 
\end{proof}


\section{Rigidity of Einstein metrics}
In general it is quite difficult to construct the map $\Psi$
explicitly, but one of the easiest consequences of the 
above discussion is the following (see \cite{Koiso1978}):
\begin{corollary}If $Ric(g) = \lambda \cdot g$ with $\lambda < 0$, 
and $H^1_E = \{0\}$ then $g$ is rigid (isolated as an Einstein metric). 
That is, if $g_t$ is a path of Einstein metrics passing through $g$,
all with Einstein constant $\lambda < 0$, then there exist a path 
of diffeomorphisms $\varphi_t : M \rightarrow M$ such that 
$g_t = \varphi_t^* g$. 
\end{corollary} 

We next discuss a few cases where it is known that $H^1_E = 0$.
For a longer list, see \cite{Koiso1978, Koiso2, Koiso3}. 
\subsection{The negative case}

In the hyperbolic case, in Lecture \ref{L1} we proved that for $n \geq 3$, 
$H^1_E = \{0\}$, so hyperbolic manifolds are locally rigid as Einstein metrics. 
In fact, something much stronger is true in dimension four:
\begin{theorem}[Besson-Courtois-Gallot \cite{BCG}]
If $(M^4,g)$ is compact and hyperbolic, then $g$ is the 
unique Einstein metric on $M$, up to scaling.
\end{theorem}
This is proved using completely different methods than we have 
discussed here (using the notion of volume entropy), 
which we will not have time to go into.  This is a generalization 
of Mostow rigidity, which says that hyperbolic metrics 
are determined up to scaling by homotopy type in dimensions $n \geq 3$ \cite{Mostow}. 
An analogous rigidity result was proved for complex hyperbolic $4$-manifolds 
using Seiberg-Witten Theory in \cite{LeBrunMostow}. 

Einstein metrics with negative sectional curvature are
also locally rigid:
\begin{exercise}{\em Show 
that any Einstein metric with negative sectional 
curvature is rigid, that is, $H^1_E = \{0\}$.  
}
\end{exercise}
In the case of $n =2$, we saw in Lecture \ref{L1} that elements of 
$H^1_E$ are Codazzi tensors. Using some elementary Riemann surface
theory, it is possible to identify these with the space
of real parts of holomorphic quadratic differentials, see for example \cite{EE}. 
By the Riemann-Roch Theorem, this space is of real dimension $6 \ell - 6$, where $\ell$ is
the genus for genus $\ell \geq 2$ \cite{DonaldsonRiemann}. 

 In the case of $n \geq 3$, Codazzi tensors do not give infinitesimal 
Einstein deformations of a hyperbolic metric. However, they do 
yield infinitesimal deformations of the locally conformally 
flat structure \cite{Lafontaine}. Hyperbolic metrics admitting such deformations 
are called {\em{bendable}}. For examples, see \cite{JohnsonMillson}.


\subsection{The positive case}

It is possible to modify the above construction so that it 
works also in the positive Einstein case, but we leave this as
an exercise. The main difference is that the gauge term 
should be chosen differently. 
We will just mention two issues that arise.
\begin{itemize}
\item
A positive Einstein metric 
can admit a nontrivial group of isometries (identity component). 
Letting ${\rm{Isom}}(M,g)$ denote the isometry group, 
we note that ${\rm{Isom}}(M,g)$ acts on the space of symmetric 
tensors, and therefore, by linearizing at the identity transformation, 
so does the space of Killing fields $\mathfrak{K}$, which is the Lie algebra of ${\rm{Isom}}(M,g)$. 
Taking this action into account, the end result is that the map $\Psi$ is 
equivariant with respect to the isometry group, and 
the actual moduli space is locally described by $\Psi^{-1}(0) / \mathfrak{K}$,
rather than just $\Psi^{-1}(0)$. 

\vspace{2mm}
\item  
In the case of the sphere, we run into the problem of 
first nontrivial eigenfunctions yielding pure trace kernel. 
However, these can also be ``gauged away'' since they arise
as tangents to conformal diffeomorphisms. See for example
\cite[Section 6]{GVRS} for details. 
\end{itemize}

\noindent
Next, we discuss a few known rigid examples in the positive case:
\begin{itemize}

\item Any spherical space form $S^n / \Gamma$ with the round metric $g_S$.  
In this case, we saw in Lecture \ref{L1} that the linearized operator
obviously has trivial kernel, see Corollary \ref{scorr}. 
\vspace{2mm}
\item  $S^2 \times S^2$ with the product metric $g_1 + g_2$. In this case, 
 we saw in Exercise~\ref{s2s2ex} that the first two eigenvalues of 
of the linearized operator are $-1$ and $1$, thus $0$ does 
not occur as an eigenvalue.

\vspace{2mm}
\item $(\CP^n, g_{{\rm{FS}}})$, where $g_{\rm{FS}}$ is the Fubini-Study metric. 
We will not have time to prove this case in these lectures,
since the nicest proof involves the theory of deformations of 
K\"ahler-Einstein metrics, and needs a considerable amount of
background in complex geometry.
\end{itemize}
It is remarked that there are examples of positive Einstein metrics 
admitting nontrivial infinitesimal Einstein deformations, 
yet which are rigid as Einstein metrics, for example
$S^2 \times \CP^{2\ell}$ \cite{Koiso3}. 
This shows that, in general, determining the Kuranishi 
is not an easy problem.


\subsection{The zero case}

 We saw that, in the case of a flat metric, $H^1$ consists of 
parallel sections. The Kuranishi map turns out to be identically 
zero in this case, since all of these are ``integrable'', 
corresponding to deformations of the flat structure. 

\begin{exercise}
\label{m2e}
{\em
Determine the moduli space of flat structures on a $2$-torus.
(Hint: a flat structure is equivalent to a choice of 
lattice in $\RR^2$.) 
}
\end{exercise}
Another special class of Ricci-flat metrics are Calabi-Yau 
metrics, which are K\"ahler Ricci-flat metrics. 
It is known that Calabi-Yau metrics are {\em{unobstructed}}:
\begin{theorem}[Bogomolov-Tian] 
\label{BTthm}
For a Calabi-Yau metric $(X,g)$,
the Kuranishi map $\Psi \equiv 0$. That is, every infinitesimal 
Einstein deformation integrates to an actual deformation. 
\end{theorem}
We will not discuss
the proof, since it involves a considerable amount 
of complex geometry, see \cite{Bogomolov, TianCY}. 

 We also mention that Dai-Wei-Wang have proved stability results 
for manifolds admitting a parallel spinor \cite{DWW1}.


\lecture{Quadratic curvature functionals}
\label{L5}

\section{Quadratic curvature functionals}

First recall that the curvature tensor admits the orthogonal 
decomposition 
\begin{align}
\label{d2}
Rm = W + \frac{1}{n-2}  E  \varowedge g
+ \frac{R}{2n(n-1)} g \varowedge g,
\end{align}
where 
\begin{align} 
E =  Ric - \frac{R}{n} g
\end{align}
is the {\em{traceless Ricci tensor}}.
The $\varowedge$ symbol is the Kulkarni-Nomizu product,  
which takes $2$ symmetric 
$(0,2)$ tensors and produces a $(0,4)$ tensor with the 
same algebraic symmetries of the curvature 
tensor, and is defined by
\begin{align*}
A \varowedge B (X,Y,Z,W) = &A(X,Z)B(Y,W) - A(Y,Z)B(X,W)\\
 &- A(X,W)B(Y,Z) + A(Y,W)B(X,Z).
\end{align*}

The tensor $W$ occurring in \eqref{d2} is called the Weyl tensor
(use \eqref{d2} to {\em{define}} the Weyl tensor), 
and is the part of the curvature tensor which lies in the
kernel of the Ricci contraction map. 
An important property of the Weyl tensor is given by:
\begin{exercise}{\em
The Weyl tensor, viewed as a $(1,3)$-tensor 
with components $W_{ijk}^{\ \ \ l}$, is pointwise 
conformally invariant. That is, if $\tilde{g} = f \cdot g$ 
where $f$ is a strictly positive function, then 
$W(\tilde{g}) = W(g)$.
}
\end{exercise}
We will now turn our attention to functionals on the space of Riemannian metrics $\mathcal{M}$ which are
quadratic in the curvature; see \cite{Besse, Blair, Smolentsev} for surveys.

A basis for the space of quadratic curvature functionals is
\begin{align}
\label{qbasis} 
\mathcal{W}(g) = \int_M |W_g|^2 dV_g, \ \ \rho(g) = \int_M |Ric_g|^2 dV_g,
\ \ \mathcal{S}(g) = \int_M R_g^2 dV_g.
\end{align}
Let us now restrict the rest of this lecture 
to dimension four. In this dimension, for $M$ compact without boundary, 
the Chern-Gauss-Bonnet formula states that
\begin{align}
\label{CGB}
32\pi^2 \chi(M) = \int_M |W_g|^2 dV_g - 2 \int_M |Ric_g|^2 dV_g 
+ \frac{2}{3} \int_M R_g^2 dV_g,
\end{align}
where $\chi(M)$ is the Euler characteristic of $M$. 
This implies that any one of the functionals in \eqref{qbasis} 
can be written as a linear combination of the other two (plus a topological term).

\begin{remark}{\em
We are using the tensor norm on $|W|^2$, that is 
$|W|^2 = W^{ijkl}W_{ijkl}$, where all indices are
summed from $1$ to $4$. This differs from the norm of $W$ as a 
mapping on $2$-forms by a factor of $4$, that is 
\begin{align}
|W|^2 = 4 \Vert \widehat{W} \Vert^2
\end{align}
(see below for the definition of $\widehat{W}$).
}
\end{remark}
We next present the Euler-Lagrange equations of these functionals:
\begin{proposition}[Berger \cite{Berger}]
\label{ELprop} 
The Euler-Lagrange equations of the functionals in
\eqref{qbasis} are given by 
\begin{align}
\label{gradW}
(\nabla \mathcal{W})_{ij} &=  -4 \Big( \nabla^k \nabla^l W_{ikjl} + \frac{1}{2} R^{kl}W_{ikjl} \Big),\\
\label{gradformF}
(\nabla \rho)_{ij} &  =
- \Delta (Ric)_{ij} - 2 R_{ikjl} R^{kl} + \nabla_i \nabla_j R
- \frac{1}{2}(\Delta R) g_{ij} + \frac{1}{2} |Ric|^2 g_{ij},\\
\label{s'}
( \nabla \mathcal{S})_{ij} & =  2 \nabla_i \nabla_j R - 2 (\Delta R) g_{ij}
- 2 R R_{ij} +   \frac{1}{2} R^2 g_{ij}.
\end{align}
\end{proposition}
\begin{proof}[Outline of proof.]
Let $g(t)$ be a path of metrics such that $g(0) = g$ and $g'(0) = h$. 
Using the formula for the derivative of the inverse of a matrix
\begin{align}
(g^{pq})' = - g^{pk} h_{kl} g^{lq}, 
\end{align}
the formula for the derivative of the volume element \eqref{dv'},
and the formula for the linearization of the Ricci tensor \eqref{linric},
equation \eqref{gradformF} follows upon integrating by parts. 
Next, recalling the formula for the linearization of 
the scalar curvature \eqref{r'},
the formula \eqref{s'} follows similarly.

Finally,  instead of computing the linearization of the Weyl 
tensor directly, use the Chern-Gauss-Bonnet formula \eqref{CGB}
to express the Euler-Lagrange equations of $\mathcal{W}$ as a
linear combination of the Euler-Lagrange equations of the other two functionals.  
Note that the formula obtained in this way shows that 
$\nabla \mathcal{W}$ depends only upon the Ricci tensor (and it covariant 
derivatives). Use the Bianchi identities to show these are
equivalent to the form \eqref{gradW}. 
\end{proof}
We point out some obvious critical metrics:
\begin{proposition} 
Any Einstein metric is critical for all of the functionals 
$\mathcal{W}, \mathcal{\rho},$ and $\mathcal{S}$. 
Also, any scalar-flat metric is critical for $\mathcal{S}$.  
\end{proposition}
\begin{proof}These statements follow 
easily from the Euler-Lagrange equations computed above
in Proposition \ref{ELprop}. 
\end{proof}

We note that the functional $\mathcal{S}$ has been deeply studied in 
K\"ahler geometry; critical points of this functional when 
restricted to a K\"ahler class are known as extremal K\"ahler metrics,
see \cite{ CalabiI, CalabiII}. The Euler-Lagrange equations of
the restricted functional are that the gradient of the scalar 
curvature is the real part of a holomorphic vector field. 
In particular, constant scalar curvature K\"ahler metrics 
are extremal.

  Also, the functional $\mathcal{W}$ has been studied 
in depth (see for example \cite{GurskyWeyl}), and it was
introduced by Rudolf Bach in 1921 \cite{Bach}.  
Thus the Euler-Lagrange tensor $\nabla \mathcal{W}$ is 
known as the {\em{Bach tensor}}. Conformal invariance
of the functional $\mathcal{W}$ implies that the Bach tensor
is also conformally invariant. 


\section{Curvature in dimension four}
If $(M^4, g) $ is oriented, the Hodge star operator on $\Lambda^2$
satisfies $*^2 = I$. The space of $2$-forms then decomposes into 
\begin{align}
\label{l2dec}
\Lambda^2 = \Lambda^2_+ \oplus \Lambda^2_-,
\end{align}
the $+1$ and $-1$ eigenspaces of the Hodge star operator, 
respectively. Note that $dim_{\RR}( \Lambda^2) = 6$, 
and $dim_{\RR}( \Lambda^2_{\pm}) = 3$. 
Elements of $\Lambda^2_+$ are called {\em{self-dual}} $2$-forms, 
and elements of $\Lambda^2_-$ are called {\em{anti-self-dual}} $2$-forms

We fix an oriented orthonormal basis $\{e_1, e_2, e_3, e_4\}$ and
denote the dual basis by $\{e^1, e^2, e^3, e^4\}$. Define 
\begin{align*}
\omega_1^{\pm} = e^1 \wedge e^2 \pm e^3 \wedge e^4,\\
\omega_2^{\pm} = e^1 \wedge e^3 \pm  e^4 \wedge e^2,\\
\omega_3^{\pm} = e^1 \wedge e^4 \pm  e^2 \wedge e^3.
\end{align*}
Note that $ * \omega_i^{\pm} = \pm \omega_i^{\pm}$, 
and $\frac{1}{\sqrt{2}} \omega_i^{\pm}$ is an orthonormal basis 
of $\Lambda^2_{\pm}$.  

In dimension $4$ there is the special coincidence that 
the curvature operator acts on $2$-forms, and the 
space of $2$-forms decomposes as above.
Recall from above that full curvature tensor decomposes as 
\begin{align}
\label{d22}
Rm = W + \frac{1}{2}  E  \varowedge g + \frac{R}{24} g \varowedge g.
\end{align}
Consider the curvature tensor as a mapping on $2$-forms 
defined by 
\begin{align}
\label{newpo}
\widehat{Rm} (\omega) & =  \frac{1}{4} \sum_{i,j, k,l} R_{ijkl} \omega_{kl}  e^i \wedge e^j,
\end{align}
where 
\begin{align}
\omega = \frac{1}{2} \sum_{i,j} \omega_{ij} e^i \wedge e^j.
\end{align}
We call this mapping the {\em{curvature operator}},
which has a corresponding decomposition, see \cite{SingerThorpe}:
\begin{align}\label{d4c}
\widehat{Rm}=
\left(
\mbox{
\begin{tabular}{c|c}
&\\
$\widehat{W}^+ + \frac{R}{12} I $&$ \widehat{E} $\\ &\\
\cline{1-2}&\\
$\widehat{E}$ & $\widehat{W}^- + \frac{R}{12} I$\\&\\
\end{tabular}
} \right).
\end{align}
The operators $\widehat{W}^+$ and $\widehat{W}^-$ 
are traceless as endomorphisms of $\Lambda^2_+$ and $\Lambda^2_-$,
respectively. 
\begin{exercise}{\em
 Prove the decomposition \eqref{d4c}. 
(Hint: this is equivalent to saying that $\widehat{W}$ 
commutes with $*$ and $\mathcal{E}$ anti-commutes with $*$, 
where $\widehat{W}$ and $\widehat{E}$ are the operators 
on $2$-forms corresponding to $W$ and $\frac{1}{2} E \varowedge g$, respectively.)
}
\end{exercise} 
A beautiful theorem relates the $L^2$-norms of the tensors $W^{\pm}$
to the topology of the manifold, and is called the Hirzebruch Signature Theorem.
\begin{theorem}[Hirzebruch \cite{Hirzebruch}] 
\label{HST}
Let $(M^4,g)$ be compact and
oriented.  Then
\begin{align*}
48 \pi^2 \tau(M) = \int_M |W^+_g|^2 dV_g - \int_M |W^-_g|^2 dV_g,
\end{align*}
where $\tau = b_2^+ - b_2^-$ is the signature of $M$.
\end{theorem}

\section{Einstein metrics in dimension four}
From Proposition \ref{ELprop}, it is not hard to see that 
any Einstein metric is critical for all three functionals
in \eqref{qbasis}. 
One of the only known obstructions to the existence of 
Einstein metrics is the following inequality which is 
called the Hitchin-Thorpe Inequality.
\begin{theorem}[Hitchin-Thorpe] If $(M^4,g)$ is Einstein and 
oriented, then 
\begin{align}
\label{HTI}
2 \chi(M) \geq 3| \tau(M)|,
\end{align}
with equality if and only if $g$ is flat or finitely covered by 
a $K3$ surface with a Ricci-flat metric. 
\end{theorem}
\begin{proof}
This inequality follows from the Chern-Gauss-Bonnet Formula 
and Hirzebruch Signature Theorem, this was first noted in 
\cite{Thorpe}, and the equality case was characterized by 
Hitchin in \cite{Hitchin1}. 

 We give an outline of the equality case: 
if equality holds, then 
one concludes that $R = 0$ and that either $W^+ = 0$ or 
$W^- = 0$. Reversing orientation if necessary, we may 
assume that $W^+ = 0$. If $g$ is flat, then we are done, 
so assume that $W^-$ does not vanish identically.  
Since $g$ is assumed to be Einstein, 
then the entire top half of \eqref{d4c} vanishes. 
This says that the bundle $\Lambda^2_+(T^*M)$ is flat. 
Since $W^- \not\equiv 0$, the Chern-Gauss-Bonnet theorem implies that $\chi(M) > 0$.
If $b_1(M)$ (the first Betti number) were non-zero, then 
from Hodge Theory, there would exists a non-trivial 
harmonic $1$-form $\alpha$. But since $Ric \equiv 0$, by the Bochner
formula on $1$-forms \eqref{Bochner}, $\alpha$ would be parallel. The
dual vector field would be a non-zero vector field on 
$M$, contradicting the fact that $\chi(M) > 0$ (from 
the Poincar\'e-Hopf Index Theorem for vector fields). 
So we conclude that $b_1(M) = 0$. By the Cheeger-Gromoll
splitting theorem, we can then conclude that $\pi_1(M)$ is 
finite \cite{CheegerGromoll}, so we just consider the 
universal cover $\tilde{M}$ of $M$. 
Since the bundle  $\Lambda^2_+(T^*\tilde{M})$ is flat and
$\tilde{M}$ is simply-connected, it must be
trivial, and consequently the holonomy can be reduced to $\rm{SU}(2)$,
which implies that $g$ is K\"ahler with vanishing first
Chern class, and must therefore be a K3 surface. 
\end{proof}
\begin{exercise}{\em
(i) Show that if $k \# \CP^2$ admits an Einstein metric then $k \leq 3$. 

\noindent
(ii) Show that if $\CP^2 \# k \overline{\CP}^2$ admits an Einstein 
metric then $k \leq 8$. 
}
\end{exercise}
We also mention there are improvements of \eqref{HTI} using 
Seiberg-Witten Theory, see for example \cite{LeBrunEMYP}.

Next, we will list some examples of Einstein metrics in dimension $4$. 
The only known compact examples with positive Einstein constant in 
dimension four:

\begin{itemize}
\vspace{2mm}
\item $S^4$ or $\RP^4$ with the round metric.

\vspace{2mm}
\item $S^2 \times S^2$ with the product metric, its orientable $\ZZ/2\ZZ$ quotient
$G(2,4)$, $\RP^2 \times \RP^2$ with the product metric, and 
$S^2 \times \RP^2$ with the product metric. 

\vspace{2mm}
\item $\CP^2$ with the Fubini-Study metric. 

\vspace{2mm}
\item $\CP^2 \# \overline{\CP}^2$ with the Page metric, an explicit ${\rm{U}}(2)$-invariant Einstein metric, see \cite{Page}. 
This admits a non-orientable quotient $\CP^2 \# \RP^4$.  

\vspace{2mm}
\item  $\CP^2 \# 2 \overline{\CP}^2$ with the Chen-LeBrun-Weber metric
\cite{CLW}. This metric is conformal to an extremal K\"ahler metric. See 
Section~\ref{LCLW} in Lecture \ref{L9} below for a more
discussion of this metric. 

\vspace{2mm}
\item $\CP^2 \# k \overline{\CP}^2$, $k = 3, \dots, 8$, admits positive
K\"ahler-Einstein metrics (Tian-Yau \cite{TianYau}, Tian \cite{TianCalabi}).
\end{itemize}

It is an interesting problem to find other topological manifolds 
admitting positive Einstein metrics, and also to possibly find
other Einstein metrics on 
the manifolds listed above. For example, 
it is unknown whether $S^4$ admits an Einstein metric which is
not of constant curvature.  However, if one exists, it is known that its 
Yamabe energy cannot be too large.
\begin{theorem}[Gursky \cite{GurskySphere}] Suppose $S^4$ 
admits a positive Einstein metric $g$ which is not isometric
to the standard round metric, normalized so that $Ric(g) = 3g$. 
Then 
\begin{align}
Vol(g) < \frac{8}{9} \pi^2 = \frac{1}{3} Vol(g_S).
\end{align}
\end{theorem}
The only known compact examples with zero Einstein constant in 
dimension four:
\begin{itemize}
\item K3 surface with Calabi-Yau Ricci-flat metric \cite{Yau}, and its quotients.

\vspace{2mm}
\item Flat metrics. 
\end{itemize}

There are many more examples of Einstein metrics with negative Einstein 
constant. Of course, any hyperbolic manifold is an example. Complex 
hyperbolic manifolds are another interesting class with negative Einstein 
constant \cite{LeBrunMostow}. 
Any K\"ahler manifold with $c_1 < 0$ carries an Einstein 
metric by Aubin-Yau \cite{AubinCY,Yau}. There are in fact many 
such manifolds, for example, any non-singular hypersurface in $\CP^3$ 
of degree $d > 4$ satisfies $c_1 < 0$. The case $d = 4$ is the K3 surface
which has $c_1 = 0$, and carries the Calabi-Yau Ricci-flat metric mentioned above.
Also, see \cite{AndersonSurvey} for a nice survey and other examples. 
\subsection{Higher dimensions}
 We note that, in dimensions $n > 4$, there is no known topological 
obstruction to the existence of an Einstein metric. There are in 
fact quite a large number of known examples of Einstein metrics in 
higher dimensions (too many to list here).
It could be the case that every compact manifold of dimension $n > 4$ admits
an Einstein metric. 

\section{Optimal metrics}
Another interesting class of metrics are called optimal metrics, and
are defined to be those metrics which {\em{globally}} minimize the functional
\begin{align}
\mathcal{R}(g) = \int_M |Rm_g|^2 dV_g. 
\end{align}
Using the formula 
\begin{align}
|Rm_g|^2 = |W_g|^2 + 2 |Ric_g|^2 - \frac{1}{3} R_g^2, 
\end{align}
the Chern-Gauss-Bonnet formula \eqref{CGB} may be written 
\begin{align}
\mathcal{R}(g) = 32 \pi^2 \chi(M) + 4 \int_M |E_g|^2 dV_g, 
\end{align}
which shows immediately that Einstein metrics are necessarily optimal.

We will not go into much more details about optimal 
metrics in general, but just make a few remarks 
taken from \cite{LeBrunOptimal}:
\begin{itemize}

\vspace{2mm}
\item There are optimal metrics which are not Einstein. 

\vspace{2mm}
\item There are compact $4$-manifolds which do not admit optimal metrics. 

\vspace{2mm}
\item There are topological $4$-manifolds which admit an optimal 
metric for some smooth structure, but do not admit any optimal 
metric for a different smooth structure. 
\end{itemize}
We may also write
\begin{align}
\label{optif}
\mathcal{R}(g) = - 32 \pi^2 (\chi(M) + 3 \tau(M)) +  \int_M \left(
\frac{R^2_g}{3} + 4 |W^+_g|^2 \right) dV_g.
\end{align}
Thus we see that another class of optimal metrics are those
with $W^+ \equiv 0$ and $R =0$, these are called 
scalar-flat anti-self-dual metrics, which we will discuss in 
more detail next. 

\begin{remark}{\em
On a related note, we mention that the functionals $\rho$ and $\mathcal{S}$ 
are known to be globally minimized in dimension four at a negative K\"ahler-Einstein
metric. This is proved in \cite{LeBrunRCMV} using Seiberg-Witten
theory, along with many other interesting results regarding minimal volumes.
}
\end{remark}
 \begin{exercise}{\em(i) Prove \eqref{optif} using 
the Hirzebruch Signature Theorem \ref{HST}. 

\noindent
(ii)
Show that
\begin{align}
32 \pi^2 (\chi(M) - 3 \tau(M)) &= \int_M \left( -|W^+_g|^2  + \frac{1}{6} R_g^2
+3 |W^-_g|^2  - 2|E_g|^2 \right) dV_g.
\end{align}
If $g$ is K\"ahler then the first two terms cancel (see \eqref{w+r} below), 
and one is left with 
\begin{align}
32 \pi^2 (\chi(M) - 3 \tau(M)) &= \int_M \left( 3 |W^-_g|^2  - 2|E_g|^2 \right) dV_g
\end{align}
(see \cite{Derdzinski}).
Show that this implies that $g_{\rm{FS}}$ is the unique K\"ahler-Einstein 
metric on $\CP^2$, up to scaling.
}
\end{exercise}


\section{Anti-self-dual or self-dual metrics}
In dimension $4$, the curvature condition

\vspace{2mm}
$W^+ = 0$ is called anti-self-dual (ASD),

\vspace{2mm}
$W^- = 0$ is called self-dual (SD).

\vspace{2mm}
Since the Weyl tensor is conformally invariant and 
the Hodge star operator on $2$-forms in dimension $4$ is 
also conformally invariant, it follows that
either condition is conformally invariant.

\begin{proposition}
If $(M^4, g_0)$ compact and  self-dual, then $g_0$ is a global minimizer of the 
functional $\mathcal{W}$, so is necessarily a critical metric for
$\mathcal{W}$.
\end{proposition}
\begin{proof}
For any metric $g$ on $M$, using the Hirzebruch Signature Theorem, we have
\begin{align*}
\mathcal{W}(g) = \int_M |W_g|^2\ dV_g &= \int_M |W^+_g|^2\ dV_g +  \int_M |W^-_g|^2\ dV_g\\
&= 48 \pi^2 \tau(M) + 2 \int_M |W^-_g|^2 dV_g \geq  48 \pi^2 \tau(M),
\end{align*}
with equality if and only if $g$ is self-dual. 
\end{proof}

The only obvious topological obstruction to existence of a self-dual 
or anti-self-dual metric comes from the Hirzebruch Signature Theorem.
\begin{proposition}
If $(M,g)$ is self-dual (anti-self-dual) then $\tau \geq 0$ ($\tau \leq 0$) 
with equality if and only if 
$g$ is locally conformally flat.
\end{proposition}

One can get a stronger restriction if one assumes the scalar curvature is positive:
\begin{proposition}
If $(M,g)$ is self-dual (anti-self-dual) and $R > 0$, then $b_2^- = 0$  ($b_2^+ = 0$). 
\end{proposition}
\begin{proof} On $2$-forms in dimension $4$, the Weitzenb\"ock formula is 
\begin{align}
\label{wb2}
(\Delta_H \omega^{\pm})_{ij} = - (\Delta \omega^{\pm})_{ij} 
-  \sum_{l,m} W_{lmij}^{\pm} \omega_{lm}^{\pm} + \frac{R}{3} \omega_{ij}^{\pm}. 
\end{align}
The result follows by choosing $\omega^{\pm}$ be harmonic,
pairing \eqref{wb2} with $\omega^{\pm}$, and integrating. 
\end{proof}
\lecture{Anti-self-dual metrics}
\label{L6}

\section{Deformation theory of anti-self-dual metrics}

There are now a wealth of examples of anti-self-dual 
metrics on $4$-manifolds. But before we get into specific
examples, let us consider the deformation theory 
of these metrics. 
Let $(M,g)$ denote an anti-self-dual $4$-manifold (the deformation 
theory of self-dual metrics is identical upon reversing orientation). 
The anti-self-dual 
deformation complex is given by 
\begin{align}
\label{thecomplex}
\Gamma(T^*M) \overset{\mathcal{K}_g}{\longrightarrow} 
\Gamma(S^2_0(T^*M))  \overset{\mathcal{D}}{\longrightarrow}
\Gamma(S^2_0(\Lambda^2_+)),
\end{align}
where $\mathcal{K}_g$ is the conformal Killing operator defined 
by 
\begin{align}
( \mathcal{K}_g(\omega))_{ij} = \nabla_i \omega_j + \nabla_j \omega_i - 
\frac{1}{2} (\delta \omega) g_{ij}, 
\end{align}
with $\delta \omega = \nabla^i \omega_i$,  
$S^2_0(T^*M)$ denotes traceless symmetric tensors, 
and $\mathcal{D} = (\mathcal{W}^+)_g'$ is the linearized self-dual Weyl curvature 
operator. This complex is elliptic \cite{KotschickKing}. 

The cohomology groups of this complex are given as follows:
\begin{align}
H^0_c(M,g) = \{ \kappa \in T^*M \ | \ \mathcal{K}_g \kappa = 0 \}.
\end{align}
Letting $\mathcal{D}_g$ denote $(W^+)'_g$, we have
\begin{align}
H^1_c(M,g) = \{ h \in S^2_0(T^*M) \  | \ \mathcal{D}_g h = 0, \ \delta_g h = 0 \}.
\end{align}
Finally, 
\begin{align}
H^2_c(M,g) = \{ Z \in S^2_0(\Lambda^2_+) \ | \ \mathcal{D}_g^* Z = 0 \},
\end{align}
where $\mathcal{D}_g^*$ is the formal $L^2$-adjoint of $\mathcal{D}_g$. 
 
If $M$ is a compact manifold then there is a formula for the index depending 
only upon topological quantities. The analytical index is given by
\begin{align}
Ind(M, g) = \dim( H^0_c(M,g)) -  \dim( H^1_c(M,g)) + \dim( H^2_c(M,g)).
\end{align}
The index is given in terms of topology via the Atiyah-Singer index theorem:
\begin{theorem}[I.M. Singer]
\label{indthm}
If $(M,g)$ is compact and anti-self-dual, then 
\begin{align}
\label{manifoldindex}
Ind(M, g)
=\frac{1}{2} ( 15 \chi(M) + 29 \tau(M)).
\end{align}
where $\chi(M)$ is the Euler characteristic and 
$\tau(M)$ is the signature of $M$
\end{theorem}
This was first computed by I.M. Singer in 1978, see also \cite[page~303]{LeBrunJAMS}
and \cite{KotschickKing}.

The cohomology groups of the complex \eqref{thecomplex} yield information about 
the local structure of the moduli space of anti-self-dual conformal
classes, which we discuss next (see also \cite{Itoh2, KotschickKing}).
Letting ${\rm{Conf}}(M,g)$ denote the conformal automorphism group, 
we note that ${\rm{Conf}}(M,g)$ acts on the space of symmetric 
tensors, and therefore, by linearizing at the identity transformation, 
so does $H_c^0(M,g)$, since the space of conformal 
Killing Fields is the Lie algebra of ${\rm{Conf}}(M,g)$. 
\begin{theorem}
If $(M,g)$ is anti-self-dual, then there is a map 
\begin{align}
\Psi: H_c^1(M,g) \rightarrow H_c^2(M,g)
\end{align} 
called the {\em{Kuranishi map}}
which is equivariant with respect to the action of the 
conformal group. The moduli space of anti-self-dual conformal
structures near $g$ (anti-self-dual metrics modulo diffeomorphism and conformal equivalence) is 
locally isomorphic $\Psi^{-1}(0) / H_c^0(M,g)$.
\end{theorem}
\begin{proof} The proof is similar to the Einstein case considered
above, and we only give a brief outline here.  
If $g$ is anti-self-dual, for $\theta \in S^2_0(T^*M)$, consider the map
\begin{align}
P_g : \Gamma( S^2_0(T^*M)) \rightarrow \Gamma( S^2_0(\Lambda^2_+(T^*M))) 
\times \Gamma(T^*M)
\end{align}
defined by 
\begin{align}
P_g(\theta) = \{ \Pi_g ( W^+(g + \theta)), \delta_g \theta\},
\end{align}
where $\Pi_g$ is projection onto $S^2_0(\Lambda^2_+)$ with respect to $g$. 
The linearized operator at $\theta = 0$ is given by 
\begin{align}
P'_g (h) = (\mathcal{D}_g h, \delta_g h).
\end{align}
This is a mixed-order elliptic operator in the sense of 
Douglis-Nirenberg \cite{DN}. The kernel 
is isomorphic to $H^1_c$, while the cokernel is 
isomorphic to $H^2_c \oplus H^0_c$.  The gauging result in 
Theorem \ref{EP} can be refined to allow conformal deformations, 
which then allows one to gauge to be transverse-traceless, 
see for example \cite[Theorem 2.11]{GVRS}, or \cite{FM}. The result 
then follows from the implicit function theorem in Lemma~\ref{IFT}.
\end{proof}
In general, it is a difficult problem to compute the Kuranishi map, 
but we have the obvious corollary:
\begin{corollary}
If $(M,g)$ is anti-self-dual and satisfies 
\begin{align}
H^0_c(M,g) = \{0\}  \mbox{ and } H^2_c(M,g) = \{0\},
\end{align}
then the moduli space of anti-self-dual 
conformal structures near $g$ is a smooth finite-dimensional 
manifold of dimension $\dim(H^1_c(M,g))$. 
\end{corollary}

\section{Weitzenb\"ock formulas}
For $(M,g)$ Einstein, with $Ric = \lambda \cdot g$, 
define the {\em{Lichnerowicz Laplacian}} by
\begin{align}
\Delta_L h_{ij} = \Delta h_{ij} + 2 R_{ipjq} h^{pq} - 2 \lambda \cdot h_{ij}.
\end{align}
We next have the following Weitzenb\"ock formulae.
\begin{theorem}[Kobayashi \cite{Kobayashi}, Itoh \cite{Itoh}]
\label{wbthm}
If $(M,g)$ is compact and self-dual Einstein with $Ric = \lambda \cdot g$, 
then 
\begin{align}
\label{wb1}
\mathcal{D^*} \mathcal{D}  h &= \frac{1}{2} \Big( \Delta_L + 2 \lambda \Big)
( \Delta_L + \frac{4}{3} \lambda \Big) h  \\
\label{wb2'}
\mathcal{D} \mathcal{D^*}  Z &= \frac{1}{12} ( 3 \Delta - 8 \lambda )
(\Delta - 2 \lambda) Z.
\end{align}
\end{theorem}
\begin{proof} One uses the formula for the adjoint operator
\begin{align}
(\mathcal{D}^* Z)_{ij} = 2 ( \nabla^k \nabla^l + \frac{1}{2}R^{kl}) Z_{ikjl}.
\end{align}
The computations are lengthy, so are left as an exercise. 
\end{proof}
\begin{remark}{\em 
We note that the gradient of $\mathcal{W}$ may also be written as
\begin{align}
\begin{split}
\nabla \mathcal{W} &= -8 ( \nabla^k \nabla^l + \frac{1}{2}R^{kl}) W^+_{ikjl}\\
& = -8 ( \nabla^k \nabla^l + \frac{1}{2}R^{kl}) W^-_{ikjl}.
\end{split}
\end{align}
The details are left as an exercise, see \cite{Itoh}. 
It follows that if $g$ is any anti-self-dual metric, then 
\begin{align}
(\nabla \mathcal{W})'(h) = -4 \mathcal{D^*} \mathcal{D}  h.
\end{align}

}
\end{remark}

\begin{exercise}{\em
Compute $H^i_c(M,g)$ for the following examples:
\begin{itemize}

\vspace{2mm}
\item $S^4$ with the round metric $g_S$. This is locally conformally
flat, so obviously anti-self-dual.

\vspace{2mm}
\item $\CP^2$ with the Fubini-Study metric $g_{\rm{FS}}$.
This is self-dual with respect to the complex 
orientation. To see this, the tensor $W^-$ must be in $S^2_0( \Lambda^{1,1}_0)$ (see \eqref{l2-k} below). 
The isometry group of $g_{\rm{FS}}$ contains ${\rm{SU}}(3)$, and the stabilizer of a 
point contains ${\rm{SU}}(2)$. It is not hard to see that 
${\rm{SU}}(2)$ acts on $\Lambda^{1,1}_0$ as the standard representation 
of ${\rm{SO}}(3) = {\rm{SU}}(2)/ \ZZ_2$. Consequently, the only tensor in  
$S^2_0( \Lambda^{1,1}_0)$ invariant under ${\rm{SU}}(2)$ is the 
zero tensor, so $W^- \equiv 0$. 

\vspace{2mm}
\item $S^1 \times S^3$ with the product metric $g$ (this is 
not Einstein, so you cannot directly use the Weitzenb\"ock formulas. But it
is locally conformally flat). What is the dimension of the moduli space near $g$?

\end{itemize}
}
\end{exercise}
We mention the following conjecture. 
\begin{conjecture}[I.M. Singer] If $(M,g)$ is anti-self-dual and $R > 0$, then
\begin{align}
H^2_c(M,g)=0.
\end{align}
\end{conjecture}
In the Einstein case, this follows easily from the Weitzenb\"ock 
formula \eqref{wb2'}. However, Hitchin proved that the only 
smooth positive ASD Einstein metrics 
are $(S^4, g_S)$ or $(\CP^2, g_{\rm{FS}})$, see \cite[Theorem 13.30]{Besse}.

A wealth of examples of anti-self-dual metrics have been found since this 
conjecture was made, and all of the ones with positive 
scalar curvature have turned out to have $H^2_c(M,g) = 0$. 
But despite all of the evidence, a proof of this conjecture 
remains elusive.


\section{Calabi-Yau metric on K3 surface}
\label{CYK3}

The K3 surface is defined to be a nondegenerate quartic surface in 
$\CP^3$, that is 
\begin{align*}
K3 = \{ [z_0,z_1,z_2,z_3] \in \CP^3 \ | \ z_0^4 + z_1^4 + z_2^4 + z_3^4 =0 \}.
\end{align*}
The topology of K3 is described by: 
$\pi_1(K3) = \{e\}$, $b_2 = 22$, $b_2^+ = 3$, $b_2^- = 19$. 

Since $c_1(K3) = 0$, by Yau's solution of the Calabi conjecture \cite{Yau}, 
K3 admits a  Ricci-flat K\"ahler metric $g_{Y}$.

\begin{proposition} $(K3,g_Y)$ is 
anti-self-dual with respect to the complex orientation.
\end{proposition}
\begin{proof}
To see this we use that fact that for any K\"ahler metric, 
$W^+$ is entirely determined by the scalar curvature. In fact, 
\begin{align}
\label{kid}
\widehat{W}^+ = \frac{R}{12} ( 3 \omega \odot \omega - I),
\end{align}
where $\omega$ is the K\"ahler form. 
To see this, one proves the following: for a K\"ahler surface
\begin{align}
\Lambda^2_+ &= \RR \cdot \omega  \oplus ( \Lambda^{2,0} \oplus \Lambda^{0,2})\\
\label{l2-k}
\Lambda^2_- & =   \Lambda_0^{1,1}.
\end{align}
Next, use the fact that if $(M, g, J)$ is K\"ahler, then 
\begin{align}
Rm(X, Y, Z, W)  &= Rm (JX, JY, Z, W) = Rm (X, Y, JZ, JW),\\
Ric(X, Y) &= Ric(JX, JY). 
\end{align}
This implies that the curvature operator 
\begin{align}
\widehat{R} \in S^2 ( \Lambda^{1,1}). 
\end{align}
Choose an ONB so that the K\"ahler form is given by $\omega_1^+$. 
Since $ \omega^2_+$ and $ \omega^3_+$ are in $\Lambda^{2,0} \oplus \Lambda^{0,2}$, which 
is orthogonal to the space of $(1,1)$-forms, they must be annihilated 
by $\widehat{W}^+ + \frac{R}{12}I$. 
The formula \eqref{kid} then follows since $\widehat{W}^+$ is traceless. 
\end{proof}

\begin{remark}{\em
Note that \eqref{kid} implies that 
 \begin{align}
\label{w+r}
|W^+_g|^2 = \frac{1}{6} R_g^2
\end{align}
for any K\"ahler metric $g$.
}
\end{remark}


\begin{exercise}{\em
Using Theorem \ref{indthm} and the Weitzenb\"ock formulas in Theorem \ref{wbthm}, 
show that:
\begin{itemize}

\vspace{2mm}
\item
$\dim(H_c^0(K3, g_Y)) = 0$.

\vspace{2mm}
\item
$\dim (H^1_c(K3, g_Y)) = 57$.

\vspace{2mm}
\item
$\dim ( H^2_c(K3, g_Y)) = 5$. 

\end{itemize}
}
\end{exercise}
In fact, using the isomorphism $S^2_0(T^*M) = \Lambda^2_+ \otimes \Lambda^2_-$,
it can be shown that $H^1_c(K3, g_Y)$ has a basis
\begin{align*}
\{\omega_I \otimes \omega^-_j, \omega_J \otimes \omega^-_j, 
\omega_K \otimes \omega^-_j\},
\end{align*}
where $\omega_I, \omega_J, \omega_K$ are a basis of the space of 
self-dual harmonic $2$-forms (these are actually K\"ahler forms
for complex structures $I, J, K$),  and
$\{\omega^-_j, j = 1, \dots, 19\}$ is a basis of
the space of anti-self-dual harmonic $2$-forms. 

Furthermore, by the Weitzenb\"ock formula, $H^1_c = H^1_E$ (infinitesimal Einstein 
deformations) and the 
moduli space is exactly $57$-dimensional; the Kuranishi map $\Psi \equiv 0$
by the Bogomolov-Tian Theorem \ref{BTthm}.

\section{Twistor methods}
No discussion of anti-self-dual metrics in dimension four can 
be complete without mentioning twistor theory, which was
first proposed by Penrose in the Lorentzian case. This 
was then studied in depth by Atiyah-Hitchin-Singer in the 
Riemannian case \cite{AHS}.

Given any oriented Riemannian $4$-manifold $(M,g)$, one may 
consider the bundle of unit-norm self-dual $2$-forms:
\begin{align}
S^2  \rightarrow \mathcal{Z}^+ (M) \rightarrow M.
\end{align}
This has a ``tautological'' almost complex structure defined
as follows. At any point in $\mathcal{Z}^+ (M)$, the horizontal 
space carries a tautological complex structure, while the vertical 
space has the complex structure of $S^2$. With the correct 
choice of orientation on the fiber, it was shown in \cite{AHS}
that this almost complex structure is integrable if and only 
if the metric $g$ is anti-self-dual.  Thus to any anti-self-dual 
four-manifold, one can associate a compact complex $3$-fold,
and techniques from complex geometry may be used. 
We only mention that the cohomology groups $H^i_c(M,g) \otimes \CC$ are
isomorphic to $H^i( \mathcal{Z}^+ (M), \Theta)$, 
the Kuranishi cohomology groups corresponding to the deformations of 
complex structure. The Kuranishi map defined above is 
exactly the Kuranishi map for this complex $3$-fold. 

 In \cite{Poon1986} and \cite{Poon1992}, Poon found examples of 
self-dual conformal classes on the connected sums
$\CP^2 \# \CP^2$ and $3 \#\CP^2$
using techniques from algebraic geometry. In \cite{LeBrun1991}, LeBrun gave a 
more explicit construction 
of ${\rm{U}}(1)$-invariant self-dual conformal classes on $n \# \CP^2$
for any $n$, and identified the twistor spaces of these metrics.
Joyce constructed a large class of toric ASD
metrics on  $n \# \CP^2$ in \cite{Joyce1995}, and these were
classified by Fujiki \cite{Fujiki}.
Rollin-Singer constructed scalar-flat K\"ahler metrics on 
$\CP^2 \# 10 \overline{\CP}^2$ in \cite{RollinSinger}. 
Honda has constructed many new examples, see for example \cite{Honda2006, Honda2007}. 
There are now so many known examples that we cannot possibly list all of them
here. 

\section{Gluing theorems for anti-self-dual metrics}

A very successful technique for producing new examples of ASD metrics is 
given by the following. 
\begin{theorem}[Donaldson-Friedman, Floer]
\label{gluthm}
If $(M_1, g_1)$ and $(M_2. g_2)$ are anti-self-dual and $H^2_c(M_i,g_i) =\{ 0 \}$
for $i = 1,2$, then there exist anti-self-dual metrics on the connected sum $M_1 \# M_2$. 
\end{theorem}
Donaldson-Friedman proved Theorem \ref{gluthm} in the case of smooth manifolds
using twistor theory, together with methods 
from the deformation theory of singular complex $3$-folds \cite{DonaldsonFriedman}.
LeBrun-Singer generalized this proof to the case of orbifolds 
with $\mathbb{Z}/2\mathbb{Z}$-orbifold points~\cite{LebSing94}.   
In \cite{Floer}, Floer gave an analytic proof for the case of the connected sum of 
$n$ copies of $\mathbb{CP}^2$. 
The strategy of his proof is to delete points from the summands, and 
conformally change the metrics to become asymptotically cylindrical. 
The metrics are then pasted together by very long cylindrical 
regions in between. An analysis of the indicial roots of the 
linearized problem on the cylinder together with a fixed point theorem 
as in Lemma \ref{IFT}, then allowed Floer to perturb to 
an exact solution. 

  There are also many interesting examples of ASD orbifolds, 
(see for example \cite{CalderbankSingerDuke, LeBrunOptimal, 
LockViaclovsky, Wright2011b} and the references therein), and it
is also an interesting problem to glue together orbifold metrics with complementary 
singularities to produce smooth examples. 
We mention that Kovalev-Singer presented a generalization of Floer's argument 
which works also in the orbifold case \cite{KovalevSinger}, but 
see also \cite{AcheViaclovsky2, RollinSinger, LeBrunMaskit} 
for some clarifications. 

We end this lecture by mentioning Taubes' stable existence claim for
anti-self-dual metrics: 
for any compact, oriented, smooth 4-manifold $M$, the manifold
$M \# n \overline{\mathbb{CP}}^2$
carries an anti-self-dual metric for some $n$, see \cite{Taubes}. 

\lecture{Rigidity and stability for quadratic functionals}
\label{L7}


\section{Strict local minimization}

 We saw that critical points of the Einstein-Hilbert functional in 
general have a saddle-point structure. However, critical points
for certain quadratic functionals have a nicer local variational structure.  
For example, one result we will discuss in this lecture is the following.
Define the functional 
\begin{align}
\mathcal{F}_\tau(g) = \int_{M} |Ric_g|^2 dV_g+ \tau \int_{M} R_g^2 dV_g.
\end{align}
\begin{theorem}[Gursky-Viaclovsky \cite{GVRS}]
\label{snit}
On $S^4$, the round metric $g_S$ is a strict local minimizer (modulo diffeomorphisms and scaling) for the functional $\mathcal{F}_\tau$ provided that
\begin{align}
\label{taurange}
-\frac{1}{3} < \tau < \frac{1}{6}.
\end{align}
\end{theorem}
Many other stability results are given in \cite{GVRS}; various 
results are proved for hyperbolic metrics, complex projective spaces, 
products of spheres, and Ricci-flat metrics. But for simplicity, 
we will only concentrate on the case of the spherical metric in this lecture.  

\subsection{The Jacobi operator}
Just as in the case of the Einstein-Hilbert functional, the second 
variation is orthogonal with respect to the splitting 
\begin{align}
S^2 (T^*M) = \{ f \cdot g\} \oplus \{\mathcal{L}(\alpha) \} \oplus
\{ \delta h = 0 , tr_g(h) = 0\}.
\end{align}
Therefore,  if $h$ is any symmetric $2$-tensor, then it decomposes as
\begin{align}
h = f \cdot g  + \mathcal{L} \alpha + z,
\end{align}
where $z$ is TT. Then 
\begin{align}
\mathcal{F}_{\tau}''(h, h)
= \mathcal{F}_{\tau} ''(f \cdot g, f \cdot g) +  \mathcal{F}_{\tau}''(z, z).
\end{align}
Consequently, 
to check the second variation, we only 
need to consider conformal variations and TT variations separately.

As mentioned above, Einstein metrics are indeed critical for $\mathcal{F}_{\tau}$.
Let us write the second variation at an Einstein metric as
\begin{align}
\mathcal{F}_{\tau} ''(h_1,h_2) = \int_M \langle h_1, J h_2 \rangle dV_g.
\end{align} 
The Jacobi operator $J$ is given explicitly in the following
for TT tensors for any Einstein metric.  
\begin{proposition}[\cite{GVRS}]
\label{gvprop1}
If $g$ is Einstein with $Ric(g) = \lambda \cdot g$ and $h$ is TT, then
the Jacobi operator of ${\mathcal{F}}_{\tau}$ is
\begin{align}
J h &=
\frac{1}{2} \Big( \Delta_L + 2 \lambda \Big)
\Big( \Delta_L + 4 \Big( 1 + 2 \tau \Big) \lambda \Big) h,
\end{align}
where
\begin{align}
\Delta_L h_{ij} = \Delta h_{ij} + 2 R_{ipjq} h^{pq} - 2 \lambda \cdot h_{ij}.
\end{align}
\end{proposition}
\noindent
The proof of this is a long computation, and will not be presented 
here. 

For conformal variations we have the following. 
\begin{proposition}[\cite{GVRS}]
\label{gvprop2}
If $g$ is Einstein with $Ric(g) = \lambda \cdot g$ and $h = f g$, then
\begin{align}
tr_g(J f) &=
\frac{4 +12 \tau}{2} (3 \Delta + 4\lambda )
 \Delta f.
\end{align}
\end{proposition}
Again, the proof is a long computation, and will not be presented here. 

\subsection{The case of the round sphere}
We will now restrict to the case of $(S^4, g_S)$. 
In this case, the Lichnerowicz Laplacian
on TT-tensors is 
\begin{align}
\Delta_L h & =  \Delta h - 8 h.
\end{align}
\begin{proposition}
The least eigenvalue of the rough Laplacian on
TT-tensors is $8$.
\end{proposition}
\begin{proof}
The proof is left as an exercise, with the following hint: use the inequality
\begin{align}
\int_{S^4} | \nabla_{i} h_{jk} + \nabla_j h_{ki} + \nabla_{k} h_{ij}|^2 dV_{g_S} \geq 0.
\end{align}
\end{proof}

Consequently, the least eigenvalue of the Lichnerowicz Laplacian
on TT-tensors is~$16$. Proposition \ref{gvprop1} then implies that if
\begin{align}
\tau < \frac{1}{6},
\end{align}
then the Jacobi operator is positive definite when restricted to
TT-tensors. This results in the upper bound in \eqref{taurange}. 

Proposition \ref{gvprop2} 
implies that the Jacobi operator is non-negative in 
conformal directions for 
\begin{align}
-\frac{1}{3} < \tau,
\end{align}
with the zero eigenvalues given by $h = f \cdot g$, where $f$ is 
a lowest nontrivial eigenfunction of $\Delta$ (by Lichnerowicz' Theorem).
This results in the lower bound in~\eqref{taurange}.


To summarize, we have shown that on $(S^4, g_S)$, 
the second variation is strictly positive on TT-tensors, 
and strictly positive in conformal directions (except for lowest 
nontrivial eigenfunction directions) in the range
\begin{align}
-\frac{1}{3} < \tau < \frac{1}{6}.
\end{align}
We next need to integrate this result to make a conclusion 
about the actual behavior of the functional in a neighborhood
of the spherical metric. To this end, 
using a modification of the Ebin-Palais slicing, we can ignore
Lie derivative directions (as before), and we can also ignore
the conformal zero eigentensors using conformal diffeomorphisms,
so the functional is in fact strictly locally minimized at the 
spherical metric modulo diffeomorphisms. For details, we refer
the reader to \cite[Section~6]{GVRS}.  

Notice that for $\tau = -(1/4)$, the functional 
is $\int |E|^2$, so is obviously strictly minimized 
for this $\tau$, but our improvement of the range of $\tau$ for
minimization has an interesting application, which we will 
discuss next. 


\subsection{A reverse Bishop's inequality}

The classical Bishop's inequality implies that if $(M^4,g)$ is a closed manifold with
$Ric(g) \geq Ric(S^4,g_S) = 3 g$, then the volume satisfies
$Vol(g) \leq Vol(S^4,g_S)$, and equality holds only if $(M,g)$ is
isometric to $(S^4, g_S)$.   An interesting consequence of strict local 
minimization for $\tau =0$ is that, locally, a ``reverse Bishop's inequality'' holds.
\begin{corollary}[\cite{GVRS}] 
\label{crb}
 On $(S^4, g_S)$, there exists a
neighborhood $U$ of $g_S$ in the $C^{2,\alpha}$-norm such that if $\tilde{g} \in U$ with
\begin{align}
 Ric(\tilde{g}) \leq 3 \tilde{g},
\end{align}
then
\begin{align}
Vol(\tilde{g}) \geq Vol(g_S),
\end{align}
with equality if and only if $\tilde{g} = \phi^{*}g_S$ for some
diffeomorphism $\phi : M \rightarrow M$.
\end{corollary}
There remain some very interesting questions:  
\begin{itemize}
\vspace{2mm}
\item What is the largest neighborhood $U$ for which this holds? 

\vspace{2mm}
\item Is the functional $\int_{S^4} |Ric_g|^2 dV_g$ globally minimized at $g_S$?

\end{itemize}

\section{Local description of the moduli space}We will next 
discuss a way to describe local structure of the moduli space of solutions
using a procedure analogous to that for the Einstein equations 
which we discussed above in Lecture \ref{L4}.  
We again consider the functional $\mathcal{F}_{\tau}$, 
and denote the Euler-Lagrange equations by 
\begin{align}
\label{gradzed}
\nabla \mathcal{F}_{\tau} = 0.
\end{align} 
Due to diffeomorphism invariance,
the linearization of (\ref{gradzed})
is not elliptic, so we have to make a gauge choice. 
We will work in transverse-traceless gauge, so define the operator
\begin{align} \label{betadef}
\beta_g h = \delta_g h - \frac{1}{4} d(tr_g h).
\end{align}
Also, since the functional is scale invariant, we will be interested 
in the space
\begin{align}
\label{S20}
\overline{S}_0^2(T^{*}M) = \Big\{ h \in S^2(T^{*}M) \ \big| \ \int_M (tr_g h) dV_g = 0
\Big\}.
\end{align}
Recall that $\mathfrak{K}$ denotes the Lie algebra of Killing vector fields.
\begin{theorem}
\label{Kurcor}
Assume $g$ is critical for $\mathcal{F}_{\tau}$ with $\tau \neq -1/3$. 
Then the space of critical metrics near $g$ modulo diffeomorphism
and scaling is locally isomorphic to 
$\Psi^{-1}(0) / \mathfrak{K}$, where $\Psi$ is a smooth mapping 
\begin{align}
\Psi : H^1_{\tau} \rightarrow H^1_{\tau},
\end{align}
with  
\begin{align*}
H^1_{\tau} = \big\{ h \in C^{\infty}(\overline{S}_{0}^2(T^{*}M))\ \big| \ 
(\nabla \mathcal{F}_{\tau})_g'h = 0,\ \beta_g h = 0 \big\}.
\end{align*}
Consequently, if $g$ is infinitesimally
rigid, then $g$ is rigid.
\end{theorem}
\begin{remark}{\em{
For $\tau = -1/3$, the functional is equivalent to the $L^2$-norm of the 
Weyl tensor, so is conformally invariant. The above result then holds 
if one restricts to the space of pointwise traceless tensors, 
and one considers the moduli 
space of conformal classes near $g$. For details, see~\cite[Section~2.3]{GVRS}. 
}}
\end{remark}
The first step to prove Theorem \ref{Kurcor} 
is to construct a nonlinear mapping whose zeroes 
correspond to the moduli space (locally).
Assume $\mathcal{U} \subset S^2(T^*M)$ is a neighborhood of the zero section, 
sufficiently small so that $\theta \in \mathcal{U}_0 \Rightarrow \tilde{g} = g + \theta$ is a metric. We define the map
\begin{align}
P_g : \mathcal{U} \rightarrow \overline{S}_0^2(T^{*}M),
\end{align}
by
\begin{align} \label{Pdef}
P_g(\theta) = \nabla \mathcal{F} (g + \theta) + \frac{1}{2} \mathcal{K}_{g + \theta}[ \beta_g \mathcal{K}_g \beta_g \theta].
\end{align}
We have the following analogue
of Proposition \ref{zerop}.
\begin{proposition} \label{ellp}
If $\tau \neq -1/3$, then the linearized operator of $P_g$ at $g$ is elliptic. 
Furthermore, if $P_g ( \theta) =0 $, and $\theta \in C^{4,\alpha}$ is sufficiently small 
for some $0 < \alpha < 1$, then $B^t( g + \theta) = 0$ and $\theta \in C^{\infty}$. 
\end{proposition}
\begin{proof}
The proof involves an integration-by-parts argument. It is crucial that the 
equations are variational (since $\mathcal{F}_{\tau}$ is the functional),
so $\delta \nabla \mathcal{F}_{\tau} = 0$. 
This is equivalent to diffeomorphism invariance of $\mathcal{F}_\tau$. 
The proof is similar to that of Proposition \ref{zerop}. 
\end{proof}
We also have an analogue of Proposition~\ref{convprop}. 
\begin{proposition}
If $g_1 = g + \theta_1$ is a critical metric in a sufficiently small $C^{k+1, \alpha}$-neighborhood of $g$ ($k \geq 3$), then there exists a $C^{k+2,\alpha}$-diffeomorphism $\phi : M \rightarrow M$ and a constant $c$ such that
\begin{align} \label{thtildef}
e^c \phi^{*}g_1 = g + \tilde{\theta}
\end{align}
with
\begin{align} \label{Pcon}
P_g(\tilde{\theta}) = 0
\end{align}
and
\begin{align}
\label{trcon}
\int_M tr_g \tilde{\theta}\ dV_g = 0.
\end{align}
\end{proposition}
\begin{proof}
The proof is a gauging argument using an infinitesimal Ebin-Palais 
gauging as done above in the proof of Theorem \ref{EP}, the details are
similar and are omitted. 
\end{proof}
We also require a proposition analogous to Proposition \ref{nonlinprop}, which 
shows that the nonlinear terms are under control. 
\begin{proposition}
\label{nstr}
There exists a constant $C$ such that if we write
\begin{align} \label{PS2}
P_g(h)= P_g(0) + S_gh + Q_g(h),
\end{align}
then for $h_1, h_2 \in C^{4,\alpha}$ of sufficiently small norm,
\begin{align} \label{quadstructure}
\Vert Q_g(h_1) -  Q_g(h_2)\Vert_{C^{\alpha}} \leq
C(  \Vert h_1 \Vert_{C^{4, \alpha}}
+ \Vert h_2 \Vert_{C^{4, \alpha}}) \cdot
\Vert  h_1 - h_2 \Vert_{C^{4, \alpha}}.
\end{align}
\end{proposition}
Since the equation is fourth order, the proof is involved, and 
we refer the reader to \cite[Lemma 2.13]{GVRS} for the details.
We then follow the same procedure as in Lecture \ref{L4} 
to construct the Kuranishi map, using the implicit 
function theorem. The details are left to the reader. 

To finish the proof, 
we note that the gauge term is also carefully chosen so that solutions of the 
linearized equation must be in the transverse-traceless gauge.
That is, if $(P_g)' h = 0$ then 
we have separately, 
\begin{align}
 ( \nabla \mathcal{F}_t)'(h) = 0  \ \mathrm{and} \  \delta \overset{\circ} h = 0,
\end{align}
which is an analogue of Proposition \ref{lgpr}.

\section{Some rigidity results}
For $h$ transverse-traceless, recall from above that the 
linearized operator of $\nabla{F}_{\tau}$ 
at an Einstein metric is given by
\begin{align}
(\nabla{F}_{\tau})' h =\frac{1}{2} \Big( \Delta_L + \frac{1}{2} R\Big)
\Big(\Delta_L + \Big( 1 +  2 \tau \Big) R \Big) h.
\end{align}
\begin{itemize}
\item This formula was previously obtained for the linearized Bach tensor,
which is the case of $\tau = -1/3$ by O. Kobayashi \cite{Kobayashi}.
\item Recall that infinitesimal Einstein deformations are given by 
TT kernel of  the operator $\Delta_L + \frac{1}{2} R$, which we
studied in Lecture \ref{L4}. These deformations are
still present, but note there is now the possibility of non-Einstein deformations.
\end{itemize}


Also, recall that for $h = f g$, we have 
\begin{align}
tr_g ( (\nabla{F}_{\tau})' h) = 2(1 + 3 \tau)  ( 3 \Delta + R) (\Delta f).
\end{align}
The rigidity question is then reduced to a separate analysis of the 
eigenvalues of $\Delta_L$ on transverse-traceless tensors, 
and of $\Delta$ on functions. 
Such an analysis of the eigenvalues of these operators yields the following rigidity 
theorems. Let $H^1_{\tau}$ denote the space of transverse-traceless kernel 
of the linearized operator. In the case of the Fubini-Study metric, we have
\begin{theorem}[\cite{GVRS}]
\label{cp2rig}
On $(\CP^2, g_{\rm{FS}})$, $H^1_{\tau} = \{0\}$ provided that $\tau < 1/6$. 
\end{theorem}
\noindent
In the case of the product metric, we have
\begin{theorem}[\cite{GVRS}] 
\label{s2s2rig}
On $(S^2 \times S^2, g_{S^2 \times S^2})$,
$H^1_{\tau} = \{0\}$ provided that $\tau < 0$ and $\tau \neq - 1/2$.
If $\tau = - 1/2$, then $H^1_\tau$ is one-dimensional and spanned by 
the element $g_1 - g_2$. 
\end{theorem}
In particular, in the case $\tau = -1/2$, 
this gives an example of a deformation of critical metrics for this
functional which is not an Einstein deformation. 

Many other rigidity results are presented in \cite{GVRS}, which
we will not go into detail here. We have stated the above two results because 
these rigidity results will play a crucial r\^ole in the
gluing construction which will be discussed in Lecture \ref{L9}. 
\section{Other dimensions}
We restricted the above discussion to dimension $4$ for simplicity, 
since quadratic functionals are scale-invariant in that dimension. 
In dimensions other than four, if $\mathcal{F}$ denotes a quadratic
curvature functional, then the volume-normalized functional
\begin{align} \label{TFtdef}
\tilde{\mathcal{F}}[g] = Vol(g)^{\frac{4}{n}-1} \mathcal{F}[g].
\end{align}
is scale-invariant. Many of the results stated above also 
have analogues in other dimensions. For example, Corollary 
\ref{crb} has an analogue in higher dimensions:
\begin{theorem}[\cite{GVRS}] \label{RevBishop}  
Let $(M,g)$ be a sphere, space form, or complex projective space,
normalized so that $Ric(g) = (n-1)g$. 
Then there exists a $C^{2,\alpha}$-neighborhood $U$ of $g$ 
such that if $\tilde{g} \in U$ with
$ Ric(\tilde{g}) \leq (n-1) \tilde{g}$,
then $Vol(\tilde{g}) \geq Vol(g)$
with equality if and only if $\tilde{g} = \phi^{*}g$ for some
diffeomorphism $\phi : M \rightarrow M$.
\end{theorem}
For the proof, and for other examples of rigidity and
stability of Einstein metrics for quadratic curvature
functionals, we refer the reader to \cite{GVRS}. 

 We mention that quadratic functionals in dimension $3$ were considered
by Anderson in \cite{AndersonI, AndersonII}. Also, critical 
points of the functional $\tilde{\mathcal{F}}_{\sigma_2}$ 
(defined above in \eqref{fsigma2}) were studied in dimension $3$ in \cite{GV2001}.
For rigidity results involving other types of functionals
see \cite{Moller}. 
 
\lecture{ALE metrics and orbifold limits}
\label{L8}
\section{Ricci-flat ALE metrics}
In order to understand limits of Einstein metrics, one
one first have some understanding of a class of complete non-compact metrics,
which are defined as follows.
\begin{definition}
\label{ALEdef}
{\em
 A complete Riemannian manifold $(X^4,g)$ 
is called {\em{asymptotically locally 
Euclidean}} or {\em{ALE}} of order $\tau$ if 
there exists a finite subgroup 
$\Gamma \subset {\rm{SO}}(4)$ 
acting freely on $S^3$ and a 
diffeomorphism 
$\psi : X \setminus K \rightarrow ( \mathbb{R}^4 \setminus B(0,R)) / \Gamma$ 
where $K$ is a compact subset of $X$, and such that under this identification, 
\begin{align}
\label{eqgdfale1in'}
(\psi_* g)_{ij} &= \delta_{ij} + O( \rho^{-\tau}),\\
\label{eqgdfale2in'}
\ \partial^{|k|} (\psi_*g)_{ij} &= O(\rho^{-\tau - k }),
\end{align}
for any partial derivative of order $k$, as
$r \rightarrow \infty$, where $\rho$ is the distance to some fixed basepoint.  
}
\end{definition}
Note that this definition is really 
just the same as Definition \ref{AFdef}, with the addition of a
group at infinity:
\subsection{Eguchi-Hanson metric}

We next recall the Eguchi-Hanson metric, which was the first example of a 
non-trivial Ricci-flat ALE space \cite{EguchiHanson}. It is given by
\begin{align}
g_{\rm{EH}} = \frac{ dr^2}{ 1 - r^{-4}} +r^2 \Big[ \sigma_1^2 + \sigma_2^2
+ (  1 - r^{-4}) \sigma_3^2 \Big],
\end{align}
where $r$ is a radial coordinate, and $\{ \sigma_1, \sigma_2, \sigma_3 \}$ is a 
left-invariant coframe on $S^3$ (viewed as the Lie group ${\rm{SU}}(2)$).
This has an apparent singularity at $r =1$, so 
redefine the radial coordinate to be $\hat{r}^2 = r^2 - 1$, 
and attach a $\CP^1$ at $\hat{r} = 0$. After taking a quotient by $\ZZ_{2}$, 
the metric then extends smoothly over the added $\CP^1$, is Ricci-flat,
ALE at infinity of order $4$, and is diffeomorphic to $\mathcal{O}(-2) \cong T^* S^2$. 
This space is hyperk\"ahler, that is, there are three independent complex structures
denoted by $I, J,$ and $K$ satisfying the quaternion relations
\begin{align}
I^2 = J^2 = K^2 = IJK = - 1.
\end{align}
Denote the corresponding K\"ahler forms
by $\omega_I, \omega_J, \omega_K$, which are 3 linearly independent 
harmonic self-dual $2$-forms. 

Since this space is non-compact, let us look at
$H^1_{E,-}(X,g)$, which we define to be decaying infinitesimal 
Einstein deformations, that is, those $h \in S^2(T^*X)$ satisfying all the conditions in 
\eqref{H1Edef}, and also which satisfy $h = O(r^{-\epsilon})$ for some $\epsilon >0$. 
Using the isomorphism $S^2_0(T^*M) = \Lambda^2_+ \otimes \Lambda^2_-$,
it follows from the discussions in Lecture \ref{L6}  that $H^1_{E,-}(X,g)$ has a basis 
\begin{align}
\{\omega_I \otimes \omega^-, \omega_J \otimes \omega^-, \omega_K \otimes \omega^-\},
\end{align}
where $\omega^-$ is a non-trivial $L^2$ harmonic $2$-form, see \cite{Page2}. 
Consequently, 
\begin{align}
\dim(H^1_{E,-}(X,g)) = 3. 
\end{align}
However, the Eguchi-Hanson metric is known to be rigid (up to scaling) 
as an Einstein ALE metric. This means that these infinitesimal 
deformations do not integrate to non-trivial Einstein deformations. 
It turns out that these elements can in fact be written as Lie derivatives,
and can be understood as gluing parameters for a certain 
gluing problem which we will discuss in more detail below.

\subsection{Hyperk\"ahler ALE metrics}

After the Eguchi-Hanson metric was found, 
Gibbons-Hawking  wrote down a metric ansatz depending on the choice of 
$n$ monopole points in $\RR^3$, giving an anti-self-dual ALE hyperk\"ahler metric 
with group $\ZZ / n \ZZ$ at infinity, which 
are called multi-Eguchi-Hanson metrics \cite{GibbonsHawking, Hitchin2}. 
In 1989, Kronheimer then classified all hyperk\"ahler ALE spaces
in dimension $4$, \cite{Kronheimer, Kronheimer2}.
To describe these,  we 
consider the following subgroups of ${\rm{SU}}(2)$:
\begin{itemize}
\item
Type $A_n, n \geq 1$: $\Gamma$ the cyclic group $\mathbb{Z}_{n+1}$,
\begin{align}
\label{u2c'}
(z_1, z_2) \mapsto (e^{2 \pi i p / (n+1)} z_1, e^{-2 \pi i p / (n+1) }z_2),
\ \ 0 \leq p \leq n.
\end{align}
acting on $\RR^4$, which is identified with $\CC^2$
via the map  
\begin{align}
(x_1, y_1, x_2, y_2) \mapsto (x_1 + i y_1, x_2 + i y_2) = (z_1, z_2).
\end{align}
Writing a quaternion $q \in \mathbb{H}$ as $\alpha + j \beta$ for 
$\alpha, \beta \in \CC$, we can also describe the action as 
generated by $e^{2\pi i / (n+1)}$, acting on the left. 
\item
Type $D_{n}, n \geq 3$: $\Gamma$ the binary dihedral group $\mathbb{D}^*_{n-2}$
of order $4(n-2)$. This is generated by $e^{\pi i / (n-2)}$ and $j$, 
both acting on the left. 
\item
Type $E_6: \Gamma= \mathbb{T}^*$, 
the binary tetrahedral group of order $24$, double cover of~$A(4)$. 
\item
Type $E_7: \Gamma= \mathbb{O}^*$, 
the binary octohedral group of order $48$, double cover of~$S(4)$.
\item
Type $E_8: \Gamma = \mathbb{I}^*$, 
the binary icosahedral group of order $120$, double cover of~$A(5)$. 
\end{itemize}

More specifically, Kronheimer showed that for any of these groups $\Gamma$, 
there do in fact exist simply-connected hyperk\"ahler ALE spaces with 
group $\Gamma$ at infinity, and moreover he completely 
classified these as hyperk\"ahler quotients. 

We will not go into details about Kronheimer's construction, 
but just briefly discuss the identification of the space of decaying infinitesimal 
Einstein deformations, given by the argument in \cite[Proposition~1.1]{Biquard}. 
The operator $\Delta_L$ acting on transverse-traceless
tensors can be identified with the 
operator $d_- d_-^*$ where 
\begin{align}
d_- : \Omega^1 \otimes \Omega^2_+
\rightarrow  \Omega^2_- \otimes \Omega^2_+ \cong \Gamma(S^2_0(T^*X))
\end{align}
is the exterior derivative. Since $\Omega^2_+$ has a basis of 
parallel sections $\{\omega_I, \omega_J, \omega_K\}$, the proposition 
follows since the $L^2$-cohomology $H^2_{(2)}(X)$ is isomorphic to the 
usual cohomology $H^2(X)$ \cite{CarronDuke}. 
Consequently,  the formal dimension of the moduli space of any such metric is $3
b_2^-$. Note by the Weitzenb\"ock formulas given in Lecture~\ref{L6}, these 
deformations are equivalent to decaying infinitesimal anti-self-dual deformations, 
see~\cite{ViaclovskyIndex}. Kronheimer has shown that the 
actual dimension of the moduli space of hyperk\"ahler metrics 
on these spaces is $ 3 b_2^- - 3$, which implies that 
there are infinitesimal deformations which are not integrable. 
As in the case of the Eguchi-Hanson metric, these can also 
be understood as gluing parameters.
The properties of these spaces are summarized in Table~\ref{hktp}.
\begin{table}[h]
\caption{Invariants of hyperk\"ahler ALE spaces.}
\label{hktp}
\begin{tabular}{l l l l l}
\hline
Type & $\Gamma$ & $|\Gamma|$& $ b_2^-$ & $\chi$ \\
\hline
$A_n, n \geq 1$ & $\mathbb{Z}_{n+1}$ & $n+1$ & $n$ & $n+1$  \\
$D_m, m \geq 3$ & $\mathbb{D}^*_{m-2}$ & $4(m-2)$ & $m$ & $m+1$ \\
$E_6$ & $ \mathbb{T}^*$ & $24$ &  $6$ & $7$ \\
$E_7$ &  $\mathbb{O}^*$  & $48$ & $7$ & $8$ \\
$E_8$ & $\mathbb{I}^*$  & $120$ & $8$ & $9$ \\
\hline
\end{tabular}
\end{table}

 To close this brief discussion of hyperk\"ahler metrics, 
we note here the following interesting conjecture due to Bando, Kasue, and Nakajima \cite{BKN}:
\begin{conjecture}
If $(M,g)$ is a simply-connected Ricci-flat ALE space in dimension four, 
then $g$ is hyperk\"ahler. 
\end{conjecture} 
\begin{remark}{\em
There are other interesting complete Einstein 
metrics with different asymptotics at infinity, known as
``ALF'', ``ALG'', and ``ALH'' gravitational instantons. 
We will not have time to discuss
these, and refer the reader to \cite{BiquardMinerbe, Minerbe} for a nice discussion of
these types of metrics. There are also many interesting examples of
gravitational instantons which have non-integral volume growth
exponent~\cite{Hein}.
}
\end{remark}

\section{Non-collapsed limits of Einstein metrics}
The results in Lecture \ref{L4} give a local description of the moduli space near
a fixed Einstein metric. One would also like to understand global 
properties of the moduli space of Einstein metrics, 
for example, what are the possible limits of sequences of Einstein 
metrics?

\begin{definition}
\label{orbdef}
{\em
 A {\em{Riemannian 
orbifold}} $(M^4,g)$ is a topological space which is a 
smooth manifold of dimension $4$ with a smooth Riemannian metric 
away from finitely many singular points.  
At a singular point $p$, $M$ is locally diffeomorphic 
to a cone $\mathcal{C}$ on 
$S^{3} / \Gamma$, where $\Gamma \subset {\rm{SO}}(4)$ 
is a finite subgroup acting freely on $S^{3}$.
Furthermore, at such a singular point, the metric is locally the 
quotient of a smooth $\Gamma$-invariant metric on $B^4$ under 
the orbifold group $\Gamma$.}
\end{definition}

 In the non-collapsing case, the following is known, due
to Anderson, Bando-Kasue-Nakajima, and Tian. 
\begin{theorem}[\cite{Anderson}, \cite{BKN}, \cite{TianCalabi}]
\label{einconv}
Let $(M_i, g_i)$ sequence of Einstein manifolds of dimension $4$ satisfying
\begin{align}
\int_{M_i} |Rm_{g_i}|^{2}dV_{g_i}  < \Lambda, \ \mathrm{diam}(g_i) < D, \ Vol(g_i) > V > 0.
\end{align}
Then for a subsequence $\{j\} \subset \{i\}$,
\begin{align}
(M_j, g_j)  \xrightarrow{\mathrm{Cheeger-Gromov}} (M_{\infty}, g_{\infty}),
\end{align}
where $(M_{\infty}, g_{\infty})$ is an orbifold with finitely many 
singular points. 
\end{theorem}
The above convergence is in the Cheeger-Gromov sense which means that 
the metrics
converge in the Gromov-Hausdorff sense to the limit space as a 
metric space, but away from the singular points, the convergence
is moreover smooth (after pulling back by diffeomorphisms). 
Rescaling such a sequence to have bounded curvature near a singular 
point and taking a pointed limit
yields Ricci-flat ALE spaces, also called ``bubbles''. 
This bubbling description can be refined; what we just described
produces a ``deepest bubble''. Choosing different scalings 
can yield a tree of ALE Ricci-flat orbifolds at a singular point, 
see \cite{Nakajima2} for a nice description of this process. 

\subsection{K3 example}
Eguchi-Hanson metrics arise as bubbles for certain sequences of 
Calabi-Yau metrics on K3, and was suggested by 
\cite{KobayashiTodorov}:
\begin{example}
There exists a sequence of Ricci-flat metrics $g_i$ on $K3$ satisfying
\begin{align}
  (K3, g_i)  \longrightarrow (T^4 / \{\pm 1\}, g_{\mathrm{flat}}).
\end{align}
At each of the 16 singular points, an Eguchi-Hanson metric on $T^* S^2$ 
``bubbles off''. 
\end{example}
Note that since each Eguchi-Hanson metric has $3$ infinitesimal 
Einstein deformations, and the torus has $10$ flat deformations, 
modulo scaling the parameter count is $57$, which is in 
nice agreement with the count made in Section \ref{CYK3} of
Lecture~\ref{L6}. 
See \cite{Page2} for a nice heuristic description of this.
Also see \cite{LebSing94} for a rigorous construction of a Calabi-Yau 
metric by anti-self-dual gluing methods, and the note of 
Donaldson  \cite{Donaldsonmodel} for a rigorous construction 
using K\"ahler-Einstein techniques.

\begin{remark}{\em 
There is another very interesting limit of
Calabi-Yau metrics on K3, known as the ``large complex 
structure  limit'' \cite{GrossWilson}.  
A sequence of these metrics collapses to a limiting $2$-sphere, 
so the limit is not described by Theorem
~\ref{einconv}. Away from 24 points, the sequence collapses with bounded
curvature, and gives a nice illustration of the $\epsilon$-regularity 
theorem of Cheeger-Tian \cite{CheegerTianJAMS}.}
\end{remark}


\subsection{Desingularization}
We next ask the follow question:
\begin{itemize}
\item Can you reverse this process? That is, can you start with an Einstein orbifold, ``glue on'' hyperk\"ahler bubbles at the singular points, and resolve to a smooth Einstein metric?
\end{itemize}
The ``answer'' is:
\begin{itemize}
\item
In general, the answer is ``no'', since this gluing problem is obstructed; 
there are always decaying infinitesimal Einstein deformations 
of non-trivial Ricci-flat ALE spaces. 
\end{itemize}

In the ASD case, the relevant operator maps between different bundles, and
the index is not necessarily zero. However, the index
is always zero in the Einstein case, so this makes the problem 
much more difficult. We remark that sometimes, Einstein metrics 
{\em{can}} be found by gluing techniques, but only when using 
some extra structure. As mentioned above, Calabi-Yau metrics on K3 can 
be produced using ASD gluing techniques \cite{LebSing94}, or K\"ahler-Einstein 
techniques  \cite{Donaldsonmodel}. For the 
$G_2$ and $Spin(7)$ cases, see \cite{Joyce, Joyce2}.



\subsection{Biquard's Theorem}
We next discuss a beautiful result which says that the answer
to the above question is ``yes'' in a certain case. 
The setting is a class of complete non-compact
Einstein metrics with negative Einstein constant. 
If $M^4$ is a compact manifold with boundary $\partial M$, then a metric 
$g$ on $M^4$ is said to be {\em{conformally compact}} if $\tilde{g} = \rho^2 g$
has some regularity (e.g., H\"older) up to the boundary, where $\rho$ is a 
defining function for the boundary which satisfies $\rho^{-1}(0) = \partial M$ 
and $d \rho \neq 0$ on $\partial M$. If $|dp|_{\tilde{g}} = 1$, then $g$ limits
to a hyperbolic metric as $\rho \rightarrow 0$, such a metric
is called {\em{asymptotically hyperbolic}}. If it is in addition Einstein  
(necessarily with negative Einstein constant),
then $(M,g)$ is called {\em{asymptotically hyperbolic Einstein}}, or AHE for short. 
We note that there is an 
induced conformal class on the boundary manifold at infinity. This definition 
should be 
thought of as a generalization of the hyperbolic ball, with the conformal class
of the round sphere at infinity. Biquard's result is the following.
\begin{theorem}[Biquard \cite{Biquard}] 
\label{Biqthm}
Let $(M^4,g)$ be an AHE metric with a 
$\ZZ/2 \ZZ$ orbifold singularity at $p \in M$. 
If $(M^4,g)$ is rigid (i.e., $g$ admits no nontrivial infinitesimal 
Einstein deformations), 
then the singularity can be resolved to a AHE metric 
by gluing on an Eguchi-Hanson metric if and only if 
\begin{align}
\det (\mathcal{R}^+)(p)  = 0,
\end{align}
where $\mathcal{R}^+$ is the upper-left $3 \times 3$ block in
\eqref{d4c}. 
\end{theorem}
As we discussed above, the Eguchi-Hanson metric admits
a $3$-dimensional space of infinitesimal Einstein deformations,
so this gluing problem is obstructed. One of these deformations 
corresponds to a scaling parameter, and the other two correspond 
to rotations in ${\rm{SO}}(4)/ {\rm{U}}(2)$. 
Biquard is able to overcome these obstructions using the freedom to
perturb the boundary conformal class of the AH Einstein metric. 

 Recently, Biqard has given a generalization of Theorem \ref{Biqthm} to 
allow orbifolds with more general ADE-type singularities,
see \cite{BiquardII} for the precise statement of this extension.

\section{$B^t$-flat metrics}
We next return to critical metrics of quadratic curvature 
functionals. We will be interested in the functional
\begin{align*} 
\mathcal{B}_{t}[g] = \int_M |W_g|^2\ dV_g + t \int_M R_g^2\ dV_g.
\end{align*}
\begin{remark}{\em
From the Chern-Gauss-Bonnet Theorem \ref{CGB}, this is really 
the most general quadratic functional in dimension $4$, up
to scaling. We have chosen to normalize this way to take
advantage of the conformal invariance of $\mathcal{W}$.
}
\end{remark}
The Euler-Lagrange equations of $\mathcal{B}_{t}$ are given by
\begin{align}
B^t \equiv B + t C = 0,
\end{align}
where $B$ is the {\em{Bach tensor}} defined by \eqref{gradW}
\begin{align}
B_{ij} \equiv -4 \Big( \nabla^k \nabla^l W_{ikjl} + \frac{1}{2} R^{kl}W_{ikjl} \Big),
\end{align}
and $C$ is the tensor defined by \eqref{s'}
\begin{align} 
C_{ij} =  2 \nabla_i \nabla_j R - 2 (\Delta R) g_{ij}
- 2 R R_{ij} +   \frac{1}{2} R^2 g_{ij}.
\end{align}
From conformal invariance of the functional $\mathcal{W}$, 
it follows that the Bach-tensor is conformally invariant. 
We will refer to such a critical metric as a {\em{$B^t$-flat metric}}.
 Note that any Einstein metric
is critical for~$\mathcal{B}_{t}$, but there are in fact
non-Einstein $B^t$-flat metrics, as we shall see.

For $t \neq 0$, by taking a trace of the E-L equations, it follows that
\begin{align}
\Delta R = 0.
\end{align}
If $M$ is compact, this implies $R = constant$. 
Consequently, the  $B^t$-flat condition is equivalent to
\begin{align}
B = 2 t R \cdot E,
\end{align}
where $E$ denotes the traceless Ricci tensor. 
That is, the Bach tensor is a constant multiple of the traceless Ricci tensor, 
which is indeed a natural generalization of the Einstein condition. 
\subsection{$B^t$-flat ALE metrics}

Of course, all of the hyperk\"ahler ALE spaces described above
are also $B^t$-flat, but there are many more non-Einstein examples. 
 
A large source of examples is the following.
If $(M,g)$ is Bach-flat and has positive scalar 
curvature, then we can convert $(M, g)$
into an asymptotically flat (AF) metric 
\begin{align*}
(N, g_N) = (M \setminus \{p\}, G^2 g)
\end{align*}
using the Green's function for
the conformal Laplacian $G$ based at $p$.
Since $(M,g)$ is Bach-flat, then $(N,g_N)$ is also Bach-flat (from
conformal invariance) and scalar-flat (since we used the Green's function).
Consequently, $g_N$ is $B^t$-flat for all $t \in \RR$.
This gives a large family 
of examples of non-trivial asymptotically flat $B^t$-flat
metrics, in contrast to the Ricci-flat case.  

By taking the sum of Green's functions based at several 
points, one can also obtain many examples with several 
ends. In the special case of the sphere, with two points, 
one obtains the Euclidean Schwarzschild metric
\begin{align}
g = \left( 1 + \frac{m}{r^2} \right)^2 g_0,
\end{align} 
where $g_0$ is the Euclidean metric. This metric 
plays a very important r\^ole in the Riemannian Penrose
Inequality, see for example
\cite{BrayPenrose, BrayLee}.

Another family of non-trivial examples is the following. 
In \cite{LeBrunnegative}, LeBrun presented the first known examples of scalar-flat ALE 
spaces of negative mass, which gave counterexamples to extending the
positive mass theorem to ALE spaces. We briefly describe these as follows. Define
\begin{align}
\label{glbdef}
g_{\rm{LB}} = \frac{ dr^2}{ 1 + Ar^{-2} + B r^{-4}} +r^2 \Big[ \sigma_1^2 + \sigma_2^2
+ (  1 + Ar^{-2} + B r^{-4}) \sigma_3^2 \Big],
\end{align}
where $r$ is a radial coordinate, $\{ \sigma_1, \sigma_2, \sigma_3 \}$ is a 
left-invariant coframe on $S^3$, and $A = n -2$, $B = 1 - n$. 
There is an apparent singularity at $r = 1$, so 
redefine the radial coordinate to be $\hat{r}^2 = r^2 - 1$, 
and attach a $\CP^1$ at $\hat{r} = 0$. After taking a quotient by $\ZZ_{n}$, 
with action given by the diagonal action
\begin{align}
\label{u2c}
(z_1, z_2) \mapsto e^{2 \pi i p / n}(z_1, z_2),  \ \ 0 \leq p \leq n -1,
\end{align}
the metric then extends smoothly over the added $\CP^1$, is ALE at infinity,
and is diffeomorphic to $\mathcal{O}(-n)$.  
The mass (as defined in \eqref{massdef}) is computed to be $-2 (n -2)$, which is 
negative when $n > 2$. 
These metrics are scalar-flat K\"ahler (so are anti-self-dual, 
and thus Bach-flat), and satisfy $b^2_- = 1, \tau = -1,$ and $\chi = 2$.

Finally, we mention that 
Calderbank and Singer produced many examples of toric 
ALE anti-self-dual metrics, which are moreover 
scalar-flat K\"ahler, and have cyclic groups at 
infinity contained in ${\rm{U}}(2)$ \cite{CalderbankSinger}.


\section{Non-collapsed limits of $B^t$-flat metrics}
For $t \neq 0$, the $B^t$-flat equation can be rewritten as 
\begin{align}
\label{criteqn}
\Delta Ric = Rm * Ric.
\end{align}
If $t = 0$, equation \eqref{criteqn} is satisfied provided 
one assumes in addition that $g$ has constant scalar curvature. 
With slightly different geometric assumptions, 
a similar orbifold-compactness theorem as in the Einstein case
holds for sequences of metrics satisfying \eqref{criteqn}:
\begin{theorem}[Tian-Viaclovsky \cite{TV, TV2, TV3}]
\label{tvt1}
Let $(M_i, g_i)$ be a sequence of $4$-dimensional manifolds satisfying \eqref{criteqn}
and
\begin{align}
\int_{M_i} |Rm_{g_i}|^{2} dV_{g_i}  < \Lambda, \ Vol(B(q,s)) > V s^4 > 0, \ b_1(M_i) < B,
\end{align}
for all  $s \leq diam(M)/2$. Then for a subsequence $\{j\} \subset \{i\}$,
\begin{align}
(M_j, g_j)  \xrightarrow{\mathrm{Cheeger-Gromov}} (M_{\infty}, g_{\infty}),
\end{align}
where $(M_{\infty}, g_{\infty})$ is a multi-fold satisfying \eqref{criteqn},
with finitely many singular points. 
\end{theorem}
Similar to the Einstein case, 
rescaling such a sequence to have bounded curvature near a singular 
point yields ALE metrics satisfying \eqref{criteqn}.  
An important difference with the Einstein case is that the ALE spaces
can have multiple ends (this is ruled out in the Einstein 
case by the Cheeger-Gromoll splitting theorem). Thus singular 
points are more general in that multiple orbifold cones can 
touch at a singular point, thus the terminology ``multi-fold''. 
Another difference is that a smooth point of the limit can 
in fact be a singular point of convergence, this cannot happen in the 
Einstein case (by Bishop's volume comparison theorem). 

 The key point in this theorem is the following related
volume growth theorem:
\begin{theorem}[Tian-Viaclovsky \cite{TV3}] 
\label{orbthm2}
Let $(M, g)$ be a metric satisfying \eqref{criteqn} on a 
smooth, complete four-dimensional manifold $M$ with
\begin{align} 
\label{l2a}
\int_{M} |{\rm{Rm}}_{g}|^2 dV_g \leq \Lambda,
\end{align}
for some constant $\Lambda$. 

Assume that 
\begin{align} 
{\rm{Vol}}( B(q,s)) &\geq V_0 s^4, \mbox{ for all } q \in M,  
\mbox{ and } s \leq diam(M)/2, \\
b_1(M) &< B_1,
\end{align}
where $V_0, B_1$ are constants.
Then there exists a constant 
$V_1$, depending only upon $V_0, \Lambda, B_1$, 
such that ${\rm{Vol}} (B(p,r)) \leq V_1 \cdot r^4$, 
for all $p \in M$ and $r > 0$.  
\end{theorem}

Once one has this volume growth estimate, the proof of 
Theorem \ref{tvt1} is fairly similar to the Einstein 
case, see \cite{TV2}. The key part of the proof of Theorem~\ref{orbthm2} is
the volume growth theorem in \cite{TV}, which depends on 
a Sobolev constant bound. Subsequently, using a point-picking 
argument, it was shown in \cite{TV3} that the Sobolev constant 
bound can be replaced with a lower volume growth bound, which is
the version we stated here. But we note that if a 
Sobolev constant bound is assumed, then 
the assumption on $b_1$ is not necessary \cite{TV3}.

\subsection{Desingularization questions}
It is natural to ask the same question that we asked in the Einstein case:
\begin{itemize}
\vspace{2mm}
\item
Can you reverse this process?
That is,  start with an critical orbifold, ``glue on'' critical ALE metrics at
the singular points, and resolve to a smooth critical metric?
\end{itemize}
The ``answer'' is still:
\begin{itemize}
\vspace{2mm}
\item
In general, the answer is ``no'', because this gluing problem is obstructed;
there are always decaying infinitesimal deformations of non-trivial $B^t$-flat 
ALE metrics. 
\end{itemize}
However, in Lecture \ref{L10} we will discuss a recent theorem which 
says that the answer is ``yes'' in certain cases.

\lecture{Regularity and volume growth}
\label{L9}
In this Lecture, we will discuss some of the main ideas
involved in the proofs of Theorems \ref{tvt1} and \ref{orbthm2}.
Also, we will give an outline of the proof of the existence of an 
Einstein metric on $\CP^2 \# 2 \overline{\CP}^2$ due to
Chen-LeBrun-Weber \cite{CLW}. 
\section{Local regularity}
We consider any system of the type
\begin{align}
\label{generaleqn}
\Delta Ric = Rm * Ric.
\end{align}
Any Riemannian metric satisfies 
\begin{align}
\label{curveqn}
\Delta Rm = L(\nabla^2 Ric) + Rm*Rm,
\end{align}
where $L(\nabla^2 Ric)$ denotes a linear expression 
in second derivatives of the Ricci tensor, and 
$Rm*Rm$ denotes a term which is quadratic in 
the curvature tensor (see \cite[Lemma 7.2]{Hamilton}). 

 For a compact $4$-manifold $(M,g)$, we define the Sobolev constant $C_S$  
as the best constant $C_S$ so that 
for all $f \in C^{0,1}(M)$ (Lipschitz) we have
\begin{align}
\label{scc}
\Vert f \Vert_{L^{4}} \leq C_S  
\Vert \nabla f \Vert_{L^2} +  Vol^{-1/4} \Vert f \Vert_{L^2}.
\end{align}

  If $(X,g)$ is a complete, noncompact $4$-manifold, the Sobolev constant $C_S$ is 
defined as the best constant $C_S$ so that 
for all $f \in C^{0,1}_c(X)$ (Lipschitz with compact support),  we have
\begin{align}
\label{siq2}
\Vert f \Vert_{L^{4}} \leq C_S \Vert \nabla f \Vert_{L^2}. 
\end{align}

The following local regularity theorem is known as an ``$\epsilon$-regularity'' theorem:
\begin{theorem}[Tian-Viaclovsky \cite{TV}]
\label{higherlocalregthm}
Assume that (\ref{generaleqn}) is satisfied, 
let $r < diam(X)/2$, and $B(p,r)$ be a geodesic 
ball around the point $p$, and $k \geq 0$. Then there exist  
constants $\epsilon_0, C_k$  (depending upon $C_S$) so that if 
\begin{align*}
\Vert Rm \Vert_{L^2(B(p,r))} = 
\left\{ \int_{B(p,r)} |Rm|^2 dV_g \right\}^{1/2} \leq \epsilon_0,
\end{align*}
then 
\begin{align*}
\underset{B(p, r/2)}{sup}| \nabla^k Rm| \leq
\frac{C_k}{r^{2+k}} \left\{ \int_{B(p,r)} |Rm|^2 dV_g \right\}^{1/2}
\leq \frac{C_k \epsilon_0}{r^{2+k}}. 
\end{align*}
\end{theorem}
In the case of harmonic curvature, $\delta Rm = 0$, one has an 
equation on the full curvature tensor 
\begin{align}
\label{harmcurveqn}
\Delta Rm = Rm* Rm.
\end{align}
In this case, the result follow by a Moser iteration procedure, for details 
we refer to \cite{Akutagawa, Anderson, Nakajima2, TianCalabi}. This is a generalization 
of an $\epsilon$-regularity theorem of Uhlenbeck \cite{Uhlenbeck1, Uhlenbeck2}.
We also note that this theorem was extended to extremal K\"ahler metrics 
by Chen-Weber \cite{CW}.
Also, for Einstein metrics, dependence on the 
Sobolev constant was removed in \cite{CheegerTianJAMS}. 

Even though second derivatives of the 
Ricci tensor occur in (\ref{curveqn}), overall the principal
symbol of the system (\ref{generaleqn}) and (\ref{curveqn})
is in triangular form. The equations (\ref{generaleqn}) and (\ref{curveqn}), 
when viewed as an elliptic system, together with the bound on the 
Sobolev constant, are the key to the proof of theorem \ref{higherlocalregthm},
which is an involved iteration procedure, and we will refer
the reader to \cite{TV} for details.  

In relation to Theorem \ref{tvt1}, this $\epsilon$-regularity 
result is the key to proving that there are only finitely 
many points at which the curvature can blow-up. This is 
because each such point must account for at least
$\epsilon_0$ of the $L^2$-norm of curvature, which is assumed
to be bounded for the sequence. 
Consequently, at strictly positive distance away from the singular points, the 
curvature is bounded, and a subsequential limiting space with finitely 
many singular points can be obtained 
using fundamental ideas of Cheeger \cite{Cheeger} and Gromov \cite{Gromov}. 
The subsequence will converge to the limit in the Gromov-Hausdorff
sense, but the convergence away from the singular points is much stronger
in the following sense.
Define $\Omega_{\delta, j} \subset M_j$ to be the set of points  
with distance to the singular set bounded below by $\delta > 0$.
For $j$ large, these subsets all be diffeomorphic, 
and after pulling-back by diffeomorphisms 
to a fixed manifold, the metrics converge in any $C^{k,\alpha}$-norm in coordinate charts
as $j \rightarrow \infty$.  
To say more about the structure of the singularities, we need
an upper volume growth estimate, which we discuss next. 

\section{Volume growth estimate}
\label{vge}
  We emphasize that, in the Einstein case, 
an upper volume growth estimate on balls 
follows from Bishop's volume comparison theorem \cite{Bishop}. 
For metrics satisfying a system of the form 
\eqref{generaleqn}, it is much more difficult to obtain 
an upper volume growth estimate since we are not assuming 
any pointwise Ricci curvature bound. 
The following is the key result:
\begin{theorem}[Tian-Viaclovsky \cite{TV}]
\label{bigthm_i}
Let $(X,g)$ be a complete, noncompact, $n$-dimensional 
Riemannian manifold with base point $p$.
Assume that there exists a constant $C_1 > 0$ so that
\begin{align}
\label{cond4_i}
Vol(B(q,s)) \geq C_1 s^n,
\end{align}
for any $q \in X$, and all $s \geq 0$.
Assume furthermore that as $r \rightarrow \infty$,
\begin{align}
\label{decay1_i}
\underset{S(r)}{sup} \ |Rm_g| &= o(r^{-2}),
\end{align}
where $S(r)$ denotes the sphere of radius $r$ centered at $p$. 
If $b_1(X) < \infty$, then $(X,g)$ has finitely many ends, 
and there exists a constant $C_2$ (depending on $g$) so that 
\begin{align}
\label{vga_i}
Vol(B(p,r)) \leq C_2 r^n.
\end{align}
Furthermore, each end is ALE of order $0$.
\end{theorem}
\begin{proof}[Outline of proof.] The entire proof of
this theorem is over 20 pages long; we only give an 
extremely rough outline containing the main ideas. 
For $s > 1$, consider a sequence of dyadic annuli
$A(s^i, s^{i+1})$.
We can assume that there is a subsequence $\{j\} \subset \{i\}$
so that 
\begin{align}
\mathcal{H}^{n-1} ( S (s^{j+1})) \geq ( 1 - \eta_j) \mathcal{H}^{n-1} ( S( s^j)) s^{n-1}
\end{align}
for some sequence $\eta_j \rightarrow 0$ as $j \rightarrow \infty$. 
Otherwise, this would contradict the lower volume growth assumption. 
Letting $A_j =A(s^j, s^{j+1})$,  we show that as $j \rightarrow \infty$, 
\begin{align}
\label{l1form}
\frac{1}{ Vol(A_j)} \int_{A_j} | \Delta r^2 - 2n| dV_g \rightarrow 0.
\end{align}
The proof of this is a long computation and uses the coarea
formula, we will omit the details. 

If we rescale the annuli to unit size, that is, 
let 
\begin{align}
(\tilde{A}_j, \tilde{g}_j) = (A_j, s^{-2j} g),
\end{align}
the curvature decay estimate \eqref{decay1_i} implies that 
\begin{align}
|Rm(\tilde{g}_j)| \rightarrow 0
\end{align}
as $j \rightarrow \infty$, so the metric $\tilde{g}_j$ is
limiting to a flat metric. 

Note that if $Vol(\tilde{A}_j) < C$ for some constant $C$, then \eqref{l1form}
would imply that $\Delta \tilde{r}^2 \rightarrow 2n$ as $j \rightarrow \infty$, 
which implies that the rescaled distance function is limiting to 
the Euclidean distance function, so the sequence of rescaled annuli 
would converge to a Euclidean annulus. However, we do not 
yet know that the volumes of the rescaled annuli are bounded. 
To deal with this, we use a contradiction argument. 
If the rescaled volumes are not bounded, then we show it
is possible to divide the annuli into finitely many regions
with large but bounded volume, and prove that there is always 
at least one ``nice'' connected, non-collapsed region.  This chopping 
procedure is one of the most delicate parts of 
the proof.  We then apply the above
rescaling argument to the sequence of ``nice'' regions, and show that these
regions converge to portions of Euclidean annuli. 
Since Euclidean annuli ``close up'', it follows that the 
entire annular regions are in fact converging to Euclidean annuli. 
This contradiction proves that the rescaled annuli have bounded volume. 
Since this can be done for any $s > 1$, the upper volume estimate
\eqref{vga_i} follows. 
It then follows that all tangent cones at infinity 
are Euclidean cones, from which we conclude that the metric is 
ALE of order $0$. 

A important technical point arises with the above argument. 
In general, annuli might have many connected boundary components,
and a sequence of connected components of annuli
with the inner boundaries having more than $1$ connected component would 
cause a problem. However, the assumption on the first Betti number
ensures that this situation cannot arise.
\end{proof}

Note that Theorem \ref{bigthm_i} is a result for noncompact spaces, but
this result does in fact imply the volume growth result stated in Theorem \ref{orbthm2}. 
This is done by a contradiction argument, see \cite{TV2,TV3} for the details. 

\section{ALE order and removable singularity theorems}
The upper volume growth estimate implies that all tangent
cones of the limit space are ALE of order $0$ which implies that the 
limit space has $C^{0}$-orbifold singularities. 
That is, after passing to a local cover as in Definition \ref{orbdef}, the metric 
only has
an extension to a $C^{0}$-metric in a neighborhood
of the origin. 
Another important ingredient in the proof of Theorem \ref{tvt1}
is therefore to prove that the singularities of 
the limit are {\em{smooth}} orbifold points, that is, 
after passing to a local cover, the metric can be extended to a 
$C^{\infty}$-metric over the origin. 
A closely related problem is to obtain 
the optimal ALE order of the spaces which bubble out, which 
we will discuss next. 

A crucial result in the Ricci-flat case was obtained 
by Cheeger-Tian: if  $(M^n, g)$ is Ricci-flat  
ALE of order $0$, there exists a change of coordinates 
at infinity so that $(M^n, g)$ is ALE of order $n$, where 
$n$ is the dimension \cite{ct}.  This generalized and extended
the work of Bando-Kasue-Nakajima \cite{BKN}, who employed improved 
Kato inequalities together with a Moser iteration argument. 
The Cheeger-Tian method has the advantage of finding the 
{\em{optimal}} order of curvature decay, without 
relying on Kato inequalities. 

In the case of anti-self-dual scalar-flat metrics, 
or scalar-flat metrics with harmonic curvature, it was proved in 
\cite{TV} that such spaces are ALE of order $\tau$
for any $\tau < 2$, using the technique of Kato inequalities. 
Subsequently, this was generalized to Bach-flat metrics 
and metrics with harmonic curvature in dimension~$4$ 
in \cite{s}, using the Cheeger-Tian technique. 
This technique was generalized to obstruction-flat metrics
in any dimension in \cite{AV12}, which is a generalization 
of the Bach-flat condition in dimension four, see \cite{Graham}.  

The method in \cite{AV12} applies to much more general systems than just
the obstruction tensors, and works in any dimension $n \geq 3$.
Given two tensor fields $A, B$, the notation $A*B$ will mean a linear combination of 
contractions of $A\otimes B$ yielding a symmetric $2$-tensor. 
The main result is: 
\begin{theorem}[Ache-Viaclovsky \cite{AV12}]
\label{maint2} Let $k = 1$ if $n = 3$, or 
$1\le k \le\frac{n}{2}-1$ if $n \geq 4$.
Assume that $(M,g)$ is scalar-flat, 
ALE of order $0$, and satisfies
\begin{align}
\Delta^{k}_{g} Ric 
= \sum_{j=2}^{k+1}\sum_{\alpha_{1}+\ldots+\alpha_{j}=2(k+1)-2j}
\nabla_{g}^{\alpha_{1}}Rm*\ldots*\nabla_{g}^{\alpha_{j}}Rm.
\label{eqsrt11in}
\end{align}
Then $(M,g)$ is ALE of order $n-2k$. 
\end{theorem}
For $k =1$, this is simply 
\begin{align}
\label{kcc2}
\Delta Ric = Rm * Rm.
\end{align}
We emphasize that this is more general than \eqref{generaleqn}, since 
the right hand side is allowed to be quadratic in the full 
curvature tensor.
This is satisfied in particular  
by scalar-flat K\"ahler metrics and metrics with 
harmonic curvature in any dimension, and also 
anti-self-dual metrics in dimension~$4$. 
These special cases were previously considered in \cite{Chen} 
using improved Kato inequalities and a Moser iteration 
technique. 
The optimal decay for scalar-flat 
anti-self-dual ALE metrics was previously considered in \cite[Proposition 13]{CLW}.
The case of extremal K\"ahler ALE metrics was considered in \cite{CW}.
As mentioned above, the cases of Bach-flat metrics and 
metrics with harmonic curvature in dimension $4$ were considered in 
\cite{s}. 

 The main idea of the proof of Theorem \ref{maint2} is 
based on the method of Cheeger-Tian from \cite{ct}, and  
is roughly to show that the optimal ALE decay rate is determined by the rates 
of decaying solutions of the gauged, linearized equation on a Euclidean cone. 
This step uses a fundamental technique of Leon Simon called
the {\em{Three Annulus Lemma}} \cite{ls1}. 
An analysis of the decay rates of solutions of the gauged linearized 
equation, together with an estimate on the nonlinear terms in the equation, 
then yields Theorem \ref{maint2}. 

The same technique also yields a removable singularity theorem 
for higher-order systems:
\begin{theorem}[Ache-Viaclovsky \cite{AV12}]
\label{tr2}
Let $k = 1$ if $n = 3$, or 
$1\le k \le\frac{n}{2}-1$ if $n \geq 4$.
Assume that $(B_{\rho}(0) \setminus \{0\},g)$ has 
constant scalar curvature and satisfies
\begin{align}
\Delta^{k}_{g} Ric 
= \sum_{j=2}^{k+1}\sum_{\alpha_{1}+\ldots+\alpha_{j}=2(k+1)-2j}
\nabla_{g}^{\alpha_{1}}Rm*\ldots*\nabla_{g}^{\alpha_{j}}Rm.
\label{eqsrt11inorb}
\end{align}
If the origin is a $C^0$-orbifold point for $g$, 
then the metric extends to a smooth orbifold metric in $B_{\rho}(0)$.
\end{theorem}
In particular, this says that the limit space in Theorem \ref{tvt1} 
is a smooth multi-fold. 
\section{Chen-LeBrun-Weber metric}
\label{LCLW}
The volume growth theorem was used in a fundamental way in 
\cite{CLW} in order to obtain an Einstein metric on $\CP^2 \# 2 \overline{\CP}^2$,
in this section we will give a brief overview of the proof. 

\begin{theorem}[Chen-LeBrun-Weber \cite{CLW}] There exists a positive Einstein 
metric on $M = \CP^2 \# 2 \overline{\CP}^2$.
\end{theorem}
\begin{proof}[Outline of Proof.]
The first step is to consider K\"ahler classes which are invariant 
under a torus action, and bilaterally symmetric. That is, 
only K\"ahler classes from which the $(-1)$ curves in the blow-up
have the same area are considered. 
The space of such K\"ahler classes is $1$-dimensional,
and one can parametrize these classes by the area of the $(-1)$ curves;
call this parameter $x$. Chen-Lebrun-Weber then consider the functional 
\begin{align}
\mathcal{A}([\omega]) = \frac{ (c_1 \cdot [\omega])^2}{[\omega]^2}
- \frac{1}{32 \pi^2} \mathcal{F} ( \xi, [\omega]),
\end{align}
where $\mathcal{F} ( \xi, [\omega])$ is the Futaki invariant, 
with $\xi$ the extremal vector field of the class~$[\omega]$, see
\cite{Futaki1, Futaki2, FutakiMabuchi}. 
It is next observed that the graph of $\mathcal{A}$ as a 
function of $x$ has a strict local minimum at
a certain positive value of $x$, call this value $x_0$.
If one can prove that an extremal K\"ahler metric $g_{x_0}$
exists in this class corresponding to $x_0$, then 
$\mathcal{S}$ would have a critical point at 
$g_{x_0}$ when restricted to the set of toric, 
bilaterally symmetric K\"ahler classes on $M$.
From \eqref{w+r}, the functional $\mathcal{W}$ would also 
have such a critical point.

To proceed further, we need to understand the structure of the 
Bach tensor for a K\"ahler metric.  The Bach tensor is
a symmetric tensor, and since we have a complex structure,
we can consider the tensors $B^+$ and $B^-$, the $J$-invariant and 
$J$ anti-invariant parts of $B$, respectively. 
A computation shows that
\begin{align}
\label{bk1}
B^+ &= - 4 \left( R E + 2 ( \nabla^2 R)_0^+ \right),\\
\label{bk2}
B^- &= 4 ( \nabla^2 R)_0^-,
\end{align}
where $(\nabla^2 R)_0^+$ and $( \nabla^2 R)_0^-$ denote the 
$J$-invariant and $J$ anti-invariant parts of the 
traceless Hessian of the scalar curvature, respectively, see \cite{Derdzinski, ACG}.
This implies the following:
\begin{proposition}[Derdzinski \cite{Derdzinski}]
\label{Derdpro}
If $(M,g,J)$ is K\"ahler and Bach-flat, then  $(M,g,J)$ is 
extremal and the metric $\tilde{g} = R_g^{-2} g$ is 
Einstein near any point with $R_g \neq 0$. 
\end{proposition}
\begin{proof} 
This follows from equations \eqref{bk1} and \eqref{bk2} 
by noting that the condition 
$( \nabla^2 R)_0^- = 0$
is exactly the condition for $g$ to be extremal K\"ahler \cite{CalabiI}, 
and the second claim follows from the conformal transformation 
formula \eqref{eintrans}.
\end{proof}
If the extremal metric $g_{x_0}$ exists, since this metric
is critical for $\mathcal{W}$ restricted to the space of 
invariant K\"ahler classes, \eqref{bk1} and \eqref{bk2} show that 
the Bach tensor can be viewed is a harmonic $(1,1)$-form and
can therefore be used as a K\"ahler variation. Consequently, 
$g_{x_0}$ would be Bach-flat. It turns out that this metric 
must have strictly positive scalar curvature, so  
Proposition \ref{Derdpro} yields the desired conformally Einstein metric.

The main part of the proof is therefore to show that the
extremal K\"ahler metric $g_{x_0}$ exists. To show this, 
a continuity argument is used. For $x$ small, extremal
K\"ahler metrics are known to exist in these
K\"ahler classes by work of Arezzo-Pacard-Singer \cite{ArezzoPS};
this was an extension of the gluing result of Arezzo-Pacard for 
constant scalar curvature K\"ahler metrics \cite{API, APII}. 
The set of K\"ahler classes
admitting extremal K\"ahler metrics is known to be open 
by LeBrun-Simanca \cite{LeBrunSimanca}. 
A compactness theorem is used to show that the set of 
$x$ for which there exists an extremal K\"ahler metric 
is also closed for $x \leq x_0$. Consequently, from 
connectedness of the interval $(0,x_0]$, an extremal 
K\"ahler metric exists at $x_0$.  

  To show the compactness, an extension of Theorem \ref{tvt1} 
to extremal K\"ahler metrics is used \cite{CW}. 
We note that the main part of \cite{CW} is to extend the $\epsilon$-regularity 
result in Theorem \ref{higherlocalregthm} 
to the class of extremal K\"ahler metrics; 
the volume growth result in Theorem \ref{bigthm_i} is still crucial in order to 
obtain the compactness theorem. Given a sequence 
of extremal K\"ahler metrics $g_{x_i}$ for $x_i \rightarrow x \leq x_0$
as $i \rightarrow \infty$, an orbifold limit can be obtained 
provided that the Sobolev constant can be 
controlled, which is proved in \cite[Section 5]{CLW}.
If the curvatures of this sequence were not bounded, then a nontrivial K\"ahler 
scalar-flat ALE space must bubble off at some point. 
However, a detailed analysis of possible bubbles,
employing the toric and bilateral symmetries, 
shows that non-trivial bubbles can be ruled out, thereby proving compactness. 
\end{proof}
\lecture{A gluing theorem for $B^t$-flat metrics}
\label{L10}
\section{Existence of critical metrics}
Let us begin by stating the main result:
\begin{theorem}[Gursky-Viaclovsky \cite{GVCritical}]
\label{gvthm}
A $B^t$-flat metric exists on the manifolds in the table for some
$t$ near the indicated value of $t_0$. 
\end{theorem}

\begin{table}[ht]
\caption{Simply-connected examples with one bubble}
\centering
\begin{tabular}{ll}
\hline\hline
Topology of connected sum& Value(s) of $t_0$ \\
\hline
$\CP^2 \# \overline{\CP}^2$ & $-1/3$ \\
$ S^2 \times S^2 \# \overline{\CP}^2 = \CP^2 \# 2  \overline{\CP}^2$ & $-1/3$, $- (9 m_1)^{-1}$\\
$ S^2 \times S^2 \# S^2 \times S^2$ & $-2(9 m_1)^{-1}$ \\
\hline
\end{tabular}\label{table}
\end{table}
The constant $m_1$ is the mass of the Green's function metric of 
the product metric $S^2 \times S^2$, defined in \eqref{massdef}.


We make some remarks: 
\begin{itemize}
\item
$M = \CP^2 \# \overline{\CP}^2$ admits an $U(2)$-invariant Einstein metric
called the ``Page metric'' \cite{Page}.  
$M$ does not admit any K\"ahler-Einstein metric, but the Page
metric is conformal to an extremal K\"ahler metric.

\vspace{2mm}
\item
$M = \CP^2 \# 2  \overline{\CP}^2$ admits a toric invariant Einstein metric 
called ``Chen-LeBrun-Weber metric'' described in the previous lecture \cite{CLW}. 
Again, $M$ does not admit any 
K\"ahler-Einstein metric, but the 
Chen-LeBrun-Weber metric is conformal to an extremal K\"ahler metric.

\vspace{2mm}
\item  $ M = S^2 \times S^2 \# S^2 \times S^2$ does not admit any K\"ahler 
metric, it does not even admit an almost complex structure. Our metric 
is the first known example of a ``canonical'' metric on this manifold.
\end{itemize}


\subsection{The approximate metric}
The critical metrics in Theorem \ref{gvthm} 
are found by perturbing from an 
``approximate'' critical metric. We describe this 
construction next. 
\begin{itemize}
\item Let $(Z, g_Z)$ and $(Y, g_Y)$ be Einstein manifolds, and assume that
$g_Y$ has positive scalar curvature. 

\vspace{2mm}
\item Choose basepoints $z_0 \in Z$ and
$y_0 \in Y$. 

\vspace{2mm}
\item Convert $(Y, g_Y)$
into an asymptotically flat (AF) metric 
\begin{align*}
(N, g_N) = (Y \setminus \{y_0\}, G^2 g_Y)
\end{align*}
using the Green's function for
the conformal Laplacian based at $y_0$.
Since $(M,g)$ is Bach-flat, then $(N,g_N)$ is also Bach-flat (from
conformal invariance) and scalar-flat (since we used the Green's function).
Consequently, $g_N$ is $B^t$-flat for all $t \in \RR$.

\vspace{2mm}
\item Let $a > 0$ be small, and consider $Z \setminus B(z_0,a)$. Scale the
compact metric to $(Z, \tilde{g} = a^{-4} g_Z)$.
Attach this metric to the metric $ ( N \setminus B(a^{-1}), g_N)$  using
cutoff functions near the boundary, to obtain a smooth metric
on the connected sum $Z \# \overline{Y}$.  
\end{itemize}


\begin{figure}[h]
\includegraphics[scale=.84]{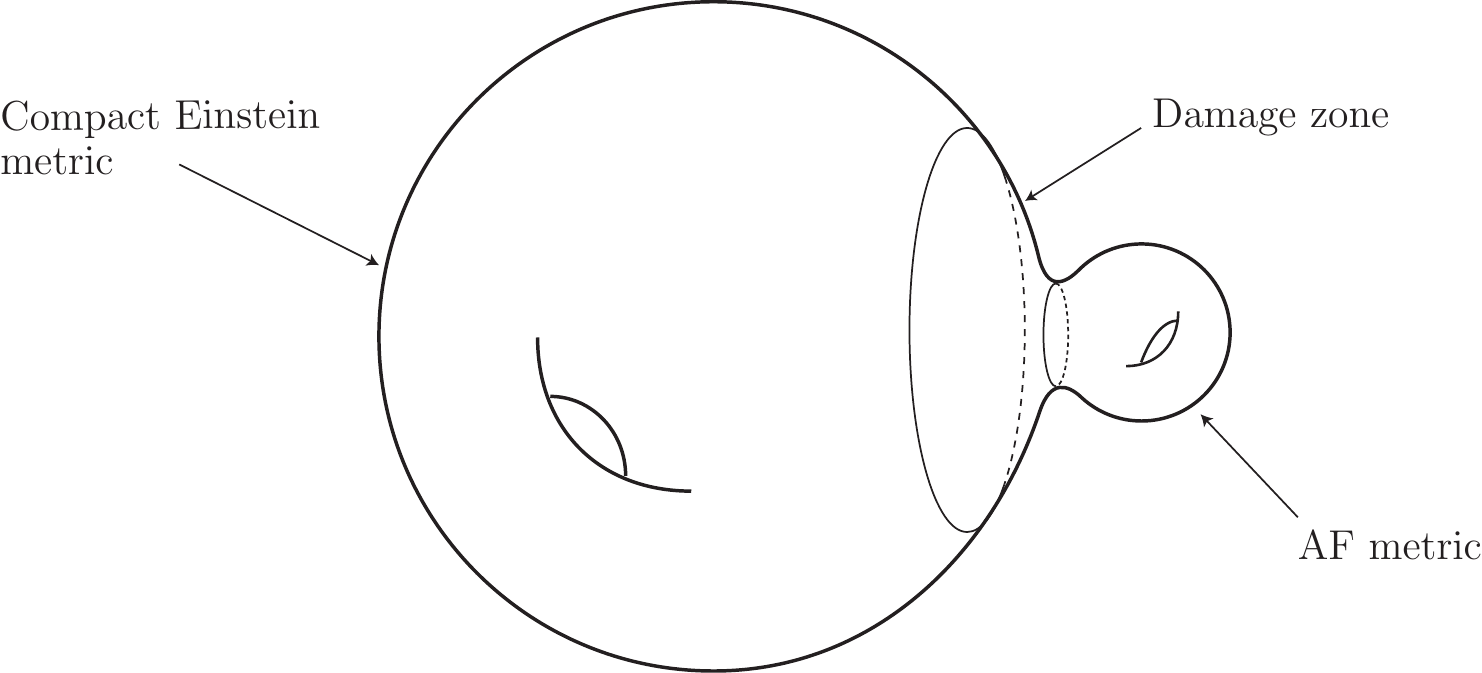}
\caption{The approximate metric.}
\label{bubblefig}
\end{figure}
Since both $g_Z$ and $g_N$ are
$B^t$-flat, this metric is an ``approximate'' $B^t$-flat metric,
with vanishing $B^t$ tensor away from the ``damage zone'', where cutoff
functions were used. 


\section{Lyapunov-Schmidt reduction}
In general, there are several degrees of freedom in this approximate metric.

\begin{itemize}
\vspace{2mm}
\item The scaling parameter $a$ 
($1$-dimensional).

\vspace{2mm}
\item Rotational freedom when attaching 
($6$-dimensional).

\vspace{2mm}
\item Freedom to move the base points of either factor
($8$-dimensional). 
\end{itemize}

\vspace{2mm}
There are a total of 15 gluing parameters, 
which yield a 15-dimensional space of 
``approximate'' kernel of the linearized operator. 
Using a Lyapunov-Schmidt reduction
argument, one can reduce the problem to that of finding a zero of the {\em{Kuranishi map}}
\begin{align}
\Psi : U \subset \RR^{15} \rightarrow \RR^{15}. 
\end{align}
\begin{itemize}
\item 
It is crucial to use certain weighted norms to find a bounded
right inverse for the linearized operator. 

\vspace{2mm}
\item
This 15-dimensional problem is too difficult in general. 
We will take advantage of various symmetries in order to reduce to 
only $1$ free parameter: the scaling parameter $a$.
\end{itemize}


\subsection{Technical theorem}
The leading term of the Kuranishi map corresponding to the scaling
parameter is given by:
\begin{theorem}[\cite{GVCritical}]
As $a \rightarrow 0$, then for any $\epsilon > 0$,  
\begin{align}
\Psi_1  &= \Big( \frac{2}{3} W(z_0) \circledast W(y_0)
+ 4t R(z_0) \mathrm{mass}(g_N)  \Big) \omega_3 a^4 + O (a^{6 - \epsilon}),
\end{align}
where $\omega_3 = Vol(S^3)$, and the product of the Weyl tensors is given by
\begin{align}
W(z_0) \circledast W(y_0) = \sum_{ijkl} W_{ijkl}(z_0) ( W_{ijkl}(y_0) + W_{ilkj}(y_0)),
\end{align}
where $W_{ijkl}(\cdot)$ denotes the components of the
Weyl tensor in a normal coordinate system at the corresponding point.
\end{theorem}

We note that the product $\circledast$  depends upon the coordinate
systems chosen,
and therefore in general depends upon a rotation parameter, and
obviously on the
base points of the gluing.

\section{The building blocks}

We next discuss the ``building blocks'' of the gluing construction. 

\subsection{The Fubini-Study metric}
We let $(\CP^2, g_{\rm{FS}})$ denote the Fubini-Study metric, scaled so that 
$Ric = 6 g$. We consider the following group actions.

\vspace{2mm}
Torus action:
\begin{align}
[z_0, z_1, z_2] \mapsto [z_0,  e^{i \theta_1} z_1, e^{i \theta_2} z_2].
\end{align}

\vspace{2mm}
Flip symmetry:  
\begin{align}
[z_0, z_1, z_2] \mapsto  [z_0, z_2, z_1]. 
\end{align}
The Green's function metric of the Fubini-Study metric $\hat{g}_{FS}$
is also known as the Burns metric, and is completely explicit
(in fact, it is the case $n=1$ of \eqref{glbdef}) with mass given by
\begin{align}
\mathrm{mass}(\hat{g}_{FS}) = 2.
\end{align}

\subsection{The product metric}
The next building block is $(S^2 \times S^2, g_{S^2 \times S^2})$,  the product of $2$-dimensional spheres of Gaussian curvature $1$,  with $Ric = g$. We 
consider the following group actions.

\vspace{2mm}
Torus action:
\begin{align}
\text{Product of rotations fixing north and south poles}.
\end{align}

\vspace{2mm}
Flip symmetry:  
\begin{align}
(p_1,p_2) \mapsto (p_2, p_1). 
\end{align}

The Green's function metric $\hat{g}_{S^2 \times S^2}$ of
the product metric does not seem to have a known explicit description.
We will denote
\begin{align}
m_1 = \mathrm{mass}(\hat{g}_{S^2 \times S^2}).
\end{align}
By the positive mass theorem of Schoen-Yau, $m_1 > 0$.
Note that since $S^2 \times S^2$ is spin, this also follows from Witten's
proof of the positive mass theorem. The value of $m_1$ has
recently been determined: 
\begin{theorem}[\cite{ViaclovskyMass}]
The mass $m_1 \sim .5872$ and may be written as an 
explicit infinite sum.
\end{theorem}
The explicit formula for the mass is lengthy, and 
will not be written here.  We just note that 
this implies that $(-9 m_1)^{-1} \sim = -.1892$ so as a corollary
we see that  the manifold $S^2 \times S^2 \# \overline{\CP}^2$ admits
a $B^t$-flat metrics for at least two different values of $t$. 
Note that it is shown in \cite{GVCritical} that 
the metrics obtained in Theorem \ref{gvthm} are not Einstein, so these
metrics are distinct from the Chen-LeBrun-Weber metric.

\section{Remarks on the proof}
We first remark that:
\begin{itemize}
\item We impose the toric symmetry and ``flip'' symmetry in order
to reduce the number of free parameters to $1$ (only the scaling parameter). 
That is, we perform an equivariant gluing.

\vspace{2mm}
\item
The special value of $t_0$ is computed by
\begin{align*}
 \frac{2}{3} W(z_0) \circledast W(y_0)
+ 4t_0 R(z_0) \mathrm{mass}(g_N) = 0.
\end{align*}
This choice of $t_0$ makes the leading term of the Kuranishi map vanish, and is 
furthermore a {\em{nondegenerate}} zero (if $R(z_0) > 0$; mass$(g_N) > 0$ by
the positive mass theorem).
\end{itemize}

We next outline the spaces used in the construction 
of the na\"ive approximate metric:
\begin{itemize}
\item (i)
$\CP^2 \# \overline{\CP}^2$; the Fubini-Study metric with a Burns metric attached
at one fixed point. This case admits a $U(2)$-action.

\vspace{2mm}
\item(ii)
$S^2 \times S^2 \# \overline{\CP}^2 = \CP^2 \# 2  \overline{\CP}^2$;
the product metric on $S^2 \times S^2$ with a Burns metric attached at
one fixed point. Alternatively, we can view this as the Fubini-Study
metric on $\CP^2$, with a Green's function $S^2 \times S^2$ metric attached at
one fixed point. For this topology, we will therefore construct two
different critical metrics. Both of these will have toric symmetry 
plus invariance under the flip symmetry. 

\vspace{2mm}
\item(iii)
$S^2 \times S^2 \# S^2 \times S^2$; the product metric on $S^2 \times S^2$ with
a Green's function $S^2 \times S^2$ metric attached at one fixed point.
This metric is toric and flip-symmetric. 
\end{itemize}

 We note that an equivariant gluing is carried out -- the 
metrics obtained in Theorem~\ref{gvthm} retain the 
indicated symmetries. 

 By imposing other discrete symmetries, we can perform the gluing operation 
with more than one bubble. For example, we can find critical metrics on
$ \CP^2 \# 3  \overline{\CP}^2$
$3 \# S^2 \times S^2$, $\CP^2 \# 3(S^2 \times S^2)$,  
$S^2 \times S^2 \# 4 \overline{\CP}^2$, and $5 \# S^2 \times S^2$
(see \cite[Table~1.2]{GVCritical}).

The product metric on $S^2 \times S^2$ admits the 
quotient $S^2 \times S^2/\ZZ_2$
where $\ZZ_2$ acts by the antipodal map on both factors.
It is  well-known that this quotient is
diffeomorphic to $G(2,4)$, the Grassmannian of $2$-planes
in $\RR^4$, see for example \cite{SingerThorpe}. Another 
quotient is $\RP^2 \times \RP^2$.
The product metric descends to an Einstein metric
on both of these quotients.
We can also use these quotient spaces as building blocks 
to obtain non-simply connected examples. We do not list all of the 
examples here, but just note that we find a critical metric on 
$G(2,4) \# G(2,4)$, which has infinite fundamental group, 
and therefore does not admit any positive Einstein metric
by Myers' Theorem. For the complete list of non-simply-connected examples, 
see \cite[Table~1.3, Table~B.1]{GVCritical}.


\subsection{Ellipticity and gauging}
The $B^t$-flat equations are not elliptic 
due to diffeomorphism invariance. A gauging procedure analogous to 
the Bianchi gauge is used. This was already discussed above
in Lecture \ref{L7}, with the following note. 
From \eqref{CGB}, we can write
\begin{align}
\mathcal{F}_{\tau} = 16 \pi^2 \chi(M) + \frac{1}{2} \mathcal{B}_{2(\tau + \frac{1}{3})}.
\end{align}
Taking gradients, we obtain the relation
\begin{align}
\label{gradrel}
\nabla \mathcal{B}_{t} = 2 \nabla \mathcal{F}_{\frac{t}{2} - \frac{1}{3}}.
\end{align}
It follows from the formula for $P$ that the linearized operator is given by
\begin{align} \label{linop}
S^t h = (B' + t C')h
+ \mathcal{K}_{g} \delta_g \mathcal{K}_g \delta_g \overset{\circ}{h} ,
\end{align}
where $B'$ and $C'$ are the linearizations of $B$ and $C$ respectively.
Therefore the discussion in Lecture \ref{L7} applies, with 
$\tau$ replaced by $\frac{t}{2}- \frac{1}{3}$.

\subsection{Rigidity}
The rigidity results we need were discussed above. We mention here 
the resulting linearized operator as we change $\tau \rightarrow t$. 
For $h$ transverse-traceless (TT), the 
linearized operator at an Einstein metric is given by
\begin{align}
S^t h = \Big( \Delta_L + \frac{1}{2} R\Big)
\Big(\Delta_L + \Big( \frac{1}{3} +  t \Big) R \Big) h.
\end{align}


Next, for $h = f g$, we have 
\begin{align}
tr_g ( S^t h) = 6t( 3 \Delta + R) (\Delta f).
\end{align}
The above rigidity results, Theorems \ref{cp2rig} and \ref{s2s2rig}, 
are then as follows (stated in terms of $t$ instead of $\tau$).
\begin{theorem}[\cite{GVRS}]
\label{cp2rigt}
On $(\CP^2, g_{\rm{FS}})$, $H^1_t = \{0\}$ provided that $t < 1$. 
\end{theorem}
\noindent
For the case of the product metric:
\begin{theorem}[\cite{GVRS}] 
\label{s2s2rigt}
On $(S^2 \times S^2, g_{S^2 \times S^2})$,
$H^1_t = \{0\}$ provided that $t < 2/3$ and $t \neq - 1/3$.
If $t = - 1/3$, then $H^1_t$ is one-dimensional and spanned by 
the element $g_1 - g_2$. 
\end{theorem}


In relation to Theorem \ref{gvthm}, the positive mass theorem says that $t_0 <0$, so luckily we are in the rigidity range of the factors. Consequently, there is 
no cokernel arising from deformations of the building blocks. 


\subsection{Refined approximate metric} 
The approximate metric described above
is not good enough. It must be improved by  matching up leading terms
of the metrics by solving certain auxiliary linear equations, 
so that the cutoff function disappears from the leading term. 
This step is inspired by the work of Biquard mentioned above. 
Let $(Z,g_Z)$ be the compact metric. In Riemannian normal coordinates, 
\begin{align}
(g_Z)_{ij}(z) = \delta_{ij} - \frac{1}{3} R_{ikjl}(z_0) z^k z^l + O^{(4)}(|z|^4)_{ij}
\end{align}
as $z \rightarrow z_0$.

Let $(N, g_N)$ be the Green's function metric of $(Y, g_Y)$, 
then we have 
\begin{align}
(g_N)_{ij}(x) &= \delta_{ij} - \frac{1}{3} R_{ikjl}(y_0) \frac{x^k x^l}{|x|^4} + 2 A
\frac{1}{|x|^2} \delta_{ij} + O^{(4)}( |x|^{-4 + \epsilon})
\end{align}
as $|x| \rightarrow \infty$, for any $\epsilon > 0$. 
Note that the constant $A$ is given by 
\begin{align}
\mathrm{mass}(g_N) = 12 A - R(y_0)/12. 
\end{align}
We consider $a^{-4} g_Z$ and let $z = a^2 x$, then we have 
\begin{align}
a^{-4} (g_Z)_{ij}(x) = \delta_{ij} - a^4 \frac{1}{3} R_{ikjl}(z_0) x^k x^l + \cdots.  
\end{align}
Note that the second order terms do not agree.  One needs to construct new metrics on the 
factors so that these terms agree. This is done by solving
the linearized equation on each factor with prescribed leading term  
the second order term of the {\em{other}} metric. We will describe 
this procedure next. 


\subsection{The obstruction}
On $(N, g_N)$, one would like to solve 
\begin{align}
\begin{split}
S^t \tilde{h} &= 0\\
\tilde{h} &= - a^4 \frac{1}{3} R_{ikjl}(z_0) x^k x^l  + O(|x|^{\epsilon}), 
\end{split}
\end{align}
as  $x \rightarrow \infty$.
However, it turns out this equation is obstructed, 
so there is not necessarily a solution. 
However, using some Fredholm theory in weighted spaces, 
it turns out that one {\em{can}} solve the modified equation
\begin{align*}
S^t \tilde{h} = \lambda \cdot k_1,
\end{align*}
where $k_1$ pairs nontrivially with the decaying cokernel $o_1$ 
on the AF space $(N,g_N)$.  (There is also a considerable 
amount of work involved in order to prove that the space of decaying cokernel is 
$1$-dimensional; the symmetries are crucial for this.)
\begin{itemize}
\vspace{2mm} 
\item 
A similar procedure is carried out on the compact factor, except this 
is unobstructed (since the compact factor is rigid), so this 
does not contribute to the leading term of the Kuranishi map. 
\end{itemize}

The leading term is the computed by the following. Pairing with the 
cokernel element $o_1$, 
\begin{align*}
\lambda &= \lim_{r \rightarrow \infty} 
\int_{B(r)} \langle S \tilde{h}, o_1 \rangle \\
&= \lim_{r \rightarrow \infty} (\mathrm{\ spherical \ boundary \ integrals}).
\end{align*}
This limit can be computed explicitly using the expansion of the cokernel element
\begin{align}
(o_1)_{ij} = \frac{2}{3} W_{ikjl}(y_0) \frac{ x^k x^l}{|x|^4}  + f g_{ij} + O( |x|^{-4 + \epsilon})
\end{align}
as $x \rightarrow \infty$,  where $f$ satisfies 
\begin{align}
\Delta f = - \frac{1}{3} \langle Ric, o_1 \rangle,
\end{align}
together with the expansion
\begin{align}
\tilde{h} &= - a^4 \frac{1}{3} R_{ikjl}(z_0) x^k x^l  + O(|x|^{\epsilon}), 
\end{align}
as  $x \rightarrow \infty$.

The complete computation is very lengthy, and we refer to \cite{GVCritical} for 
the details. 


\subsection{Final remarks}
The proof shows that there is a dichotomy. 
Either
\begin{itemize}
\item(i) there is a critical metric at exactly the
critical $t_0$, in which case there would necessarily be a 1-dimensional
moduli space of solutions for this fixed $t_0$, or
\item (ii) for each value of the gluing parameter
$a$ sufficiently small, there will be a critical metric for a
corresponding value of $t_0 = t_0(a)$. 
The dependence of $t_0$ on $a$ will depend on the next term in the expansion of
the Kuranishi map. 
\end{itemize}


\bibliography{Park_City}
\end{document}